\documentclass[11pt]{amsart}
\usepackage[left=1.5in, right=1.5in]{geometry}
\usepackage[latin1]{inputenc}

\usepackage{amsfonts}
\usepackage[toc,page,title,titletoc,header]{appendix}
\usepackage{graphicx,psfrag,epsfig,multirow,setspace}
\usepackage{amssymb,amsmath,amscd,amsthm,verbatim}
\usepackage{mathrsfs}
\usepackage{caption,subcaption,hyperref}

\usepackage{color}

\numberwithin{equation}{section}

\newtheorem{Thm}{Theorem}
\newtheorem{Lm}{Lemma}[section]
\newtheorem{sublemma}[Lm]{Sublemma}
\newtheorem{Prop}[Lm]{Proposition}
\newtheorem{Cor}[Lm]{Corollary}
\newtheorem{Def}[Lm]{Definition}
\newtheorem{Rk}[Lm]{Remark}

\newtheorem{Conjecture}{Conjecture}

\def\bdef{\begin{Def}}
\def\endef{\end{Def}}
\def\bthm{\begin{Thm}}
\def\ethm{\end{Thm}}
\def\bprop{\begin{Prop}}
\def\enprop{\end{Prop}}
\def\blm{\begin{Lm}}
\def\elm{\end{Lm}}
\def\bcor{\begin{Cor}}
\def\ecor{\end{Cor}}
\def\bfig{\begin{picture}}
\def\efig{\end{picture}}
\def\be{\begin{eqnarray}}
\def\ee{\end{eqnarray}}
\def\beal{\begin{aligned}}
\def\enal{\end{aligned}}
\newcommand{\bcon}{\begin{Conjecture}}
\newcommand{\econ}{\end{Conjecture}}

\newcommand{\sign}{{{\rm sign}}}

\newcommand{\gm}{\gamma}

\newcommand{\al}{\alpha}

\newcommand{\cR}{\mathcal{R}}

\newcommand{\cF}{\mathcal{F}}
\newcommand{\cD}{\mathcal{D}}
\newcommand{\cL}{\mathcal{L}}
\newcommand{\cS}{\mathcal{S}}

\newcommand{\cQ}{\mathcal{Q}}

\newcommand{\fh}{\mathfrak{h}}

\newcommand{\R}{\mathbb{R}}

\newcommand{\dt}{\delta}
\newcommand{\lb}{\lambda}

\newcommand{\eps}{\varepsilon}
\newcommand{\T}{\mathbb{T}}

\newcommand{\Ger}{\mathbf{G}}

\newcommand{\Loc}{\mathbb{L}}
\newcommand{\Glob}{\mathbb{G}}
\newcommand{\Id}{\mathrm{Id}}
\newcommand{\Span}{\text{span}}
\newcommand{\Ker}{\text{Ker}}
\newcommand{\lin}{{\bf l}}
\newcommand{\brlin}{{\bar\lin}}
\newcommand{\brrlin}{{\bar\brlin}}

\newcommand\ba{\boldsymbol{a}}
\newcommand\bb{\boldsymbol{b}}

\newcommand\bL{\boldsymbol{L}}

\newcommand\bx{\boldsymbol{x}}
\newcommand\by{\boldsymbol{y}}
\newcommand\bomega{{\boldsymbol{\omega}}}

\newcommand\bbV{{\mathbb{V}}}
\newcommand\bbW{{\mathbb{W}}}

\newcommand{\brc}{\bar c}
\newcommand{\brE}{\bar E}

\newcommand{\bre}{\bar e}
\newcommand{\brre}{\bar \bre}
\newcommand{\brG}{\bar G}
\newcommand{\brg}{\bar g}
\newcommand{\brrg}{\bar \brg}
\newcommand{\brK}{\bar K}

\newcommand{\brt}{\bar t}
\newcommand{\brv}{\bar{\mathbf u}}
\newcommand{\brrv}{\bar \brv}

\newcommand\brdelta{{\bar\delta}}
\newcommand\brtheta{{\bar\theta}}

\newcommand\brchi{{\bar\chi}}
\newcommand\brtau{{\bar\tau}}

\newcommand{\cK}{\mathcal{K}}
\newcommand{\cP}{\mathcal{P}}
\newcommand{\cU}{\mathcal{U}}
\newcommand{\cV}{\mathcal{V}}

\newcommand\he{\hat e}
\newcommand\hg{\hat g}

\def\hC{{\hat{C}}}
\def\hG{{\hat{G}}}
\def\hL{{\hat{L}}}

\def\hw{{\hat{w}}}
\def\hlin{{\hat{\lin}}}
\def\hGamma{{\hat{\Gamma}}}

\def\FF{{\mathbb{F}}}
\def\II{{\mathbb{I}}}

\def\ta{\tilde{a}}
\def\tb{\tilde{b}}
\def\tC{\tilde{C}}
\def\te{\tilde{e}}
\def\tG{\tilde{G}}
\def\tg{\tilde{g}}

\def\tL{\tilde{L}}
\def\tp{\tilde{p}}

\def\tV{\tilde{V}}
\def\tbbV{\tilde{\bbV}}
\def\tw{\tilde{w}}

\def\tZ{\tilde{Z}}
\def\tGamma{\tilde{\Gamma}}
\def\tdelta{\tilde{\delta}}

\title[Noncollision Singularities in the 2-Center-2-Body problem]{Non-Collision singularities in the Planar Two-Center-Two-Body problem}
\author{Jinxin Xue and Dmitry Dolgopyat}
\address{University of Chicago, Chicago, IL, 60637}
\email{jxue@math.uchicago.edu}
\address{University of Maryland, College Park, MD, 20740}
\email{email: dmitry@math.umd.edu}

\date\today
\begin{document}
\maketitle
\begin{abstract}
  In this paper, we study a restricted four-body problem called planar two-center-two-body problem. In the plane, we have two fixed centers $Q_1$ and $Q_2$ of masses 1, and two moving bodies $Q_3$ and $Q_4$ of masses $\mu\ll 1$. They interact via Newtonian potential. $Q_3$ is captured by $Q_2$, and $Q_4$ travels back and forth between two centers. Based on a model of Gerver, we prove that there is a Cantor set of initial conditions which lead to solutions of the Hamiltonian system whose velocities are accelerated to infinity within finite time avoiding all earlier
  collisions. 
This problem is a simplified model for the planar four-body problem case of the Painlev\'{e} conjecture.
\end{abstract}
\begin{spacing}{0.45}
\tableofcontents
\end{spacing}
\renewcommand\contentsname{Index}

\section{Introduction}
\subsection{Statement of the main result}

We study a two-center two-body problem. Consider two fixed centers $Q_1$ and $Q_2$ of masses $m_1=m_2=1$
located at distance $\chi$ from each other and
two small particles $Q_3$ and $Q_4$ of masses $m_3=m_4=\mu\ll 1$.
$Q_i$s interact with each other via Newtonian potential.
If we choose coordinates so that $Q_2$ is at $(0,0)$ and $Q_1$ is at $(-\chi, 0)$
then the Hamiltonian of this system can be written as
\begin{equation}H=\dfrac{|P_3|^2}{2\mu}+\dfrac{|P_4|^2}{2\mu}-\dfrac{\mu}{|Q_3|}-\dfrac{\mu}{|Q_3-(-\chi,0)|}-\dfrac{\mu}{|Q_4|}-\dfrac{\mu}{|Q_4-(-\chi,0)|}-\dfrac{\mu^2}{|Q_3-Q_4|}.\label{eq: main}\end{equation}

We assume that the total energy of the system is zero.

We want to study singular solutions of this system, that is, the solutions which can not be continued for all positive times.
We will exhibit a rich variety of singular solutions. Fix $\eps_0<\chi.$ Let $\bomega=\{\omega_j\}_{j=1}^\infty$
be a sequence of 3s and 4s.

\begin{Def}
We say that $(Q_3(t), Q_4(t))$ is a {\bf singular solution with symbolic sequence $\bomega$} if there exists a positive
increasing sequence $\{t_j\}_{j=0}^\infty$ such that
\begin{itemize}
\item $t^*=\lim_{j\to\infty} t_j<\infty.$
\item $|Q_3(t_j)-Q_2|\leq \eps_0,$ $|Q_4(t_j)-Q_2|\leq \eps_0.$
\item For $t\in [t_{j-1}, t_j]$, $|Q_{7-\omega_j}(t)-Q_2|\leq \eps_0$ and
$\{Q_{\omega_j}(t)\}_{t\in [t_{j-1}, t_j]}$ leaves the $\eps_0$ neighborhood of $Q_2$, winds around $Q_1$ exactly once then reenters the $\eps_0$ neighborhood of $Q_2$.
\item $\displaystyle \limsup_{t\uparrow t^*} |\dot{Q}_i(t)|\to \infty$ for $i=3,4.$
\end{itemize}
\end{Def}
During the time interval $[t_{j-1}, t_j]$ we refer to $Q_{\omega_j}$ as the traveling particle and
to $Q_{7-\omega_j}$ as the captured particle. Thus $\omega_j$ prescribes which particle is
the traveler during the $j$ trip. The phrase that the traveler winds around $Q_1$ exactly once means that
the angle from $Q_1$ to the traveler changes by $2\pi+O(1/\chi).$

We denote by ${\Sigma_\bomega}$ the set of initial conditions of singular orbits with symbolic sequence $\bomega.$ Note that if
$\bomega$ contains only finitely many 3s then there is a collision of $Q_3$ and $Q_2$ at time $t^*.$
If $\bomega$ contains only finitely many 4s then there is a collision of $Q_4$ and $Q_2$ at time $t^*.$
Otherwise at we have a {\it collisionless singularity} at $t^*.$

\bthm
\label{ThMain}
There exists $\mu_*\ll 1$ such that for $\mu<\mu_*$ the set $\Sigma_\bomega\neq\emptyset.$

Moreover there is an open set $U$ on the zero energy level and a foliation of $U$ by two-dimensional surfaces such that for any leaf
$S$ of our foliation $\Sigma_\bomega\cap S$ is a Cantor set.
\ethm
\begin{Rk}
By rescaling space and time variables we can assume that $\chi\gg 1.$ In the proof we shall make this assumption and
set $\eps_0=2.$
\end{Rk}
\begin{Rk}
It follows from the proof that the Cantor set described in Theorem \ref{ThMain} can be chosen to depend continuously
on $S.$ In other words $\Sigma_\bomega$ contains a set which is locally a product of a five dimensional disc and
a Cantor set. The fact that on each surface we have a Cantor set follows from the fact that we have a freedom of
choosing how many rotations the captured particle makes during $j$-th trip.
\end{Rk}
\begin{Rk}
The construction presented in this paper also works for small nonzero energies. Namely, it is sufficient that the
total energy is much smaller than the kinetic energies of the individual particles. The assumption that the total energy
is zero is made to simplify notation since then the energies of $Q_3$ and $Q_4$ have the same absolute values.
\end{Rk}
\begin{Rk}
One can ask if Theorem \ref{ThMain} holds for other choices of masses. The fact that the masses of
the fixed centers $Q_1$ and $Q_2$ are the same is not essential and is made only for convenience.
The assumption that $Q_3$ and $Q_4$ are light is important since it allows us to treat their interaction as a
perturbation except during the close encounters of $Q_3$ and $Q_4.$
The fact that the masses of $Q_3$ and $Q_4$ are equal allows us to use an explicit periodic solution
of a certain limiting map $(${\em Gerver map}$)$ which is found in \cite{G2}. It seems likely that the conclusion
of Theorem \ref{ThMain} is valid if $m_3=\mu, m_4=c\mu$ where $c$ is a fixed constant close to 1 and $\mu$ is
sufficiently small but we do not have a proof of that.
\end{Rk}

\subsection{Motivations.}
\subsubsection{Non-collision singularities in N-body problem}

Our work is motivated by the following fundamental  problem in celestial mechanics.
{\it Describe the set of initial conditions  of the Newtonian N-body problem leading to global solutions.}
The compliment to this set splits into the initial conditions leading to the collision and non-collision singularities.

It is clear that the set of initial conditions leading to collisions is non-empty for all $N>1$ and it is shown in
\cite{Sa1} that it has zero measure. Much less is known about the non-collision singularities. The main motivation
for our work is provided by
following basic problems.

\bcon
\label{ConNE}
The set of non-collision singularities is non-empty for all $N>3.$
\econ

\bcon
\label{ConZM}
The set of non-collision singularities has zero measure for all $N>3.$
\econ

Conjecture \ref{ConNE} probably goes back to Poincar\'{e} who was motivated by King Oscar II prize problem
about analytic representation of collisionless solutions of the $N$-body problem. It was explicitly mentioned
in Painlev\'{e}'s lectures \cite{Pa} where the author proved that for $N=3$ there are no non-collision singularities.
Soon after Painlev\'{e}, von Zeipel showed that if the system of $N$ bodies has a non-collision singularity then some
particle should fly off to infinity in finite time. Thus non-collision singularities seem quite counterintuitive. However
in \cite{MM} Mather and McGehee constructed a system of four bodies on the line where the particles go to infinity
in finite time after an infinite number of binary collisions (it was known since the work of Sundman \cite{Su} that
binary collisions can be regularized so that the solutions can be extended beyond the collisions). Since
Mather-McGehee example had collisions it did not solve Conjecture \ref{ConNE} but it made it plausible.
Conjecture \ref{ConNE} was proved independently by Xia \cite{X} for the spacial five-body problem and by Gerver
\cite{G1} for a planar $3N$ body problem where $N$ is sufficiently large.
The problem still remained open for $N=4$ and for small $N$ in the planar case. However in \cite{G2} (see also \cite{G3})
Gerver
sketched a scenario which may lead to a non-collision singularity in the planar four-body problem. Gerver has not
published the details of his construction due to a large amount of computations involved (it suffices to mention that
even technically simpler large $N$ case took 68 pages in \cite{G1}). The goal of this paper is to realize Gerver's
scenario in the simplified setting of two-center-two-body problem. 

Conjecture \ref{ConZM} is mentioned by several authors, see e.g. \cite{Sim, Sa3, K}.
It is known that the set of initial conditions leading to the collisions has zero measure \cite{Sa1} and that the same is true
for non-collisions singularities if $N=4.$ To obtain the complete solution of this conjecture one needs to understand better
of the structure of the non-collision singularities and our paper is one step in this direction.

\subsubsection{Well-posedness in other systems} Recently the question of global well-posedness in PDE attracted a lot of attention motivated in
part by the Clay Prize problem about well-posedness of the Navier-Stokes equation.
One approach to constructing a blowup solutions for PDEs is to find a fixed point of
a suitable renormalization scheme and to prove the convergence towards this fixed point (see e.g. \cite{LS}). 
 The same scheme is
also used to analyze two-center-two-body problem and so we hope that the techniques developed in this paper
can be useful in constructing singular solutions in more complicated systems.

\subsubsection{Poincar\'{e}'s second species solution.}
In his book \cite{Po}, Poincar\'{e} claimed the existence of the so-called second species solution in three-body problem, which are periodic orbits converging to collision chains as $\mu\to 0$. 
The concept of second species solution was generalized to the non-periodic case.
In recent years significant progress was made in understanding second species solutions in 
both restricted \cite{BM, FNS} and full \cite{BN}
three-body problem.
However the understanding of general second species solutions generated by infinite aperiodic collision chains is still incomplete.
Our result can be considered as a generalized version of second species solution. All masses are positive and there are infinitely many close encounters.
Therefore the techniques developed in this paper can be useful in the study of the second species solutions.

\subsection{Extension to the four-body problem}
Consider the same setting as in our main result but suppose that $Q_1$ and $Q_2$ are also free (not fixed). Then we can expect that
during each encounter light particle transfers a fixed proportion of their energy and momentum to the heavy particle. The exponential
growth of energy and momentum would cause $Q_1$ and $Q_2$ to go to infinity in finite time leading to a non-collision singularity.

Unfortunately a proof of this involves a significant amount of additional computations due to higher dimensionality of the full four-body problem.
A good news is that similarly to the problem at hand, the Poincar\'{e} map of the full four-body problem will have only two strongly expanding directions whose origin could be understood by looking at our two-center-two-body problem.
The other directions will be dominated
by the most expanding ones. This allows our strategy to extend to the full four-body problem leading to the complete solution of the Painlev\'{e} conjecture. However, due to the length of the arguments, the details are presented in a separate paper
\cite{Xu}.

\subsection{Plan of the paper}
The paper is organized as follows. Section~\ref{section: main proof}
and~\ref{section: hyperbolic} constitute the framework of the proof. In Section \ref{section: main proof}
we give a proof of the main Theorem~\ref{ThMain} based on a careful study of the hyperbolicity of the 
Poincar\'e map. In Section~\ref{section: hyperbolic}, we summarize all calculations
needed in the proof of the hyperbolicity.
All the later sections provide calculations needed in Sections \ref{section: main proof} and
~\ref{section: hyperbolic}.
We define the local map to study the local interaction between $Q_3$ and $Q_4$ and global map to cover the time interval when $Q_4$ is traveling between $Q_1$ and $Q_2$. Sections ~\ref{section: equation},~\ref{section: variational},~\ref{section: boundary} and~\ref{section: switch foci} are devoted to
the global map, while Sections \ref{section: localappr},\ref{section: localC0}, and~\ref{section: local} study
local map.
Relatively short Sections \ref{subsection: plan} and \ref{ScCons} contain some technical results pertaining to both local and global maps.
Finally, we have two appendices. Appendix~\ref{section: appendix} contains
an introduction to the Delaunay coordinates for Kepler motion, which are
used extensively in our calculations. In Appendix~\ref{section: gerver}, we summarize the 
information about Gerver's model from \cite{G2}.

\section{Proof of the main theorem}\label{section: main proof}
\subsection{Idea of the proof}
The proof of the Theorem~\ref{ThMain} is based on studying the hyperbolicity of the Poincar\'{e} map. Our system has four degrees of freedom.
We pick the zero energy surface and then consider a Poincar\'{e} section. The resulting Poincar\'{e} map is six dimensional.
In turns out that for orbits of interest (that is, the orbits where the captured particle rotates around $Q_2$ and the traveler moves back and forth between
$Q_1$ and $Q_2$) there is an invariant cone field which consists of vectors close to a certain two dimensional subspace
such that all vectors in the cone are strongly expanding. This expansion comes from the combination of shearing
(there are long stretches when the motion of the light particles is well approximated by the Kepler motion and so the derivatives
are almost upper triangular) and twisting caused by the close encounters between $Q_4$ and $Q_3$ and between $Q_4$ and $Q_1.$
We restrict our attention to a two dimensional surface whose tangent space belong to the invariant cone and construct on such a surface
a Cantor set of singular orbits as follows. The two parameters coming from  the two dimensionality of the surface will be used to control the
phase of the close encounter between the particles and their relative distance. The strong expansion will be used to ensure that the
choices made at the next step will have a little effect on the parameters at the previous steps. This Cantor set construction based on the
instability of near colliding orbits is also among the key ingredients of the singular orbit constructions in \cite{MM} and~\cite{X}.

\subsection{Main ingredients}
In this section we present the main steps in proving Theorem \ref{ThMain}. In Subsection \ref{SSGer}
we describe a simplified model for constructing singular solutions given by Gerver \cite{G2}.
This model is based on the following simplifying assumptions:

\begin{itemize}
\item $\mu=0,\ \chi=\infty$ so that $Q_3$(resp. $Q_4$) moves on a standard ellipse (resp. hyperbola). 
\item The particles $Q_3,Q_4$ do not interact except during a close encounter.
\item Velocity exchange during close encounters can be modeled by an elastic collision.
\item The action of $Q_1$ on light particles can be ignored except that
during the close encounters of the traveler particle with $Q_1$ the angular momentum of
the traveler with respect to $Q_2$ can be changed arbitrarily.
\end{itemize}
The main conclusion of \cite{G2} is that the energy of the captured particle can be increased by
a fixed factor while keeping the shape of its orbit unchanged. Gerver designs a two step procedure 
with collisions having the following properties:
\begin{itemize}
\item The incoming and outgoing asymptotes of the traveler are horizontal.
\item The major axis of the captured particle remains vertical.
\item After two steps of collisions, the elliptic orbit of the captured particle has the same eccentricity but smaller semimajor axis compared with the elliptic orbit before the first collision (see Fig 1 and 2).
\end{itemize}
For quantitative information, see Appendix~\ref{section: gerver}.

Since the shape is unchanged after the two trips described above the procedure
can be repeated. Then the kinetic energies of the particles grow exponentially and so the time needed
for $j$-th trip is exponentially small. Thus the particles can make infinitely many trips in finite time leading
to a singularity. Our goal therefore is to get rid of the
above mentioned simplifying assumptions.
\begin{figure}[ht]
\begin{center}
\includegraphics[width=0.8\textwidth]{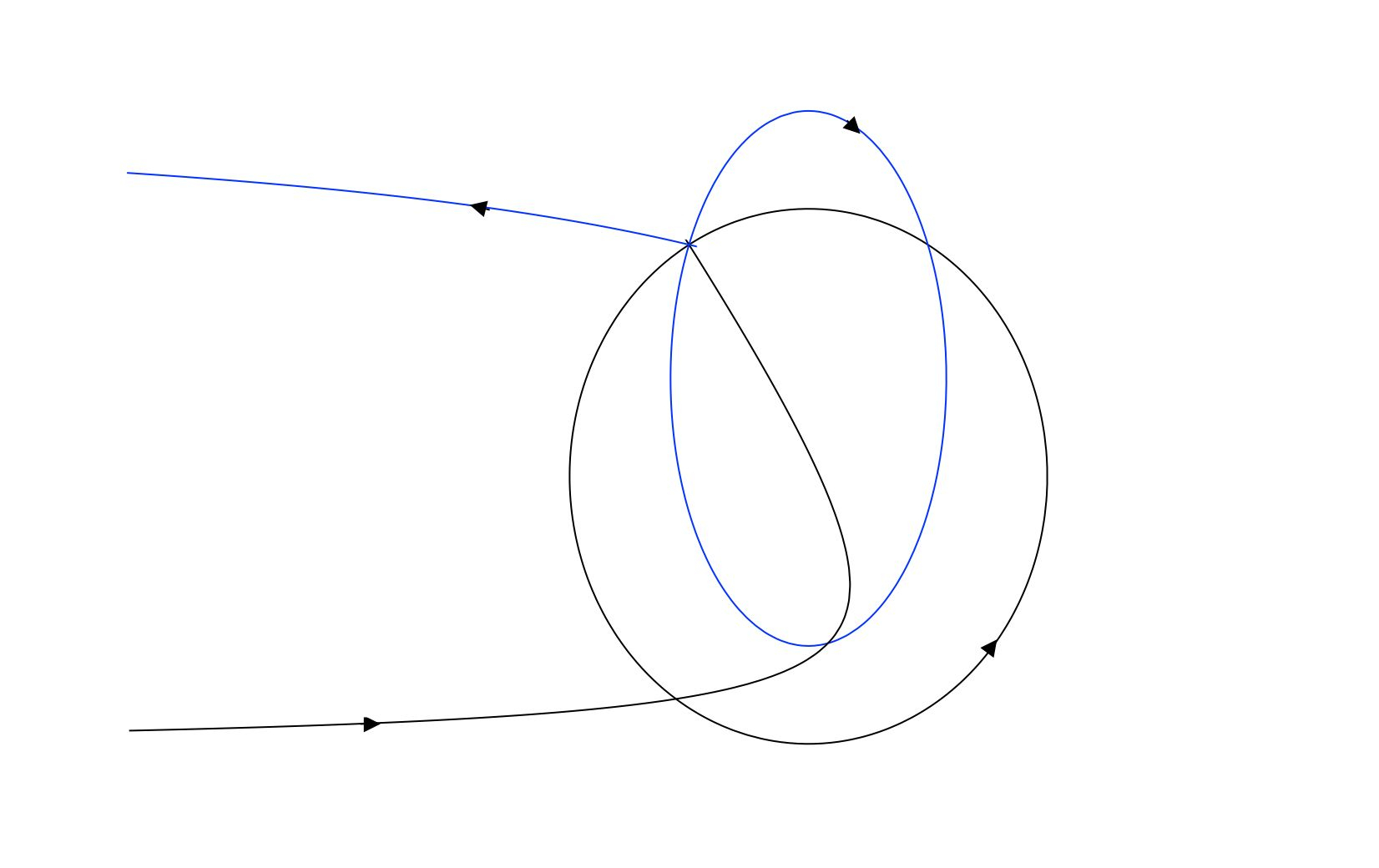}
\caption{Angular momentum transfer}
\end{center}
\label{fig:coll1}
\end{figure}

\begin{figure}[ht]
\begin{center}
\includegraphics[width=0.8\textwidth]{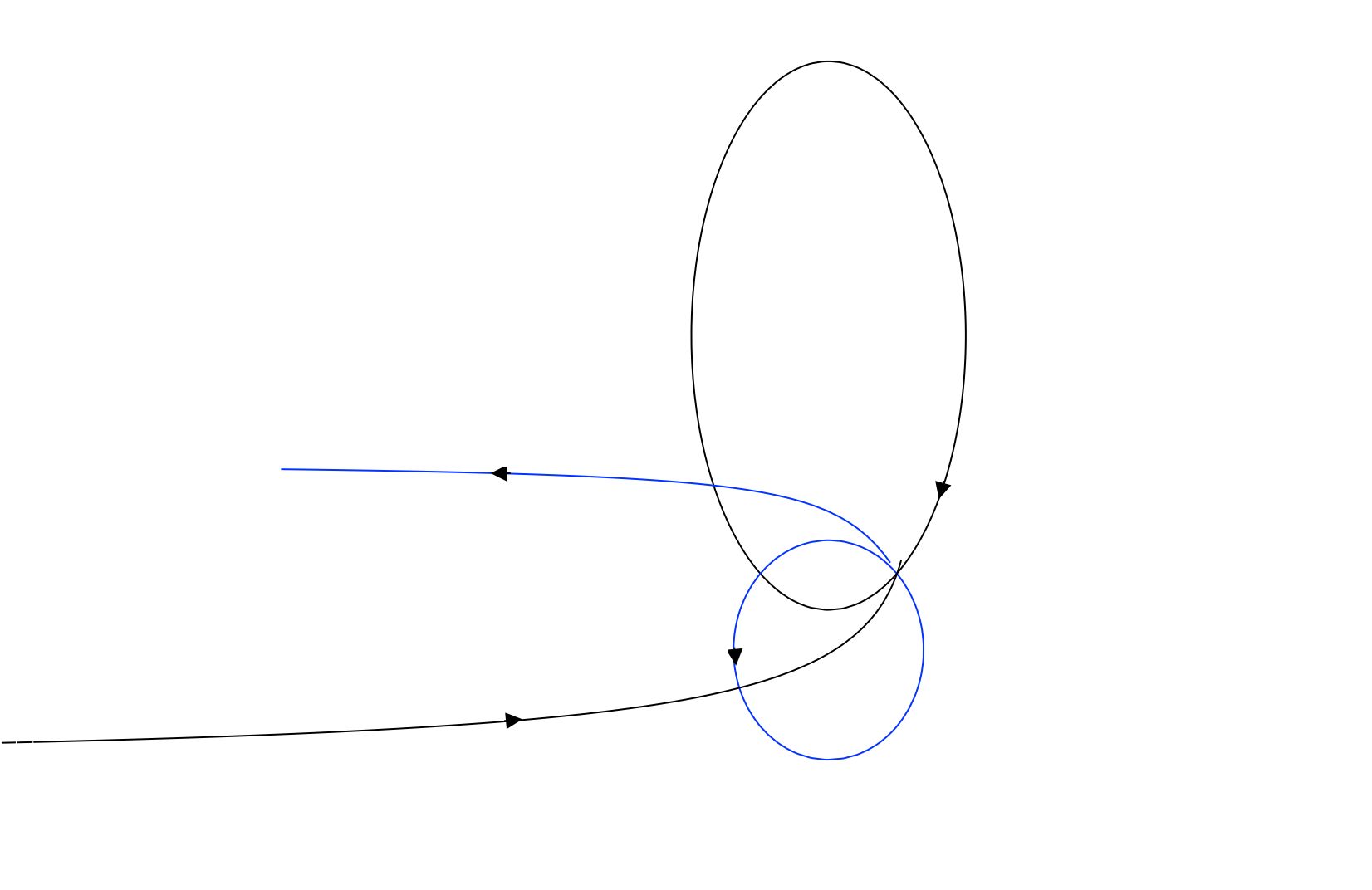}
\caption{Energy transfer}
\end{center}
\label{fig:coll2}
\end{figure}

In Subsection \ref{SSLoc} we study near collision of the light particles.
This assumption that velocity exchange can be modeled by elastic collision
is not very restrictive since both energy and momentum are conserved during the exchange
and any exchange of velocities conserving energy and momentum amounts to rotating the relative velocity
by some angle and so it can be effected
by an elastic collision. In Subsection \ref{SSGlob} we state a result saying that away from the close encounters the 
interaction between the light particles as well as the action of $Q_1$ on the particle which is captured by $Q_2$
can indeed be disregarded.
In Subsection \ref{SSAdm} we study the Poincar\'{e}
map corresponding to one trip of the traveller particle around $Q_1.$ After some technical preparations we present
the main result of that section--Lemma \ref{LmChooseAM} which says that after this trip the angular momentum of
the traveler particle indeed can change in an arbitrary way. Finally in Subsection \ref{SSProof} we show
how to combine the above ingredients to construct a Cantor set of singular orbits.

In \eqref{eq: main}, we make the change of variables $P_i=\mu v_i, i=3,4$ and divide the Hamiltonian by $\mu$. This rescaling changes the symplectic form by a conformal factor but does not change the Hamiltonian equations. The rescaled Hamiltonian, still denoted by $H$ has the following form
\begin{equation}\label{eq: hamloc}H=\frac{|v_3|^2}{2}+\frac{|v_4|^2}{2}-\frac{1}{|Q_3|}-\frac{1}{|Q_3+(\chi,0)|}-\frac{1}{|Q_4|}-\frac{1}{|Q_4+(\chi,0)|}-\frac{\mu}{|Q_3-Q_4|}.\end{equation}
We have $v_i=\dot Q_i$ and we use $x,y$ to denote the components of $Q$, $Q_i=(x_i,y_i),\ i=3,4$. 
The orbit of Kepler motion can be parametrized by four variables $(v,Q)\in \R^4$ or in Delaunay coordinates $(L,\ell,G,g)$.
The symplectic transformation between the two coordinates is given explicitly in Appendix \ref{section: appendix}.
The geometric meanings of the Delaunay variables are as follows. For elliptic motion, $L^2$ is the length of the
semi major axis, $LG$ is the length of the semi minor axis, and $g$ is the argument of periapsis (direction).
These three variables characterize the shape of the ellipse.
The variable $\ell$ called mean anomaly indicates the position of the moving body on the ellipse.
For Kepler hyperbolic motion, Delaunay coordinates can also be introduced and have similar meanings.
See Appendix \ref{section: appendix} for more details. In the following we use subscript $3,4$ to denote the corresponding variables for $Q_3$ or $Q_4$. 

\subsection{Gerver map}
\label{SSGer}

Following \cite{G2}, we discuss in this section the limit case $\mu=0,\chi=\infty$.
We assume that $Q_3$ has elliptic motion and $Q_4$ has hyperbolic motion with respect to the focus $Q_2$.
Since $\mu=0,$  $Q_3$ and $Q_4$ do not interact unless they have exact collision.
Since we assume that $Q_4$ just comes from the interaction from $Q_1$
located at $(-\infty, 0)$ and the new traveler particle is going to interact with $Q_1$ in the future,
the slope of incoming asymptote $\theta_4^-$ of $Q_4$ and that of the outgoing asymptote $\bar{\theta}^+$ of the traveler particle
should satisfy $\theta^-_4=0, \ \bar{\theta}^+=\pi$.

The Kepler motions of $Q_3$ and $Q_4$ has three first integrals $E_i, G_i$ and  $g_i$
where $E_i$ denotes the energy, $G_i$ denotes the angular momentum and $g_i$ denotes the argument of periapsis.
Since the total energy of the system is zero we have $E_4=-E_3.$ Note that
\begin{equation}
\label{En-L}  
E_3:=\frac{-1}{2L^2_3}=\frac{|v_3|^2}{2}-\frac{1}{|Q_3|}.
\end{equation}
It turns out convenient to use eccentricities
\begin{equation}
\label{Ec-GE}  
e_i=\sqrt{1+2G^2_i E_i}
\end{equation}
instead of $G_i$ since
the proof of Theorem \ref{ThMain} involves a renormalization transformation and $e_i$
are scaling invariant.
The Gerver map describes the parameters of the elliptic orbit change during the interaction of $Q_3$ and $Q_4.$
The orbits of $Q_3$ and $Q_4$ intersect in two points. We pick one of them.
{\it We label the intersection points in the reverse chronological order with respect to
  the motion of $Q_4.$} (This labeling is done so that the first intersection point is used at the first
step of the Gerver's construction and the second point is used at the second step of the Gerver
construction, see Figures 1 and 2.). Thus
we use a discrete parameter $j\in\{1, 2\}$ to describe which intersection point is selected.

Since $Q_3$ and $Q_4$ only interact when they are at the same point the only effect of the interaction is to change their
velocities. Any such change which satisfies energy and momentum conservation can be described by an
elastic collision. That is, velocities before and after the collision are related by
\begin{equation}
\label{Ellastic}
 v_3^+=\dfrac{v_3^-+v_4^-}{2}+\left|\dfrac{v_3^--v_4^-}{2}\right| n(\al), \quad
v_4^+=\dfrac{v_3^-+v_4^-}{2}-\left|\dfrac{v_3^--v_4^-}{2}\right| n(\al),
\end{equation}
where $n(\al)$ is a unit vector making angle $\al$ with $v_3^--v_4^-.$

With this in mind we proceed to define the Gerver map
$\Ger_{e_4, j, \omega}(E_3, e_3, g_3).$
This map depends on two discrete parameters $j\in \{1, 2\}$ and
$\omega\in\{3, 4\}.$ The role of $j$ has been explained above, and $\omega$ will tell us which particle will be the traveler
after the collision.

To define $\Ger$ we assume that $Q_4$ moves along the hyperbolic orbit with parameters $(-E_3, e_4, g_4)$
where $g_4$ is fixed by requiring that the incoming asymptote of $Q_4$ is horizontal. We assume that $Q_3$ and $Q_4$
arrive to the $j$-th intersection point of their orbit simultaneously. At this point their velocities are changed by \eqref{Ellastic}.
After that the particle proceed to move independently.
Thus $Q_3$ moves on an orbit with parameters $(\brE_3, \bre_3, \brg_3)$, and
$Q_4$ moves on an orbit with parameters
$(\brE_4, \bre_4, \brg_4).$

If $\omega=4$, we choose $\al$ in \eqref{Ellastic}
so that after the exchange $Q_4$ moves on hyperbolic
orbit and $\brtheta_4^+=\pi$ and let
$$ \Ger_{e_4, j, 4}(E_3, e_3, g_3)=(\brE_3, \bre_3, \brg_3).$$
If $\omega=3$ we choose $\al$ in \eqref{Ellastic}
so that after the exchange $Q_3$ moves on hyperbolic
orbit and $\brtheta_3^+=\pi$ and let
$$ \Ger_{e_4, j, 3}(E_3, e_3, g_3)=(\brE_4, \bre_4, \brg_4).$$

\begin{Rk}
  If the index $j$ is used to define to the Gerver map then we refer to $j$-th intersection point of the orbits of
  $Q_3$ and $Q_4$ as {\bf Gerver collision point}. We refer to Appendix \ref{section: gerver} for the coordinates
  of Gerver's collision points.
  It is important in Gerver's model that if $Q_3$ and $Q_4$
  have a close encounter near the Gerver point then they do not have another close encounter before the next
  trip of the traveller particle. This fact is proven in \cite{G2}. For the reader's convenience we
  reproduce Gerver's argument in Section \ref{subsection: nocollision}.
\end{Rk}  

\emph{In the following, to fix our notation, we always call the captured particle $Q_3$ and the traveler $Q_4$.
}

Below
we denote the ideal orbit parameters in Gerver's paper \cite{G2} of $Q_3$ and $Q_4$ before the first (respectively second) collision with * (respectively **).
Thus, for example, $G_4^{**}$ will denote the angular momentum of $Q_4$ before the second collision. Moreover, the
actual values after the first (respectively, after the second) collisions are denoted with a $bar$ or $double\ bar$.


Note $\Ger$ has a skew product form
$$ \bre_3=f_e(e_3, g_3, e_4), \quad
\brg_3=f_g(e_3, g_3, e_4), \quad
\brE_3=E_3 f_E(e_3, g_3, e_4) . $$
This skew product structure will be crucial in the proof of Theorem \ref{ThMain} since it will allow us to iterate $\Ger$
so that $E_3$ grows exponentially while $e_3$ and $g_3$ remains almost unchanged.

The following fact plays a key role in constructing singular solutions.

\blm [\cite{G2}]
\label{LmGer}
Assume that the total energy of the $Q_2,Q_3,Q_4$ system is zero.
\begin{itemize}
\item[(a)]For $E_3^*=\frac{1}{2}, g_3^*=\frac{\pi}{2}$ and for any
  $\ e_3^*\in (0,\frac{\sqrt 2}{2})$, there exist $e^*_4, e^{**}_4,\lb_0>1$ such that 
$$ (e_3, g_3, E_3)^{**}=\Ger_{e_4^*, 1, 4}\left(e_3, g_3, E_3 \right)^*, \quad (e_3, -g_3, \lambda_0 E_3)^*=\Ger_{e_4^{**}, 2, 4}\left(e_3, g_3, E_3 \right)^{**}, $$
where $E_3^{**}=E_3^*=\frac{1}{2},\ g_3^{**}=g_3^*=\frac{\pi}{2}$ and $e_3^{**}=\sqrt{1-e_3^{*2}}$.
\item[(b)] There is a constant $\brdelta$ such that if $(e_3,g_3,E_3)$
  lie in a $\brdelta$ neighborhood of $(e_3^*,g_3^*,E_3^*),$  then there exist smooth functions $e_4'(e_3, g_3),$ $e_4''(e_3, g_3),$ and $\lambda(e_3, g_3,E_3)$ such that
$$ e_4'(e_3^*, g_3^*)=e_4^*, \quad e_4''(e_3^*, g_3^*)=e_4^{**}, \quad \lambda(e^*_3, g^*_3,E^*_3)=\lb_0,$$
\begin{equation}\nonumber
\begin{aligned}
(\bre_3, \brg_3, \brE_3)&=\Ger_{e_4'(e_3, g_3), 1, 4}\left(e_3, g_3, E_3 \right),\\ 
(e^*_3, -g^*_3, \lb(e_3,g_3,E_3)E_3^*)&=\Ger_{e_4''(e_3, g_3), 2, 4}\left(\bre_3, \brg_3, \brE_3 \right). \end{aligned}\end{equation}
\end{itemize}
\elm
In Section \ref{subsection: local3}, we will give a set of equations
(equations \eqref{eq: polarcollision1}-\eqref{eq: polarcollision9}) whose solutions give the map
$\Ger$, and the smoothness of $e',e''$ follows from the implicit function theorem. We remark that $e',e''$ do not depend on $E_3$ since $e_4,e_3,g_3$ are rescaling invariant, and we can always rescale $E_3$ to $E_3^*.$
Part (a) allows us to increase energy after two collisions without changing the shape of the orbit in the limit
case $\mu=0, \chi=\infty.$ Part (b) allows us to fight against the perturbation coming from the fact
that $\mu>0$ and $\chi<\infty.$ Lemma \ref{LmGer} is a slight restatement of the main result of \cite{G2}. Namely part (a) is proven in Sections 3 and 4 of \cite{G2}
and part (b) is stated in Section 5 of \cite{G2} (see equations (5-10)--(5-13)). The proof of part (b) proceeds by a routine
numerical computation. For the reader's convenience we review the proof of Lemma \ref{LmGer} in Appendix \ref{section: gerver} explaining how the numerics is done.

\begin{Rk}
We try to minimize the use of numerics in our work. The use of numerics is always preceded by mathematical derivations. Readers can see that the numerics in this paper can also be done without using computer.
We prefer to use the computer since computers are more reliable than humans when doing routine computations.
\end{Rk}

\subsection{Asymptotic analysis, local map}
\label{SSLoc}

Starting from this section, we work on the Hamiltonian system \eqref{eq: main}. We assume that the two centers are at distance $\chi\gg 1$ and that $Q_3,Q_4$ have positive masses $0<\mu\ll 1$. We will see below that $\chi$ grows exponentially to infinity under iterates due to the renormalization, so we always assume $1/\chi\ll\mu\ll 1$ without loss of generality. Therefore the motions of $Q_3$ and $Q_4$ can be approximated by Kepler motions at least for a  short time interval if they are away from collisions.
We use the Delaunay coordinates $(L,\ell,G,g)_{3,4}$ (elliptic for 3 and hyperbolic for 4) to describe the motions of $Q_3$ and $Q_4$ when $Q_3$ and $Q_4$ are in a $O_{\chi\to\infty}(1)$ neighborhood of $Q_2$. We assume $Q_3$ is captured by $Q_2$. Namely, the energy $E_3$ of $Q_3$ is negative where the energy \eqref{En-L}   is the sum of the kinetic energy and the potential energy relative to $Q_2$. The system has four degrees of freedom. By restricting to the zeroth energy level and picking a Poincar\'e section, we get a six dimensional space as our phase space on which the Poincar\'e map is defined. The Poincar\'e section is chosen as $\{x_4=-2,\ \dot{x}_4>0 \}$. 
We choose the orbit parameters as $(E_3,\ell_3, e_3,g_3, e_4,g_4)\in \R^4\times \T^2$ which are obtained from
the Delaunay variables using \eqref{En-L}--\eqref{Ec-GE}. 
The energy $E_4$ of $Q_4$ is eliminated using energy conservation and $\ell_4$ is treated as the new time, which is also eliminated by considering the Poincar\'e map instead of flow. 

We consider initial conditions in the following sets. We denote $$K:=\max_{\dagger=*,**}\Vert d\Ger_{e_4^\dagger, 1, 4}\left(e_3, g_3, E_3 \right)^\dagger\Vert+1,\quad K':=\max_{\dagger=*,**}\Vert d(e_4',e_4'')(e_3,g_3)^\dagger\Vert+1.$$ 
Given $\delta<\brdelta/(KK')$ where $\bar\dt$ is in Lemma \ref{LmGer}, consider open sets in the phase space (zero energy level and the Poincar\'e section $\{x_4=-2,\ \dot{x}_4>0 \}$) defined by
$$ U_1(\delta)=\left\{ \left|E_3-\left(-\dfrac{1}{2}\right)\right|,\ |e_3-e_3^*|,\ |g_3-g_3^*|,\
|\theta_4^-|<\dt,\ |e_4-e_4^*|<K'\delta \right\}, $$
$$ U_2(\delta)=\left\{ |E_3-E_3^{**}|,\ |e_3-e_3^{**}|,\ |g_3-g_3^{**}|,\
|\theta_4^-|<K\dt,\ |e_4-e_4^{**}|<KK'\delta \right\}. $$
In both $U_1(\dt)$ and $U_2(\dt)$, the angle $\ell_3$ can take any value in
$\T^1$. 

{\it Throughout the paper, we reserve the notations $K,K',\dt,\bar\dt$. }

We let particles move until one of the
particles  moving on hyperbolic orbit reaches the surface $\{x_4=-2,\ \dot{x}_4<0\}$. We measure the final orbit parameters $(\brE_3,\bar{\ell}_3, \bre_3,\brg_3,\bre_4,\brg_4)$. We call the mapping moving initial positions
of the particles to their final positions the {\bf local map} $\Loc$.
In Fig. 3 of Section \ref{SSExp} the local map is to the right of the section $\{x=-2\}$. We are only interested in those initial conditions in $U_j(\delta),\ j=1,2$ which lead to close encounter between $Q_3$ and $Q_4$, since otherwise $Q_4$ moves on one slightly perturbed hyperbola with non-horizontal outgoing asymptote and will escape from the system (Sublemma \ref{KeepDirection}). To select these initial conditions of interest, we impose one more boundary condition. 
 
\blm
\label{LmLMC0} Fix any constant $C_1>0$ and $j\in \{1,2\}$. 
Suppose that  the initial orbit parameters
$(E_3,\ell_3,e_3,g_3,e_4,g_4)$ are chosen in $U_j(\dt),$ 
such that the orbit passes through a $\dt$ neighborhood of the $j$-th Gerver's collision point, and
the traveler particle$($s$)$ satisfy
$|\theta_4^-|\leq C_1\mu$ and $|\bar{\theta}_4^+-\pi|\leq C_1\mu$.
Then the following asymptotics holds uniformly
\[(\brE_3, \bre_3, \brg_3) =\Ger_{e_4,j,4}(E_3,e_3,g_3)+o(1),
\mathrm{\ as\ }1/\chi\ll \mu\to 0.\]
\elm
Thus the condition that the orbit parameters of $Q_4$ (in particular $\bar{\theta}_4^+$)
change significantly
forces $Q_3$ and $Q_4$ to have a closer encounter. 
The lemma tells us that Gerver map
is a good $C^0$ approximation of the local map $\Loc$ for the real case $0<1/\chi\ll\mu\ll 1$ for the orbits of interest.
Lemma \ref{LmLMC0} will be proven in Section \ref{ScC0Loc}.
\subsection{Asymptotic analysis, global map}
\label{SSGlob}
As before we assume that the two centers are at distance $\chi\gg 1.$ Fix a large constant~$C_2.$
We assume that initially $Q_3$ moves on an elliptic
orbit, $Q_4$ moves on hyperbolic orbit and $\{x_4(0)=-2, $ $\dot{x}_4(0)<0\}.$
We assume that $|y_4(0)|<C_2$ and that, after moving
around $Q_1,$  $Q_4$ hits the surface  $\{x_4=-2,\ \dot{x}_4>0\}$ so that $|y_4|<C_2$.
We call the mapping moving initial positions
of the particles to their final positions the (pre) {\bf global map} $\Glob$. 
In Section \ref{SSAdm} we will slightly modify the definition of the global map but it will not change the essential features 
discussed here. 
In Fig. 3 from Section \ref{SSExp}, the global map is to the left of the section $\{x=-2\}$.
We let $(E_3,\ell_3, e_3, g_3,e_4, g_4)$ denote the initial orbit parameters measured in the section $\{x_4=-2,\ \dot{x}_4<0\}$ and
$(\brE_3,\bar{\ell}_3, \bre_3, \brg_3, \bre_4, \brg_4)$ denote the final orbit parameters measured in the section $\{x_4=-2,\ \dot{x}_4>0\}$.
\blm
\label{LmGMC0}
Assume that $|y_4|<C_2$ holds both at initial and final moments and assume that we have initially
$|E_3-E^{\dagger}_3|, |e_3-e_3^\dagger|,|g_3-g^\dagger_3|<2\bar\dt$ where $\dagger=*$ or $**$ 
and $(E_3^\dagger, e_3^\dagger, g_3^\dagger)$ are defined in Lemma \ref{LmGer}.
Then there exists $C_3$ such that 
uniformly in $\chi, \mu$ we have the following estimates
\begin{enumerate}
\item[(a)] $|\brE_3-E_3|\leq C_3\mu, \quad |\bar e_3-e_3|\leq C_3\mu, \quad |\brg_3-g_3|\leq C_3\mu.$
\item[(b)] $|\theta_4^+-\pi|\leq C_3\mu,\quad |\brtheta_4^-|\leq C_3\mu.$
\item[(c)]  The flow time between the initial and final moments bounded by $C_3\chi$.
\end{enumerate}
\elm
The proof of this lemma is given in Section~\ref{section: equation}. Notice that in the above two lemmas,
we control the orbit parameters $E_3,e_3,g_3,\theta_4$,
but we do not talk about $\ell_3,e_4$ (recall that $g_4$ can be solved from $\theta_4,L_4,G_4$). Most of the work of the paper is devoted to showing that there are two strongly expanding directions of the Poincar\'e map which enable us to prescribe $\ell_3,e_4$ arbitrarily. 

We also need the following fact which says that $Q_3$ if initially captured by $Q_2$ will always be captured. 

\begin{Lm}
  \label{Lm: boundQ3}
  Let $C_2$ be as in Lemma \ref{LmGMC0}.  Suppose the initial orbit parameters $\bx=(E_3,\ell_3,e_3,g_3,e_4,g_4)\in U_j(\dt)$ and the image
  $\Glob\circ\Loc(\bx)$ has $|y_4|\leq C_2.$  Then there are constants $\mu_0, \chi_0, D$ such that for $\mu\leq \mu_0$
  and $\chi\geq \chi_0$  we have $|Q_3(t)|\leq 2-D$ for all $t$ up to the time needed to define
  $\Glob\circ\Loc.$ 
\end{Lm}
The proof of this lemma is also given in Section \ref{ScC0Loc}.

\subsection{Admissible surfaces}
\label{SSAdm}
Given a sequence $\bomega$ we need to construct orbits having singularity with symbolic sequence $\bomega.$

We will study the Poincar\'{e} map $\cP=\Glob\circ\Loc$ to the surface $\{x_4=-2, $ $\dot{x_4}>0\}.$ It is a composition of the local and global maps defined in the previous sections.

We will also need the renormalization map $\cR$ defined as follows. In Cartesian coordinates, we partition our six dimensional section $\{x_4=-2, $ $\dot{x}_4>0\}$ into 
coordinate cubes of size $1/\sqrt{\chi}$. We next evaluate $E_3$ at the center of each cube and denote its value by $-\lb/2$, where $\lb>1$ is $\bar\dt$-close to $\lb_0$ in Lemma \ref{LmGer}. 
The locally constant map $\cR$ amounts to zooming in the configuration $Q_i=(x_i,y_i), i=3,4,$ by multiplying by $\lb$ and slowing down the velocity $v_i,\ i=3,4$ by dividing through $\sqrt{\lb}.$ 
In addition we reflect the coordinates along the $x$ axis. In Cartesian coordinates, the renormalization takes the form
\begin{equation}\label{EqRenorm}
\cR\left((v_{i,x},v_{i,y}), (x_i,y_i),H,t\right)= \left(\frac{(v_{i,x},-v_{i,y})}{\lb^{1/2}}, \lb(x_i,-y_i),\frac{H}{\lb},\lb^{3/2}t\right),\ i=3,4.
\end{equation}
Since the renormalization $\cR$ sends the section $\{x_4=-2\}$ to $\{x_4/\lb=-2\}$, we push forward each cube along the flow to the section $\{x_4=-2/\lb, $ $\dot{x_4}>0\}.$ {\it We include the piece of orbits from the section $\{x_4=-2, $ $\dot{x_4}>0\}$ to $\{x_4=-2/\lb, $ $\dot{x_4}>0\}$ to the global map and apply the $\cR$ to the section $\{x_4=-2/\lb, \ \dot{x_4}>0\}$.} This is then followed by a reflection. We have $\cR(\{x_4=-2/\lb,\ \dot{x_4}>0\})= \{x_4=-2, \ \dot{x_4}>0\},$ and \[\cR(E_3,\ell_3,e_3,g_3,e_4,g_4)=( E_3/\lb, \ell_3,e_3,-g_3,e_4,-g_4), \]
where minus signs are the effect of the reflection. 

 Note that the rescaling changes
(for the orbits of interest, increases) the distance between the fixed centers by sending $\chi$ to $\lb\chi$. Observe that at each step we have the
freedom of choosing the centers of the cubes. We describe how this choice is made in Section \ref{section: hyperbolic}.
In the following we give a proof of the main theorem based on the three lemmas, whose proofs are in the next section.

We need to define cone fields $\cK_1$ on $T_{U_1}(\R^4\times \T^2)$ and $\cK_2$ on $T_{U_2}(\R^4\times \T^2).$ 
Fix a small constant $\eta.$
\begin{Def}\label{DefKone}
 Let $\cK_1$ to be the set of vectors which make an angle less than a small number $\eta$ with
$\Span(d\cR w_2, \tw),$ and $\cK_2$ to be the set of vectors which make an angle less than $\eta$ with
  $\Span(w_1, \tw),$ where
  $$\tw=\frac{\partial}{\partial \ell_3} \text{ and }
  w_j=\frac{\partial e_4}{\partial G_4}\frac{\partial }{\partial e_4}-\frac{ L_4}{L_4^2+G_4^2}\frac{\partial }{\partial g_4},\  j=1,2.$$
\end{Def}

\blm
\label{LmKone}
There is
a constant $c>0$
such that
for all $\bx\in U_1(\delta)$ satisfying $\mathcal{P}(\bx)\in U_2(\dt)$,  
and for all $\bx\in U_2(\delta)$ satisfying $\cR\circ\mathcal{P}(\bx)\in U_1(\dt)$, 
\begin{itemize}
\item[(a)] $d\cP (\cK_1)\subset \cK_2$, $d(\cR\circ \cP) (\cK_2)\subset \cK_1$.
\item[(b)] If  $v\in\cK_1$, then $\Vert d\cP(v)\Vert \geq c\chi \Vert v\Vert $.\\
If $v\in\cK_2$, then $\Vert d(\cR\circ \cP)(v)\Vert \geq c\chi \Vert v\Vert .$ 
\end{itemize}
\elm

We call a two dimensional $C^1$ surface $S_1\subset U_1(\delta)$ (respectively $S_2\subset U_2(\delta)$)
{\bf admissible} if $TS_1\subset \cK_1$ (respectively $TS_2\subset \cK_2$).
Then item (a) of Lemma \ref{LmKone} implies that the image of admissible surface is also admissible.
More precisely, if $S_1$ is admissible and $\cP(S_1)\cap U_2(\dt)\neq\emptyset$, then $T_{U_2(\dt)}\cP(S_1)\subset \cK_2$.
A similar statement holds for the higher iterates.

From the explicit construction of the cones we get the following lemma. 
\blm
\label{LmAdm}
\begin{itemize}
\item[(a)] The vector $\tw=\frac{\partial}{\partial \ell_3}$ is in $\cK_i.$
\item[(b)] For any plane $\Pi$ in $\cK_i$ the projection map
$ \pi_{e_4, \ell_3}=(de_4, d\ell_3):\Pi\to\R^2$
is one-to-one.
In other words
$(e_4, \ell_3)$ can be used as coordinates on admissible surfaces.
\end{itemize}
\elm
Using the invariance of the cone fields, we can reduce the six dimensional Poincar\'e map to a two dimensional map defined on a cylinder. The reduction is done as follows.  
We introduce the following cylinder sets $$\mathcal C_1(\dt)=(e^*_4-K'\dt,e^*_4+K'\dt)\times \T^1,\quad  \mathcal C_2(\dt)= (e^{**}_4-KK'\dt,e^{**}_4+KK'\dt)\times \T^1.$$
By Lemma \ref{LmAdm}, each piece of admissible surface $S$ in $U_j(\dt)$ is a graph of a function
$\cS$ of the variables $(e_4,\ell_3)\in \mathcal C_j(\dt).$
Hence $\cP (\cS(e_4,\ell_3))$ becomes a function of two variables $(e_4,\ell_3)$.
However, $\cP(\cS(\cdot,\cdot))$ is well defined only on subsets of small measure in $\mathcal C_j(\dt),$ since 
for most points $(e_4,\ell_3)\in \mathcal C_j(\dt)$ 
the points $\cS(e_3,\ell_3)$ have orbits for which $Q_4$ escapes from the system.
The next lemma shows that certain open set $V$ can always be found in $\mathcal C_j(\dt)$ on which
$\cP(\cS(\cdot,\cdot))$ is defined and has large image where we call an admissible surface
$S$ {\bf large} if $\pi_{e_4, \ell_3}S$ contains $\mathcal C_j(\dt).$ In particular, given
$e_4\in (e^{*}_4-K'\dt,e^{*}_4+K'\dt)$ or $(e^{**}_4-KK'\dt,e^{**}_4+KK'\dt)$,  we can prescribe $\ell_3$ arbitrarily.

 Since the part of $\cP(\mathcal S)$ consisting of points which land on $U_1(\delta)$ or $U_2(\delta)$
 is also admissible by Lemma \ref{LmKone}, we can apply Lemma \ref{LmAdm} again to project the image to the 
 $(e_4,\ell_3)$ cylinder. Therefore we introduce the notation 
\[\mathcal Q_1:=\pi_{e_4, \ell_3}\cP(\cS(\cdot,\cdot)),\quad \mathcal Q_2:=\pi_{e_4, \ell_3}\cR\circ \cP(\cS(\cdot,\cdot)),\]
whenever they are defined. $\mathcal Q_j$ is a map from a subset of $\mathcal C_j(\dt)$ to 
$\mathcal C_{3-j}(\dt),\ j=1,2$. 


\blm
\label{LmChooseAM}
For any $0<\dt\leq\bar\dt/(KK')$, we have the following.
\begin{itemize} 
\item[(a)] Given a large admissible surface $S_1\subset U_1(\delta)$ and $\te_4\in (e^*_4-K'\dt,e^*_4+K'\dt)$
 there exists $\tilde{\ell}_3$ such that
$\cP(\cS_1(\te_4, \tilde{\ell}_3))\in U_2(\delta).$ Moreover if $|\te_4- e^*_4|<K'\dt-1/\chi$, then
there is a neighborhood $V(\te_4)\subset \mathcal C_1(\dt)$ of $(\te_4, \tilde{\ell}_3)$
such that $\mathcal Q_1$ maps $V$ surjectively to $\mathcal C_2(\dt)$.
\item[(b)] Given a large admissible surface $S_2\subset U_2(\delta)$ and $\te_4\in (e^{**}_4-KK'\dt,e^{**}_4+KK'\dt)$ there exists $\tilde{\ell}_3$ such that
$\cR\circ \cP(\cS_2(\te_4, \tilde{\ell}_3))\in U_1(\delta).$ Moreover if $|\te_4- e_4^{**}|<KK'\dt-1/\chi$, then
there is a neighborhood $V(\te_4)\subset \mathcal C_2(\dt)$ of $(\te_4, \tilde{\ell}_3)$
such that $\mathcal Q_2$ maps $V$  surjectively to $\mathcal C_1(\dt)$.
\item[(c)] For points in $V(\te_4)$ from parts $(a)$ and $(b)$, there exist $c,\mu_0,\chi_0$ such that for $\mu<\mu_0,\,\chi>\chi_0$, the particles avoid collisions before the next return and
  the minimal distance $d$ between the particles satisfies
$$c\mu\leq d \leq \frac{\mu}{c}.$$
\end{itemize}
\elm

Note that by Lemma \ref{LmKone} the diameter of $V(\te_4)$ is $O(\delta/\chi).$
The proof of Lemma \ref{LmChooseAM} is given in Section \ref{SSKone}.

\subsection{Construction of the singular orbit}
\label{SSProof}

Fix a number $\eps$ which is much smaller than $\delta$ but is much larger than both $\mu$ and $1/\chi.$
Pick $(\he_3, \hg_3)$ so that
$$ |\he_3-e_3^*|\leq \frac{\delta}{2}, \quad |\hg_3-g_3^*|\leq \frac{\delta}{2}. $$
Let $S_0$ be an admissible surface such that the diameter of $S_0$ is much larger
than $1/\chi$ and such that on $S_0$ we have
$$|e_3-\he_3|<\eps, \quad |g_3-\hg_3|<\eps.$$
For example, we can pick a point $\bx\in U_1(\delta)$ and let $\hw$ be a vector in $\cK_1(\bx)$ such that
$\frac{\partial}{\partial \ell_3}(\hw)=0.$ Then let
$$ S_0=\{(E_3, \ell_3, e_3, g_3, e_4, g_4)(\bx)+a\hw+(0, b, 0, 0, 0, 0)
\text{ where } |a|\leq \eps/\brK,\ b\in \T^1\} $$
and $\brK$ is a large constant.

 We wish to construct a singular orbit in $S_0.$
We define $S_j$ inductively so that $S_j$ is a component of $\cP(S_{j-1})\cap U_2(\delta)$ if $j$ is odd
and $S_j$ is a component of $(\cR\circ\cP)(S_{j-1})\cap U_1(\delta)$ if $j$ is even (we shall show below that
such components exist). Let $\bx=\lim_{j\to\infty}(\cR\cP^2)^{-j} S_{2j}.$ We claim that $\bx$ has singular orbit.

We define $t_0=0$ and let $t_j$ be the time of $\bx$'s $2j$-th visit to the section $\{x_4=-2$, $\dot x_4>0\}$. Since the global map gives only $O(\mu)$ small oscillation to
$E_3=\frac{|v_3|^2}{2}-\frac{1}{|Q_3|}$ by Lemma \ref{LmGMC0},
and the local map is approximated by the Gerver map by Lemma \ref{LmLMC0}, we apply Lemma \ref{LmGer} to get the unscaled energy of $Q_3$ satisfies $-E_3(t_j)\geq \frac{1}{2}(\lambda_0-\tdelta)^{j/2}$ where $\tdelta\to 0$ as $\delta\to 0,\mu\to 0.$  For the local map part in the rescaled system, by part (c) of Lemma \ref{LmAdm}, $Q_3$ and $Q_4$ stay away from collision. By the continuity of the flow 
there is an upper bound $\tau$ of the flow time defining the local map for those initial values satisfying the assumption of Lemma \ref{LmLMC0}. Therefore without doing the rescalings, during the $j$-th trip the time spent during the local map part is bounded from above by $\tau/(\lb_0-\tilde\dt)^{3j/4}$ using \eqref{EqRenorm}.
For the global map part, we note that,  by \eqref{eq: hamloc}, the velocity of $Q_4$ during the trip $j$ is  $|v_4(t_j)|> \sqrt{2|E_3(t_j)|}\geq (\lambda_0-\tdelta)^{j/4}.$ 
According to the definition of the renormalization $\cR$, the rescaled distance between $Q_1$ and $Q_2$ is 
$\chi_j=|2E_3(t_j)|\chi_0$, where $\chi_0=|Q_1-Q_2|$ is the distance in the system without rescalings, and using part (c) of Lemma \ref{LmGMC0}, we have that without rescaling the time defining the global map during the $j$-th trip
is less than
$$\chi_j/|2E_3(j)|^{3/2}\leq \mathrm{const.}\chi_0  (\lambda_0-\tdelta)^{-j/4}.$$
 Therefore combining the above analysis for the local and global maps, we have $$|t_{j+1}-t_j|\leq \mathrm{const.}\chi_0 (\lambda_0-\tdelta)^{-j/4}$$
and so $t_*=\lim_{j\to\infty} t_j<\infty$ as needed. It is also clear from the estimate of $-E_3(t_j)$ and $|v_4(t_j)|$ that $\limsup_{t\to t^*}|v_i(t)|=\limsup_{t\to t^*}|\dot Q_i(t)|=\infty$, $i=3,4$. 

It remains to show that for each $j$ 
we can find a component of $\cP(S_{2j})$ inside $U_2(\delta)$  and
a component of $(\cR\circ \cP(S_{2j+1}))$ inside $U_1(\delta).$  

We proceed inductively. So we assume that the statement holds for $j'<j$ and that there exist $(\he_{3,j}, \hg_{3,j})$
such that on $S_{2j}$ we have
\begin{equation}
  \label{OscS2j}
|e_3-\he_{3,j}|\leq \eps, \quad |g_3-\hg_{3,j}|\leq \eps.   
\end{equation}
Note that due to rescaling defined in subsection \ref{SSAdm} we have that on $S_{2j}$
$$ \left|E_3-\frac{1}{2}\right|=O(\mu). $$
Since $S_{2j}$ is admissible it is a graph of a map $\cS_{2j}: \mathcal C_1(\dt)\to \R^4\times \T^2.$
Let
\begin{equation}
\label{S-2j+1}  
S_{2j+1}=\cP(\cS_{2j}(V(e_4'(\he_{3,j}, \hg_{3,j})))).
\end{equation}
We claim that $S_{2j+1}$ is a large admissible surface in $U_2(\dt).$
Indeed, by Lemma \ref{LmGMC0}(b) $\theta_4^-=O(\mu)$ on $S_{2j+1}.$ Also $e_4$ on $S_{2j+1}$ satisfies
$|e_4-e_4^{**}|\leq K K' \dt$
since $\cQ_1$ maps $V(e_4'(\he_{3,j}, \hg_{3,j}))$ onto $\mathcal C_2(\dt).$ 
Therefore we have the required control on the orbit parameters of $Q_4.$

Next, Lemmas \ref{LmLMC0} and \ref{LmGMC0} show that on $S_{2j+1}$ we have
$$ |e_3-e_3^{**}|\leq K\eps,\quad
|g_3-g_3^{**}|\leq K \eps \text{ and }
|E_3-E_3^{**}|\leq K\eps. $$
Thus $S_{2j+1}\subset U_2(\dt)$ and by Lemma \ref{LmKone}, $S_{2j+1}$ is admissible. 
In fact, it is a large admissible surface due to Lemma \ref{LmAdm}(a).

In addition, since $S_{2j+1}\subset U_2(\dt)$ it follows that $\cP: S_{2j}\to S_{2j+1}$ is strongly expanding.
We claim that this implies that the oscillations of $e_3$ and $g_3$ of $S_{2j+1}$ are less than $\eps$ if
$\mu$ is small enough. Namely, by Lemma \ref{LmKone}(b) the preimage of $S_{2j+1}$ has size $O(1/\chi).$
Hence
$e_3$ and $g_3$ have oscillations of size $O(1/\chi)$ on
$\cS_{2j}V(e_4'(\he_{3,j}, \hg_{3,j}))$ while Lemmas \ref{LmLMC0} and \ref{LmGMC0} show that the oscillations
do not increase much after application of local and global maps. Thus there are numbers $\te_{3,j}$ and $\tg_{3,j}$ such that on $S_{2j+1}$
$$ |e_3-\te_{3,j}|\leq \eps, \quad |g_3-\tg_{3,j}|\leq \eps . $$

Since $S_{2j+1}$ is admissible, it is a graph of a map $\cS_{2j+1}:\mathcal C_2(\dt)\to \R^4\times \T^2.$ Let
\begin{equation}
\label{S-2j+2}  
S_{2j+2}=\cR\circ \cP(\cS_{2j+1}(V(e_4''(\he_{3, j}, \hg_{3,j})))).
\end{equation}
The same argument as for $S_{2j+1}$ shows that $S_{2j+2}$ is a large admissible surface in $U_1(\dt)$ and that 
\eqref{OscS2j} holds on $S_{2j+2}$ (with $j$ replaced by $j+1$). The only caveat is that the surfaces $S_{2j}$ are not smooth but
only piecewise smooth since the rescaling map $\cR$ is discontinuous. 
However we can use the freedom to choose the appropriate partition in the definition of $\cR$ to ensure that
$\cR$ is continuous on the preimage of $V(e_4'(\he_{3,j}, \hg_{3,j}))$ so that $\cS_{2j} V(e_4'(\he_{3,j}, \hg_{3,j}))$ is a smooth
surface.

This completes the construction of a singular orbit.

\begin{Rk}
In fact we do not need to use exactly $e'(\he_{3,j}, \hg_{3,j})$ and $e''(\he_{3,j}, \hg_{3,j})$ in \eqref{S-2j+1} and \eqref{S-2j+2}.
Namely any $V(e_4^\dagger)$ and $V(e_4^\ddagger)$ would do provided that
$$\left|e_4^\dagger-e_4'(\he_{3,j}, \hg_{3,j})\right|<\eps, \quad \left|e_4^\ddagger-e_4''(\he_{3,j}, \hg_{3,j})\right|<\eps. $$
Different choices of $e_4^\dagger$ and $e_4^\ddagger$ allow us obtain different orbits. Since such
freedom exists at each step of our construction we have a Cantor set of singular orbits with a given
symbolic sequence $\bomega.$
\end{Rk}

\section{Hyperbolicity of the Poincar\'{e} map}\label{section: hyperbolic}

\subsection{Construction of invariant cones}\label{SSKone}
Here we derive Lemma \ref{LmKone}, \ref{LmAdm} and \ref{LmChooseAM} dealing with the
asymptotics of the derivative of local and global maps.

\blm
\label{LmDerLoc} Fix $j\in \{1,2\}$ meaning the first or second collision.\\
$(a)$ Let $\tilde \theta$ be a small constant. Consider  
$\bx\in U_j(\delta)$ satisfying \begin{itemize}
\item[(1)] the orbit with initial value $\bx$ passes through a $\dt$ neighborhood of the $j$-th Gerver's collision point. 
\item[(2)] $|\theta_4^-(\bx)|\leq C_1\mu$ where $C_1$ is as in Lemma \ref{LmLMC0}.
\item[(3)] $\by=\Loc(\bx)\in \{x_4=-2,\ \dot x_4<0\}$ satisfies $|\bar{\theta}_4^+(\by)-\pi|\leq \tilde\theta.$
\end{itemize} Then there exist continuous functions $\mathbf u_j(\bx,\bar{\theta}_4^+),$ $\lin_j(\bx)$ and $B_j(\bx,\bar{\theta}_4^+)$ such that
$$ d\Loc(\bx)=\dfrac{1}{\mu}( \mathbf u_j(\bx,\bar{\theta}_4^+)+o(1))\otimes(\lin_j(\bx)+o(1)) +B_j(\bx,\bar{\theta}_4^+)+o(1),\quad \mathrm{as} \ 1/\chi\ll\mu\to0 . $$
$(b)$ Moreover there exist
a linear functional $\hlin_j$, a vector $\hat{\mathbf{u}}_j$ and a matrix $\hat B_j$ with bounded norms, such that if we take further limits $\delta\to0 $ and $\tilde\theta\to 0$, we have
$$ \lin_j(\bx)\to\hlin_{j},  \quad \mathbf{u}_j(\bx,\bar{\theta}_4^+)\to\hat{\mathbf{u}}_{j}, \quad B_j(\bx,\bar{\theta}_4^+)\to 
\hat B_j .$$
\elm
This lemma is proven in Section~\ref{section: local}.
\blm
\label{LmDerGlob}
Fix $j\in \{1,2\}$ meaning the first or second collision.\\
Let $\bx \in \{x_4=-2,\ \dot x_4<0\}$ and $\by=\Glob(\bx) \in \{x_4=-2,\ \dot x_4>0\}$ be such that
$|y_4(\bx)|\leq C_2,$ $|y_4(\by)|\leq C_2$ where $C_2$ is as in Lemma \ref{LmGMC0}. Then
\begin{itemize}
\item[(a)]
there exist continuous linear functionals $\brlin_j(\bx)$ and $\brrlin_j(\bx)$ and vectorfields $\brv_j(\by)$ and $\brrv_j(\by),$ such that as $1/\chi\ll\mu\to0 $ 
$$ d\Glob(\bx)=\chi^2\left (\brv_j(\by)+o(1)\right) \otimes\left(\brlin_j(\bx)+o(1)\right)+\chi \left(\brrv_j(\by)+o(1)\right) \otimes\left(\brrlin_j(\bx)+o(1)\right)+O(\mu\chi). $$
\item[(b)]If $\bx\in U_j(\dt)$ satisfies $\Glob\circ\Loc(\bx)\in U_{3-j}(\dt)$ for $j=1$ or $\cR\circ \Glob\circ\Loc( \bx)\in U_{3-j}(\dt)$ for $j=2,$
and the orbit with initial value $\bx$ passes through a $\dt$ neighborhood of the $j$-th Gerver's collision point,
then there exist vector $w_j$ and linear functionals $\hat\brlin_j, \hat\brrlin_j$
such that for $\delta\to 0$, we have
$$ \brlin_j(\bx)\to\hat\brlin_j, \quad \brrlin_j(\bx)\to\hat\brrlin_j,\quad  \Span(\brv_j(\by), \brrv_j(\by))\to \Span(w_j, \tw).$$
\item[(c)]Finally if we define in Delaunay coordinates  \begin{equation}
\label{eq: ul}
\begin{split}
\hat\brlin&=\left(\frac{ G_4/ L_4}{ L_4^2+ G_4^2}, 0,0,0,-\dfrac{1}{ L_4^2+ G_4^2} , -\frac{1}{ L_4}\right),\quad  \hat\brrlin=(1,0,0,0,0,0), \\
w&=\left(0,0,0,0,1,\frac{L_4}{ L_4^2+ G_4^2}\right)^T,
\qquad  \tilde{w}=(0,1,0,0,0,0)^T, 
\end{split}
\end{equation}
then $\hat\brlin_j$ and $\hat\brrlin_j$ are obtained from $\hat\brlin$ and $\hat\brrlin$ respectively by evaluating $G_4,L_4$ at Gerver's collision point immediately after the $j$-th collision, and $w_j$ is obtained from $w$ by evaluating $G_4,L_4$ at Gerver's collision point immediately before the $(3-j)$-th collision.
\end{itemize}
\elm
\begin{Rk}
We remark that the $w_{j},\ j=1,2$ in Definition \ref{DefKone} is the same as the $w_j$ here, but written in different coordinates. 
\end{Rk}

This lemma is proven in Section~\ref{SSExp}.
\blm
\label{LmNonDeg}
The following non degeneracy conditions are satisfied for $E_3^{*}=-1/2,e_3^*=1/2,g_3^*=\pi/2$.
\begin{itemize}
\item[(a1)] $ \Span(\hat{\mathbf{u}}_{1}, B(\hlin_{1}(\tw) d\cR w_2-\hlin_{1}(d\cR w_2)\tw))$ is transversal to $\Ker(\hat\brlin_{1})\cap \Ker(\hat\brrlin_{1}).$
\item[(a2)] $de_4(\Span(d\cR w_2, d\cR \tw))\neq 0.$
\item[(b1)]
$ \Span(\hat{\mathbf{u}}_{2}, B(\hlin_{2}(\tw) w_1-\hlin_{2}(w_1)\tw))$ is transversal to $\Ker(\hat\brlin_{2})\cap \Ker(\hat\brrlin_{2}).$
\item[(b2)] $de_4(w_1)\neq 0.$
\end{itemize}
\elm
This lemma is proven in Section~\ref{SSTrans}.
\begin{proof}[Proof of Lemma \ref{LmKone}]
Consider for example the case where $\bx\in U_2(\delta).$
We claim that if $\delta, \mu$ are small enough then
$d\Loc(\Span(w_1, \tw))$ is transversal to $\Ker \brlin_{2}\cap \Ker \brrlin_{2}.$ Indeed take
$\Gamma$ such that $\lin (\Gamma)=0.$ If $\Gamma=a w_1+\ta \tw$ then
$a \lin_2 (w_1)+\ta \lin_2 (\tw)=0.$
It follows that the direction of $\Gamma$ is close to the direction of
$\hGamma=\hlin_2 (\tw) w_1-\hlin_2(w_1) \tw. $ Next take $\tGamma=bw+\tb\tw$ where
$b \lin_2 (w_1)+\tb \lin_2 (\tw)\neq 0.$
Then
the direction of  $d\Loc \tGamma $ is close to $\hat{\mathbf{u}}_{2}$ and the direction of
$d\Loc(\Gamma)$ is close to $B(\hGamma)$ so our claim follows.

Thus for any plane $\Pi$ close to $\Span(w_1, \tw)$ we have that
$d\Loc(\Pi)$ is transversal to $\Ker \brlin_{2}\cap \Ker \brrlin_{2}.$
Take any $Y\in \cK_2.$ Then either
$Y$ and $w_1$ are linearly independent  or $Y$ and $\tw$ are linearly independent.
Hence $d\Loc(\Span(Y, w_1))$ or $d\Loc(\Span(Y, \tw))$ is transversal to $\Ker \brlin_{2}\cap \Ker \brrlin_{2}.$
Accordingly either $\brlin_{2}(d\Loc(Y))\neq 0$
or $\brrlin_{2}(d\Loc(Y))\neq 0.$ If $\brlin_{2}(d\Loc(Y))\neq 0$ then the direction of $d(\Glob\circ \Loc)(Y)$
is close to $\brv.$ If $\brlin_{2}(d\Loc(Y))=0$ then the direction of $d(\Glob\circ \Loc)(Y)$
is close to $\brrv.$ In either case $d(\cR \Glob\circ\Loc)(Y)\in \cK_{1}$ and
$\Vert d(\Glob\circ\Loc)(Y)\Vert \geq c\chi \Vert Y\Vert .$
This completes the proof in the case $\bx\in U_2(\delta).$ The case where $\bx\in U_1(\delta)$ is similar.
\end{proof}


To prove Lemma \ref{LmChooseAM} we need two auxiliary results.
\begin{sublemma}
\label{LmHitTarget}
In the notation and setting of part $(a)$ of Lemma \ref{LmChooseAM}, given $\te_4$ there exists $\tilde{\ell}_3$ such that $\cP(S_1(\te_4, \tilde{\ell}_3))\in U_2(\delta).$ There is a corresponding statement to part $(b)$ of Lemma \ref{LmChooseAM}. 
\end{sublemma}
The proof of this sublemma is postponed to Section~\ref{SSAngMom}.
\begin{sublemma}
\label{LmCover}
Let $\cF$ be a map on $\R^2$ which fixes the origin and such that if
$|\cF(z)|<R$ then $\Vert d\cF(X)\Vert \geq \brchi \Vert X\Vert .$ Then for each $a$ such that
$|a|<R$ there exists $z$ such that $|z|<R/\brchi$ and $\cF(z)=a.$
\end{sublemma}

\begin{proof}
Without the loss of generality we may assume that $a=(r,0).$ Let $V(z)$ be the direction field defined by the condition
that the direction of $d\cF(V(z))$ is parallel to $(1,0).$ Let $\gamma(t)$ be the integral curve of $V$
passing through the origin and parameterized
by the arclength. Then $\cF(\gamma(t))$ has form $(\sigma(t), 0)$ where $\sigma(0)=0$ and
$|\dot{\sigma}(t)|\geq\brchi$ as long as $|\sigma|<R.$ Now the statement follows easily.
\end{proof}

\begin{proof}[Proof of Lemma \ref{LmChooseAM}]$(a)\ $
We claim that it suffices to show that for each
$(\bre_4, \bar{\ell}_3)$ such that $|\bre_4-e_4^{**}|<\sqrt{\delta} $ there exist $(\he_4, \hat{\ell}_3)$ such that
\begin{equation}
\label{EqChooseAMPh}
\mathcal Q_1(\he_4, \hat{\ell}_3)=(\bre_4, \bar{\ell}_3).
\end{equation}
Indeed in that case Sublemma \ref{KeepDirection} from Section \ref{SSPrLmPos} says that the outgoing asymptote is almost horizontal.
Therefore by Lemma \ref{LmLMC0}
our orbit has $(E_3, e_3, g_3)$ close to
$\Ger_{\te_4, 2, 4}(E_3(\he_4, \hat{\ell}_3), e_3(\he_4, \hat{\ell}_3), g_3(\he_4, \hat{\ell}_3)).$
Next Lemma \ref{LmGMC0} shows that after the application of $\Glob$, $(E_3, e_3, g_3)$ change little and
$\theta_4^-$ becomes $O(\mu)$ so that $\cP(S_1(\he_4, \hat{\ell}_3))\in U_2(\delta).$

We will now prove \eqref{EqChooseAMPh}. Due to Lemma \ref{LmKone}
we can apply Sublemma \ref{LmCover}
to the covering map $\tilde{\mathcal Q}_1: \R^2\to\R^2$ with $\brchi=c\chi$
obtaining \eqref{EqChooseAMPh}. This completes the proof of part (a).

Part (b) is similar to part (a).

Part (c) follows from Lemma~\ref{Lm: landau} proven in Section \ref{ScC0Loc}.
\end{proof}
\subsection{Expanding directions of the global map}
\label{SSExp}

Estimating the derivative of the global map is the longest part of the paper. It occupies Sections \ref{subsection: plan}--\ref{section: switch foci}.

It will be convenient to use the Delaunay coordinates $(L_3,\ell_3,G_3,g_3)$ for $Q_3$ and $(G_4, g_4)$ for $Q_4.$
Delaunay coordinates are action-angle coordinates for the Kepler problem. We collect some facts about the Delaunay
coordinates in Appendix \ref{section: appendix}.

We divide the plane into several pieces by lines $x_4=-2$ and $x_4=-\frac{\chi}{2}$. Those lines cut the orbit of $Q_4$ into 4 pieces:
\begin{itemize}
\item $\left\{x_4=-2,\ \dot{x}_4<0\right\}\to \left\{x_4=-\frac{\chi}{2},\ \dot{x}_4<0\right\}$. We call this piece $(I)$.
\item $\left\{x_4=-\frac{\chi}{2},\ \dot{x}_4<0\right\}\to \left\{x_4=-\frac{\chi}{2},\ \dot{x}_4>0\right\}$ turning around $Q_1$. We call it $(III)$.
\item $\left\{x_4=-\frac{\chi}{2},\ \dot{x}_4>0\right\}\to \left\{x_4=-2,\ \dot{x}_4>0\right\}$. We call it $(V)$
\item $\left\{x_4=-2,\ \dot{x}_4>0\right\} \to \left\{x_4=-2,\ \dot{x}_4<0\right\}$ turning around $Q_2$.
\end{itemize}
We composition of the first three pieces constitutes the global map. The last piece defines the local map. See Fig 3.
Notice that when we define $\cR$ in Section \ref{SSAdm}, after the second collision in Gerver's construction, the global map sends $\{x_4=-2,\ \dot x_4<0\}$ to $\{x_4=-2/\lb,\ \dot x_4>0\}$. Then $\cR$ sends $\{x_4=-2/\lb,\ \dot x_4>0\}$ to $\{x_4=-2,\ \dot x_4>0\}$ before applying local map. So without leading to confusion, when we are talking about sections after the second collision, we always talk about $\cR\circ\Glob$ so that the section $\{x_4=-2,\dot x_4<0\}$ is sent to $\{x_4=-2,\dot x_4>0\}$.
\begin{figure}[ht]
\begin{center}
\includegraphics[width=0.7\textwidth]{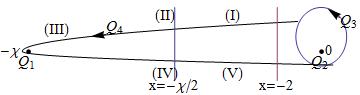}
\caption{Poincar\'e sections}
\end{center}
\end{figure}

The line $x_4=-\frac{\chi}{2}$ is convenient because
if $Q_4$ is moving to the right of the line $x_4=-\frac{\chi}{2}$, its motion can be treated as a hyperbolic motion focused at $Q_2$ with perturbation
caused by $Q_1$ and $Q_3$. If $Q_4$ is moving to the left of this line, its motion can be treated as a hyperbolic motion focused at $Q_1$ perturbed by $Q_2$ and $Q_3$.

Since we use different guiding centers to the left and right of the line of $x_4=-\frac{\chi}{2}$ we will need to change variables when $Q_4$ hits this line.
This will give rise to two more matrices for the derivative of the global map: $(II)$ will correspond to the change of coordinates
from right to left and $(IV)$ will correspond for the change of coordinates from left to right. Thus
$d\Glob=(V)(IV)(III)(II)(I).$ In turn, each of the matrices $(II)$ and $(IV)$ will be products of three matrices corresponding to
changing one variable at a time. Thus we will have $(II)=[(iii)(ii)](i)$ and $(IV)=(iii')[(ii')(i')]$.

The asymptotics of the above mentioned matrices is presented in the two propositions below.

To refer to a certain subblock of a matrix $(\sharp)$, we use the following convention:
\[(\sharp)=\left[\begin{array}{c|c}
(\sharp)_{33}&(\sharp)_{34}\\
\hline
(\sharp)_{43}&(\sharp)_{44}
\end{array}\right].\]
Thus $(\sharp)_{33}$ is a $4\times 4$ matrix and
$(\sharp)_{44}$ is a $2\times 2$ matrix.
To refer to the $(i,j)-th$ entry of a matrix $(\sharp)$ (in the Delaunay coordinates mentioned above) we use $(\sharp)(i,j)$.
For example, $(I)(1,3)$ means the derivative of $L_3$ with respect to $G_3$ when the orbit moves between
sections $\{x_4=-2\}$ and $\left\{x_4=-\frac{\chi}{2}\right\}.$



\begin{Prop}
\label{Thm: matrices}
Under the assumptions of Lemma \ref{LmDerGlob}
the matrices introduced above satisfy the following estimates.
\begin{equation}(I)=\mathrm{Id}+\left[
\begin{array}{cc|c}
O(\mu)& O(\mu)_{1\times 3}&O(\mu)_{1\times 2}\\
O(\chi)& O(\mu\chi)_{1\times 3}& O(\mu\chi)_{1\times 2}\\
O(\mu)_{2\times 1}&O(\mu)_{2\times 3}&O(\mu)_{2\times 2}\\
\hline
O(1)_{2\times 1}&O(\mu)_{2\times 3}&O(1)_{2\times 2}\\
\end{array}
\right],\nonumber
\end{equation}
\begin{equation}
{\small\begin{aligned}
&(i)=\left[\begin{array}{cc|cc}
1& 0&0&0\\
0&\mathrm{Id}_{3}&0_{3\times 1}&0_{3\times 1}\\
\hline
0&0_{1\times 3}&1&0\\
\frac{\tG_{4R}/k_R \tL_3}{k_R^2\tL_{3}^2+\tG_{4R}^2}+O(\frac{1}{\chi})&O(\frac{1}{\chi^2})_{1\times 3}& -\frac{1}{k_R^2 \tL_{3}^2+\tG_{4R}^2}+O(\frac{1}{\chi})& -\frac{1}{k_R\tilde L_3}+O(\frac{1}{\chi})
\end{array}
\right]
\end{aligned}
},\nonumber
\end{equation}
\begin{equation}
\begin{aligned}
&[(iii)(ii)]=\left[\begin{array}{cc|cc}
1& 0&0&0\\
0&\mathrm{Id}_{3}&0_{3\times 1}&0_{3\times 1}\\
\hline
O(1/\chi)&O(1/\chi^3)_{1\times 3} &1&-\chi\\
O(1/\chi)&O(1/\chi^3)_{1\times 3}& \frac{1}{ \tL_3}+O(1/\chi)& -\frac{\chi}{\tL_3} +O(1)
\end{array}
\right]
\end{aligned},\nonumber
\end{equation}

\begin{equation}
(III)=\mathrm{Id}+\left[
\begin{array}{cc|c}
O(1/\chi)& O(1/\chi^2)_{1\times 3}&O(\mu/\chi)_{1\times 2}\\
O(\chi)& O(1/\chi)_{1\times 3}& O(1)_{1\times 2}\\
O(1/\chi)_{2\times 1}&O(1/\chi^2)_{2\times 3}&O(\mu/\chi)_{2\times 2}\\
\hline
O(\mu)_{2\times 1}&O(\mu/\chi)_{2\times 3}&O(1)_{2\times 2}\\
\end{array}
\right],\nonumber
\end{equation}
\begin{equation}
\begin{aligned}
&[(ii')(i')]=\left[\begin{array}{cc|cc}
1& 0&0&0\\
0&\mathrm{id}_{3}&0_{3\times 1}&0_{3\times 1}\\
\hline
O(1)&O(1/\chi^2)_{1\times 3}&\frac{\chi}{\hL_3^2}+O(1)&\frac{\chi}{\hL_3}+O(1)\\
O(1/\chi)&O(1/\chi^3)_{1\times 3}& \frac{1}{\hL_3^2}+O(1/\chi)& \frac{1}{\hL_3}+O(1/\chi)
\end{array}
\right]
\end{aligned},\nonumber
\end{equation}
\begin{equation}
{\small
\begin{aligned}
&(iii')=\left[\begin{array}{cc|cc}
1& 0&0&0\\
0&\mathrm{Id}_{3}&0_{3\times 1}&0_{3\times 1}\\
\hline
0&0_{1\times 3}&1&0\\
-\frac{\hG_{4R}/(k_R)}{(k_R^2\hL_3^2+G_{4R}^2)}+O(\frac{1}{\chi})&O(\frac{1}{\chi^2})_{1\times 3}& \frac{k_R\hL_3}{k_R^2 \hL_3^2+\hG_{4R}^2}+O(\frac{1}{\chi})& k_R \hL_3+O(\frac{1}{\chi})
\end{array}
\right]
\end{aligned}
},\nonumber
\end{equation}
\begin{equation}
(V)=\mathrm{Id}+\left[
\begin{array}{cc|c}
O(\mu\chi)& O(\mu)_{1\times 3}&O(\mu)_{1\times 2}\\\
O(\chi)& O(\mu)_{1\times 3}& O(1)_{1\times 2}\\
O(\mu\chi)_{2\times 1}&O(\mu)_{2\times 3}&O(\mu)_{2\times 2}\\
\hline
O(\mu\chi)_{2\times 1}&O(\mu)_{1\times 3}&O(1)_{2\times 2}\\
\end{array}
\right].\nonumber
\end{equation}
where $k_R=1+\mu,$ $\tilde L_3,\tilde G_4$ are the initial values of $\Glob$ of $L_3,G_4$ and $\hat L_3,\hat G_4$ are the final values of $\Glob$ of $L_3,G_4$. Moreover, the matrix of the renormalization map $\mathcal R$ has the form $\mathrm{diag}\{\sqrt{\lb},1, -\sqrt{\lb},-1,-\sqrt{\lb},-1\}$, where the constant $\lb$ is the dilation rate defined in Section \ref{SSAdm} and the ``$-$" appears due to the reflection.
\end{Prop}

\begin{Prop}
\label{PrExact}
\begin{itemize}
\item[(a)]
The $O(\chi)$ entries in the matrices $(I),(III),(V)$ are $c_I\chi, c_{III}\chi, c_V\chi$, where $c_I,c_{III},c_V\neq 0$ and have the same sign. 

\item[(b)] The $O(1)$ blocks in Proposition \ref{Thm: matrices}
can be written as a continuous function
of $\bx$ and $\by$ plus an error which vanishes in the limit $\mu\to 0, \chi\to\infty.$
Moreover the $O(1)$ blocks have the following limits for orbits of interest.
\begin{equation*}
\begin{aligned}
&(I)_{44}=
\left[\begin{array}{cc}
1-\frac{\tilde L^2_4}{2 (\tilde L_4^2+\tilde G_4^2)} & -\frac{\tilde L_4}{2}\\
\frac{\tilde L_4^3}{2 (\tilde L_4^2+\tilde G_4^2)^2}& 1+\frac{\tilde L_4^2}{ 2(\tilde L_4^2+\tilde G_4^2)} \\
\end{array}
\right], 
\quad (III)_{44}=\left[\begin{array}{cc}
\frac{1}{2} & -\frac{ L_4}{2}\\
 \frac{3}{2 L_4}& \frac{1}{2} \\
\end{array}
\right],\\ 
& (V)_{44}=\left[\begin{array}{cc}
1+\frac{1/2\hat L^2_4}{\hat L_4^2+\hat G_4^2} &-1/2\hat L_4\\
 \frac{1/2\hat L_4^3}{(\hat L_4^2+\hat G_4^2)^2}& 1-\frac{1/2\hat L_4^2}{\hat L_4^2+\hat G_4^2} \\
\end{array}
\right]. \\
&\mathrm{In\ addition\ for\ map\ (I)\ we\ have}\\
&((I)(5,1),(I)(6,1))^T=
\left( \frac{\tilde G_4\tilde L_4}{2(\tilde L_4^2+\tilde G^2_4)},
-\frac{\tilde G_4\tilde L_4^2}{2(\tilde L_4^2+\tilde G^2_4)^2}\right)^T. \\
\end{aligned}\label{eq: twist}
\end{equation*}
where $tilde,\ hat$ have the same meanings as in the previous proposition.
\end{itemize}
\end{Prop}


The estimates of $(I),(III),(V)$ from Proposition \ref{Thm: matrices} are proven in Sections \ref{section: equation}--\ref{section: boundary}.
The estimates of $(II),(IV)$ are given in Section~\ref{section: switch foci}.
Proposition \ref{PrExact}
is proven in Section~\ref{SSSol}.
Now we prove Lemma~\ref{LmDerGlob} based on the Proposition \ref{PrExact}.
\begin{proof}[Proof of Lemma~\ref{LmDerGlob}.]

For the matrices $(I),(III),(V)$, we separate the $(2,1)$ entry from the matrices to get the following decompositions into a tensor part and a remainder. 
\begin{equation}\label{EqDecom135}
\begin{aligned}
(I)=c_I\chi \bar u\otimes \bar l+R_I,\quad (III)&=c_{III}\chi \bar u\otimes \bar l+R_{III},\quad (V)=c_V\chi \bar u\otimes \bar l+R_{III}\\\end{aligned}
\end{equation}
where $\bar u=(0,1,0,0,0,0)$, $\bar l=(1,0,0,0,0,0)$ and the tensor term picks out the $O(\chi)$ entry in each matrix. 
For the matrices $[(iii)(ii)]$ and $[(ii')(i')]$, we separate the leading terms of the $44$ blocks to get the following decompositions
$$[(iii)(ii)]= \chi u_{II}\otimes l_{II}+R_{II},\quad [(ii')(i')]=\chi u_{IV}\otimes l_{IV}+R_{IV}$$
where $$u_{II}=\left(0,0,0,0,1,\frac{1}{L_3}\right),\quad l_{II}= \left(0,0,0,0,\frac{1}{\chi},-1\right)$$$$u_{IV}=\left(0,0,0,0,1,\frac{1}{\chi}\right),\quad l_{IV}= \left(0,0,0,0,\frac{1}{L_3^2},\frac{1}{L_3}\right).$$ 
Notice $l_{IV}\cdot \bar u=\bar l\cdot u_{II}=0.$
Multiplying $(ii')(i')(III)(iii)(ii)$, we get 
\begin{equation*}
\begin{aligned}
&(ii')(i')(III)(iii)(ii)=(\chi u_{IV}\otimes l_{IV}+R_{IV})(III)(\chi u_{II}\otimes l_{II}+R_{II})\\
&=\chi^2 u_{IV}\otimes l_{IV}(III) u_{II}\otimes l_{II}+\chi R_{IV}(III) u_{II}\otimes l_{II}+\chi u_{IV}\otimes l_{IV}(III)R_{II}+R_{IV}(III)R_{II}\\
&=\chi^2 u_{IV}\otimes l_{IV}R_{III} u_{II}\otimes l_{II}+\chi R_{IV}R_{III} u_{II}\otimes l_{II}+\chi u_{IV}\otimes l_{IV}R_{III}R_{II}+R_{IV}(III)R_{II}\\
\end{aligned}
\end{equation*}
We define $c=l_{IV}R_{III} u_{II}$ and $v=R_{IV}R_{III} u_{II}, v'= l_{IV}R_{III}R_{II}.$

Continuing the computation we get 
\begin{equation}\label{EqMultiply432}
\begin{aligned}
&=c\chi^2  u_{IV}\otimes  l_{II}+\chi v\otimes l_{II}+\chi u_{IV}\otimes v'+R_{IV}(III)R_{II}\\
&=c\chi^2 \left(u_{IV}+\frac{v}{c\chi}\right)\otimes \left(l_{II}+\frac{v'}{c\chi}\right)-\frac{1}{c}v\otimes v'+c_{III}\chi R_{IV} \bar u\otimes \bar l R_{II}+R_{IV}R_{III}R_{II}
\end{aligned}
\end{equation}
 where 
\begin{enumerate}
\item[(a)]
\begin{equation}\nonumber
\begin{aligned}
c&=\left(0_{1\times 4},\frac{1}{L_3^2},\frac{1}{L_3}\right)\left(\left[\begin{array}{c|c}
\mathrm{id}_4&0\\
\hline
0&(III)_{44}\end{array}\right]+\left[\begin{array}{c|c}
O(1/\chi)&O(1)\\
\hline
O(\mu)&O(\mu)\end{array}\right] \right)\left(0_{1\times 4},1,\frac{1}{L_3}\right)^T\\
&\to \left(\frac{1}{L_3^2},\frac{1}{L_3}\right)(III)_{44}\left(1,\frac{1}{L_3}\right)^T=\frac{2}{L_3^2}\neq 0.
\end{aligned}
\end{equation}
\item[(b)] $v=O\left(\frac{\mu}{\chi},1, \frac{\mu}{\chi},\frac{\mu}{\chi},1,\frac{1}{\chi}\right)^T$ and $v'=O\left(\frac{1}{\chi},\frac{\mu}{\chi},\frac{\mu}{\chi},\frac{\mu}{\chi},\frac{1}{\chi},1\right)$.
\item[(c)] $R_{IV} \bar u=\bar u+O(0_{1\times 4},1/\chi^2,1/\chi^3)$ and $\bar l R_{II}=\bar l.$
\item[(d)]
The remainder $R_{IV}R_{III}R_{II}$ is explicitly computed\\$  =\left[\begin{array}{c|c}\mathrm{id}_4&0\\
\hline
0&0
\end{array}\right]
+O\left[\begin{array}{cc|cc}
\frac{1}{\chi}&(\frac{1}{\chi^2})_{1\times 3}&\frac{\mu}{\chi^2}&\frac{\mu}{\chi}\\
0&(\frac{1}{\chi})_{1\times 3}&\frac{1}{\chi}&1\\
(\frac{1}{\chi})_{2\times 1}&(\frac{1}{\chi^2})_{2\times 3}&(\frac{\mu}{\chi^2})_{2\times 1}&(\frac{\mu}{\chi})_{2\times 1}\\
\hline
1&(\frac{\mu}{\chi})_{1\times 3}&\frac{1}{\chi}&1\\
\frac{1}{\chi^2}&(\frac{\mu}{\chi^2})_{1\times 3}&\frac{1}{\chi^2}&\frac{1}{\chi}
\end{array}\right]$. 
\end{enumerate}
We next multiply $(i)(I)$ from the right and $(V)(iii')$ from the left to $(ii')(i')(III)(iii)(ii)$ to get
$$
(V)(IV)(III)(II)(I)=(c_I+c_{III}+c_V)\chi \bar u\otimes \bar l+c\chi^2\bar{\bar u}\otimes \bar{\bar l}+O(\mu\chi)
$$
where \begin{enumerate}
\item 
we have the following limit using Proposition \ref{PrExact}
$$\bar{\bar u}=(V)(iii')\left(u_{IV}+\frac{v}{c\chi}\right)\to w+\mathrm{const.} \bar u,\quad \bar{\bar l}=\left(l_{II}+\frac{v'}{c\chi}\right) (i)(I)\to \hat\brlin$$ as $1/\chi\ll\mu\to 0$. In fact $\bar{\bar u}$ is essentially the fifth column of $(iii')$ and $\bar{\bar l}$ is essentially the sixth row of $(i)$. 
\item 
using item (b) above, we have $(V)(iii')v\otimes v'(i)(I)=O(1)$. The estimate $(V)(iii')v=O(1)$ essentially picks out the second, fifth and sixth columns of $(V)$, the estimate $v'(i)(I)=O(1)$ essentially picks out the first, fifth and sixth rows of $(I)$, and the $O(\chi)$ entries in $(I)$ and $(V)$ are suppressed by the small entries of $v'$ (the second entry) and $v$ (the first entry) respectively.
\item 
using item (c) above, we have 
$$ (V)(iii')(c_{III}\chi R_{IV} \bar u\otimes \bar l R_{II} )(i)(I)=c_{III}\chi\bar u\otimes \bar l+O(1).$$
\item 
using the decomposition of $(I)$ and $(V)$ in \eqref{EqDecom135}, we can verify that $$c_V\chi \bar u\otimes \bar l(iii') R_{IV}R_{III}R_{II}(i) c_I\chi \bar u\otimes \bar l=(c_I+c_V)\chi \bar u\otimes \bar l+O(1).$$ Here the $O(1)$ estimate of the remainder comes from the $O(1/\chi^2)$ estimate of the $(1,2)$ entry of the matrix $R_{IV}R_{III}R_{II}$, which in turn comes from the same estimate of the $(1,2)$ entry of $(III).$
\item All the remaining terms in $(V)(iii')R_{IV}R_{III}R_{II}(i)(I)$ other than item (4) above are absorbed into 
$O(\mu\chi)$. (Note that the special structure of the matrices is important for the estimate. 
In particular, though the first column of $(V)$ and the second row of $(I)$ are large, the first row and second column in $R_{IV}R_{III}R_{II}$ are small.) \qedhere
\end{enumerate}
\end{proof}

\subsection{Checking transversality}
\label{SSTrans}
In Lemma \ref{LmDerLoc} and \ref{LmDerGlob} when we take limits $\tilde\theta,\dt,\mu,1/\chi\to0$,
the dynamics in the phase space reduces to Gerver's case.
The limiting vectors $\hat\lin_j,\hat{\bar\lin}_j,\hat{\brrlin}_j$ and $\hat{\mathbf u}_j,w, \tilde w$
can be computed explicitly and evaluated at Gerver's collision points.
In the following lemmas we consider Gerver's orbits with the choice of
$E_3^*=-\frac{1}{2}$ and $e_3^*=\frac{1}{2}$. All the other orbit parameters are
determined by $E_3^*,e_3^*$ as shown in Appendix \ref{subsection: numerics}. 

The $O(1/\mu)$ part of $d\mathbb L$ in Lemma~\ref{LmDerLoc} satisfies the following estimates. 
 \begin{Lm}
The asymptotics $\hat\lin_j$ and $\hat{\mathbf u}_j$ of the vectors $\lin_j, \mathbf{u}_j$ in the $O(1/\mu)$ part of the matrix $d\mathbb L$ satisfy the following:
\begin{itemize}
\item[(a)] \[\hat\lin_j\cdot \tw\neq 0,\quad \hat\lin_j\cdot w_{3-j}\neq 0,\quad  \hat\brlin_j\cdot\hat{\mathbf{u}}_j\neq 0,\] $j=1,2$ meaning the first or the second collision.
\item[(b)] If $Q_3$ and $Q_4$ switch roles after the collisions, the vectors $\hat{\mathbf{u}}_1$ and $\hat{\mathbf{u}}_2$ get a ``$-$" sign.
\label{Lm: local4}
\end{itemize}
\end{Lm}
To check the nondegeneracy condition, it is enough to know the following.
\begin{Lm} Let $\bx\in U_j(\delta)$ and $|\bar\theta_4^+-\pi|<\tilde\theta\ll 1$ be as in Lemma \ref{LmDerLoc}.
If we take the directional derivative at $\bx$ of the local map along a direction
$\Gamma\in span\{\brv_{3-j},\brrv_{3-j}\}\subset T_{\bx} U_j(\dt)$, such that $$\brlin_j\cdot (d\Loc \Gamma)=0,\ j=1,2,$$ then
$\lim_{1/\chi\ll\mu \to 0}\frac{\partial E_3^+}{\partial \Gamma}$
is a continuous function of both $\bx$ and $\bar\theta_4^+$, where $E_3^+$ $($respectively $\bar\theta_4^+$$)$ is the energy of $Q_3$ $($respectively outgoing asymptote of $Q_4)$ after the close encounter with $Q_4$.
If we take further limits $\dt\to 0$ and $\tilde\theta\to 0$, we have
$$\lim_{\dt,\tilde\theta\to 0}\;\;\;\lim_{1/\chi\ll\mu \to 0}\dfrac{\partial E_3^+}{\partial \Gamma}\neq 0.$$
\label{Lm: local3}
\end{Lm}
The proofs of the two lemmas are postponed to Section~\ref{section: local}. Now we can check the nondegeneracy condition.
\begin{proof}[Proof of Lemma~\ref{LmNonDeg}.]
We prove (b1) and (b2). The proofs of (a1) and (a2) are similar and are left to the reader.

To check (b2), $de_4$ we differentiate $e_4=\sqrt{1+(G_4/L_4)^2}$ to get
\[de_4=\dfrac{1}{e_4}\left(\dfrac{G_4}{L^2_4}dG_4-\dfrac{G^2_4}{L_4^3}dL_4\right).\] Thus \eqref{eq: ul} gives
$de_4 w=\frac{G_4}{L^2_4}\neq 0$ as claimed.

Next we check (b1) which is equivalent to the following condition
\begin{equation} \det\left(\begin{array}{rr} \hat\brlin_{2}(\hat{\mathbf{u}}_{2}) & \hat\brlin_{2}(\hat B_2\Gamma')
) \cr
                                            \hat\brrlin_{2}(\hat{\mathbf{u}}_{2}) & \hat\brrlin_{2}(\hat B_2 \Gamma')
                                           \end{array}\right)\neq 0. \label{eq: nondeg}\end{equation}
where $\Gamma'=\hat\lin_{2}(\tw)w_{1}-\hat\lin_{2}(w_{1})\tw.$ The vector $\Gamma'\neq 0$ due to part (a) of Lemma \ref{Lm: local4}.

Let $\Gamma$ be a vector satisfying $\hat\brlin_2\cdot (d\Loc\Gamma)=0$ and
chosen as follows.
$d\Loc \Gamma$ is a vector in $span\{\hat{\mathbf{u}}_i,\hat B_i\Gamma'_i\}$, so it
can be represented as $d\Loc \Gamma_i=b \hat{\mathbf{u}}_2+b'\hat B_2\Gamma'.$
Thus we can take $b=-\hat\brlin_2\cdot \hat B_2\Gamma'$ and $b'=\hat\brlin_2(\hat{\mathbf{u}}_2)$
to ensure that $d\Loc \Gamma_i\in Ker \hat\brlin_2$.
Note that we have $b'\neq 0$ by part (a) of Lemma \ref{Lm: local4}.
Hence
\[\det\left(\begin{array}{rr} \hat\brlin_2(\hat{\mathbf{u}}_2) & \hat\brlin_2(\hat B_2\Gamma') \cr
                                           \hat \brrlin_2(\mathbf u_2) & \hat\brrlin_2(\hat B_2\Gamma')
                                            \end{array}\right)=\dfrac{1}{b'}\det\left(\begin{array}{rr} \hat\brlin_2(\hat{\mathbf{u}}_2) & \hat\brlin_2(d\Loc \Gamma) \cr
                                            \hat\brrlin_2(\hat{\mathbf{u}}_2) & \hat\brrlin_2(d\Loc \Gamma)
                                            \end{array}\right)=\hat\brrlin_2(d\Loc \Gamma)
                                            \]
where the last equality holds since $\hat\brlin_2(d\Loc \Gamma)=0.$
By \eqref{eq: ul} $\hat\brrlin_i=(1,0,0,0,0,0)$. Therefore
$\hat\brrlin_2(d\Loc \Gamma)=\frac{\partial E_3^+}{\partial \Gamma}$
and so (b2) follows from Lemma~\ref{Lm: local3}.
\end{proof}

\begin{Rk}\label{RkPhysics}
Let us describe the physical and geometrical meanings of the vectors $\brlin,\brrlin, \brv,\brrv,\lin,\mathbf u$ and the results in this section. 
\begin{enumerate}
\item The structure of $d\Loc$ shows that a significant change of the behavior of the outgoing orbit parameters occurs when we vary the orbit parameters in the direction of $\lin$, which is actually varying the closest distance (called impact parameter) between $Q_3$ and $Q_4$ (see Section \ref{section: local}, especially Corollary \ref{Cor}).   The vector $w$ in $d\Glob$ shows
  that after the global map, the variable $G_4$ gets significant change as asserted by Lemma \ref{LmChooseAM}. So $\hat\lin_i \cdot w_{3-i}\neq 0$ in Lemma \ref{Lm: local4} means that by changing $G_4$ after the global map, we can change the impact parameter and hence change the outgoing orbit parameters after the local map significantly. Similarly we see $\hat\lin_i \cdot \tw\neq 0$ means the same outcome by varying $\ell_3$ instead of $G_4.$
\item The result $\hat\brlin_i\cdot \hat{\mathbf{u}}_i\neq 0$ in Lemma \ref{Lm: local4} means that by changing the outgoing orbit parameter of the local map in $\hat{\mathbf{u}}$ direction, which is in turn changed significantly by changing the impact parameter in the local map, we can change the final orbit parameter of the global map in the $\brv$ direction significantly. The vector $\hat\brlin$ has clear physical meaning. If we differentiate the outgoing asymptote $\theta^+_4=g_4^++\arctan\frac{G_4^+}{L_4^+}$, where $+$ means after close encounter of $Q_3$ and $Q_4$, we get $d\theta_4^+=L_4^+ \hat\brlin$. 

\item Lemma \ref{Lm: local3} means that if we vary the incoming orbit parameter of the local map in the direction $\Gamma$ such that there is no significant change of the outgoing parameters of the local map in certain direction, then the energy (and, hence, semimajor axis) of the ellipse after $Q_3,Q_4$ interaction will change accordingly. One may think this as varying the incoming orbit parameter while holding the outgoing asymptotes unchanged. The change of energy means the change of periods of the ellipses according to Kepler's law. Ellipses with different periods will accumulate huge phase difference during one return time $O(\chi)$ of $Q_4$. This is the mechanism that we use to fine tune the phase of $Q_3$ such that $Q_3$ comes to the correct phase to interact with $Q_4$. Since the phase is defined up to $2\pi$, we get a Cantor set as initial condition of singular orbits. 
\end{enumerate}
\end{Rk}
\section{$C^0$ estimates for global map}
\label{section: equation}
\subsection{Equations of motion in Delaunay coordinates}
\label{subsection: eqmotion}
We use Delaunay variables to describe the motion of $Q_3$ and $Q_4$
(for reader's convenience we collect the basic properties of Delaunay variables
in Appendix~\ref{section: appendix}).
We have eight variables $(L_3, \ell_3, G_3, g_3)$ and $(L_4, \ell_4, G_4, g_4)$. We consider the Hamiltonian \eqref{eq: hamloc}.

When $Q_4$ is moving to the left of the section $\left\{x_4=-\chi/2\right\}$, we consider the motion of $Q_3$ as elliptic motion with focus at $Q_2$, and $Q_4$ as hyperbolic motion with focus at $Q_1$, perturbed by other interactions. We can write the Hamiltonian in terms of Delaunay variables as
\[H_L=-\dfrac{1}{2L_3^2}+\dfrac{1}{2L_4^2}-\dfrac{1}{|Q_4|}-\dfrac{1}{|Q_3-(-\chi,0)|}-\dfrac{\mu}{|Q_3-Q_4|}.\]
When $Q_4$ is moving to the right of the section $\left\{x_4=-\chi/2\right\}$, we consider the motion of $Q_3$ as an elliptic motion with focus at $Q_2$,
and that of $Q_4$ as a hyperbolic motion with focus at $Q_2$ attracted by the pair $Q_2,Q_3$ which has mass $1+\mu$ plus a
perturbation.
For $|Q_4|\geq 2$ we have the following
Taylor expansion where $O$ is in the sense $|Q_4|\to\infty,$
\[\dfrac{\mu}{|Q_3-Q_4|}=\dfrac{\mu}{|Q_4|}+\dfrac{\mu Q_4\cdot Q_3}{|Q_4|^3}+O\left(\dfrac{\mu}{|Q_4|^3}\right). \]
Hence the Hamiltonian takes form
\[H=\dfrac{v_3^2}{2}+\dfrac{v_4^2}{2}-\dfrac{1}{|Q_3|}-\dfrac{1+\mu}{|Q_4|}-\dfrac{1}{|Q_3-(-\chi,0)|}-
\dfrac{1}{|Q_4-(-\chi,0)|}-\dfrac{\mu Q_3\cdot Q_4}{|Q_4|^3}+O\left(\dfrac{\mu}{|Q_4|^3}\right). \]
In terms of the Delaunay variables we have
\begin{equation}H_R=-\dfrac{1}{2L_3^2}+\dfrac{(1+\mu)^2}{2L_4^2}-\dfrac{1}{|Q_3+(\chi,0)|}-\dfrac{1}{|Q_4+(\chi,0)|}-\dfrac{\mu Q_4\cdot Q_3}{|Q_4|^3}+O\left(\dfrac{\mu}{|Q_4|^3}\right).\label{eq: Hr}\end{equation}
We shall use the following notation.
The coefficients of $\frac{1}{2L_4^{2}}$ in the Hamiltonian will be called $k_L=1$ and $k_R=1+\mu.$
The terms in the Hamiltonian containing $Q_4$ will be denoted by
\begin{equation}
\label{eq: V}
V_R=-\dfrac{1}{|Q_4+(\chi,0)|}-\dfrac{\mu Q_4\cdot Q_3}{|Q_4|^3}+O\left(\dfrac{\mu}{|Q_4|^3}\right),\text{ and }
V_L=-\dfrac{1}{|Q_4|}-\dfrac{\mu}{|Q_3-Q_4|}.
\end{equation}
Here subscripts L and R mean that the corresponding expressions are used when $Q_4$ is to the left (respectively to the right)
of the line $Q=-\frac{\chi}{2}.$
Likewise for the terms containing $Q_3$ we define
\begin{equation}
\begin{aligned}
&U_R=-\dfrac{1}{|Q_3+(\chi,0)|}-\dfrac{\mu Q_4\cdot Q_3}{|Q_4|^3}+O\left(\dfrac{\mu}{|Q_4|^3}\right),\ U_L=-\dfrac{1}{|Q_3-(-\chi,0)|}-\dfrac{\mu}{|Q_3-Q_4|}.
\end{aligned}
\label{eq: U}
\end{equation}
The use of subscripts $R,L$ here is the same as above.
Let us write down the full Hamiltonian equations with the subscripts $R$ and $L$ suppressed.
\begin{equation}
\begin{aligned}
\begin{cases}
&\dot{L}_3=-\dfrac{\partial Q_3}{\partial \ell_3}\cdot\dfrac{\partial U}{\partial Q_3},\quad \dot{\ell}_3=\dfrac{1}{L^3_3}+\dfrac{\partial Q_3}{\partial L_3}\cdot\dfrac{\partial U}{\partial Q_3},\\
&\dot{G}_3=-\dfrac{\partial Q_3}{\partial g_3}\cdot\dfrac{\partial U}{\partial Q_3}, \quad\dot{g}_3=\dfrac{\partial Q_3}{\partial G_3}\cdot\dfrac{\partial U}{\partial Q_3},\\
&\dot{L_4}=-\dfrac{\partial Q_4}{\partial \ell_4}\cdot\dfrac{\partial V}{\partial Q_4},\quad\dot{\ell}_4=-\dfrac{k^2}{L^3_4}+\dfrac{\partial Q_4}{\partial L_4}\cdot\dfrac{\partial V}{\partial Q_4},\\
&\dot{G}_4=-\dfrac{\partial Q_4}{\partial g_4}\cdot\dfrac{\partial V}{\partial Q_4},\quad\dot{g}_4=\dfrac{\partial Q_4}{\partial G_4}\cdot\dfrac{\partial V}{\partial Q_4}.\\
\end{cases}
\end{aligned}\label{eq: Hamiltonian eq}
\end{equation}

Next we use the energy conservation to eliminate $L_4$. Setting $H=0$, we have
\begin{equation}
\begin{aligned}
\dfrac{L^3_4}{k_R^2}&=k_RL_3^3
\cdot\left(1-3L_3^2\left(\dfrac{1}{|Q_3+(\chi,0)|}+\dfrac{1}{|Q_4+(\chi,0)|}\right.\right.\\
&+\left.\left.\dfrac{\mu Q_4\cdot Q_3}{|Q_4|^3}+O\left(\dfrac{\mu}{|Q_4|^3}\right)+ O(1/\chi^2)\right)\right):=k_R L_3^3+W_R,\\
\dfrac{L^3_4}{k_L^2}&=k_LL_3^3\left(1-3L_3^2\left(\dfrac{1}{|Q_3+(\chi,0)|}+\dfrac{1}{|Q_4|}-\dfrac{\mu}{|Q_4-Q_3|}+O(1/\chi^2)\right)\right)\\
:&=k_L L_3^3+W_L.
\end{aligned}\label{eq: W}
\end{equation}

We use $\ell_4$ as the independent variable. Dividing \eqref{eq: Hamiltonian eq} by $\dot{\ell_4}$ and using
\eqref{eq: W} to eliminate $L_4$ we obtain

\begin{equation}
\begin{aligned}
&\begin{cases}
\dfrac{d L_3}{d \ell_4}&=(kL_3^3+W)\dfrac{\partial Q_3}{\partial \ell_3}\cdot\dfrac{\partial U}{\partial Q_3}\left(1+(kL_3^3+W)\dfrac{\partial Q_4}{\partial L_4}\cdot\dfrac{\partial V}{\partial Q_4}\right)\\
\dfrac{d \ell_3}{d \ell_4}&=-(kL_3^3+W)(\dfrac{1}{L_3^3}+\dfrac{\partial Q_3}{\partial L_3}\cdot\dfrac{\partial U}{\partial Q_3}) \left(1+(kL_3^3+W)\dfrac{\partial Q_4}{\partial L_4}\cdot\dfrac{\partial V}{\partial Q_4}\right)\\
\dfrac{d G_3}{d \ell_4}&=(kL_3^3+W)\dfrac{\partial Q_3}{\partial g_3}\cdot\dfrac{\partial U}{\partial Q_3}\left(1+(kL_3^3+W)\dfrac{\partial Q_4}{\partial L_4}\cdot\dfrac{\partial V}{\partial Q_4}\right)\\
\dfrac{d g_3}{d \ell_4}&=-(kL_3^3+W)\dfrac{\partial Q_3}{\partial G_3}\cdot\dfrac{\partial U}{\partial Q_3}\left(1+(kL_3^3+W)\dfrac{\partial Q_4}{\partial L_4}\cdot\dfrac{\partial V}{\partial Q_4}\right)\\
\dfrac{d G_4}{d \ell_4}&=(kL_3^3+W)\dfrac{\partial Q_4}{\partial g_4}\cdot\dfrac{\partial V}{\partial Q_4}\left(1+(kL_3^3+W)\dfrac{\partial Q_4}{\partial L_4}\cdot\dfrac{\partial V}{\partial Q_4}\right)\\
\dfrac{d g_4}{d \ell_4}&=-(kL_3^3+W)\dfrac{\partial Q_4}{\partial G_4}\cdot\dfrac{\partial V}{\partial Q_4}\left(1+(kL_3^3+W)\dfrac{\partial Q_4}{\partial L_4}\cdot\dfrac{\partial V}{\partial Q_4}\right)\\
\end{cases}\\
&+O\left(\dfrac{\mu}{|Q_4|^3}+1/\chi^2\right).
\end{aligned}\label{eq: 6 equations}
\end{equation}
We shall use the following notation:
$X=(L_3, \ell_3, G_3, g_3),$ $Y=(G_4, g_4).$

\subsection{{\it a priori} bounds}
In this section, we give some estimates that will be used to estimate the derivatives of the global map in later sections. 
\subsubsection{Estimates of positions}
We have the following estimates for the positions.
\begin{Lm}
\label{Lm: position}
Given $C$ and $D>0$ there exists $C'$ such that
if
\begin{equation}
\label{APriori34}
|Q_3|<2-D,\quad |Q_{4y}|<C
\end{equation}
then
\begin{itemize}
\item[(a)] we have
\begin{equation}
\label{LmPosEqA}
 \left|\dfrac{\partial Q_3}{\partial X}\right|<C';
\end{equation}
\item[(b)]  when $Q_4$ is moving to the right
of the section $\{x_4=-\chi/2\}$ we have
\begin{equation}\label{LmPosEqB}
|Q_4(\ell_4)| \begin{cases}&\geq 2,\quad \mathrm{if\ }|\ell_4^*|\leq |\ell_4|\leq C\\
 &\in\left[\frac{1}{2},2\right] L_4^2(\ell_4^*)|\ell_4|,\quad \mathrm{if\ }|\ell_4|\geq C,
 \end{cases}\end{equation}
 where $\ell_4^*$ is the value of $\ell_4$ restricted on $x_4=-2;$\\
when $Q_4$ is moving to the left of the section $\{x_4=-\chi/2\},$ we have
\begin{equation}\label{LmPosEqC}
|Q_4(\ell_4)-Q_1|\leq 2 L_4^2 (\ell_4^*)|\ell_4|+C'\end{equation}
 for some constant $C'$ where $\ell_4^*$ is the value of $\ell_4$ on the section $\{x_4=-\chi/2\}$.
\end{itemize}
\end{Lm}
This lemma justifies the following intuitive facts. Since $Q_3$ and $Q_4$ are away from close encounter, the motion of $Q_3$ is almost Kepler elliptic motion hence we get item (a). The motion of $Q_4$ is a perturbed Kepler hyperbolic motion for both the left and the right case, hence for most of the time $Q_4$ as a function of the time $\ell_4$ is almost linear (item (b)). To give the complete proof we have to
use the Hamiltonian equations. See Section \ref{SSPrLmPos}. The next several lemmas relies on the conclusion of this lemma. 



\begin{Lm}
\label{LmQderT}
If inequalities \eqref{APriori34}, \eqref{LmPosEqB}, \eqref{LmPosEqC} are valid and in addition
\begin{equation}
\label{LBounded}
1/C\leq |L_3|, |L_4|\leq C, \quad |G_3|, |G_4|<C,
\end{equation}
then we have
$$\dfrac{\partial Q_4}{\partial \ell_4}=O(1), \quad \dfrac{\partial Q_4}{\partial (L_4,G_4,g_4)}=O(\ell_4), \quad
\dfrac{\partial Q_4}{\partial g_4}\cdot Q_4=0\text{ and }\dfrac{\partial Q_4}{\partial G_4}\cdot Q_4=O(\ell_4)$$
as $|\ell_4|\to\infty$.
\end{Lm}
\begin{proof}
This follows directly from
Lemma~\ref{LM: 1stder} in Appendix~\ref{subsubsection: 1stderivative}.
\end{proof}

\subsubsection{Estimates of potentials}

\begin{Lm}
Under the assumptions of Lemma \ref{LmQderT}
we have the following estimates for the potentials $U,V,W$ as $1/\chi\ll\mu\to 0$:\begin{itemize}
\item[(a)] When $Q_4$ is moving to the right of the section $\{x_4=-\chi/2\}$, we have \[V_R,\ U_R,\ W_R=O\left(\dfrac{1}{\chi}+\dfrac{\mu}{\ell_4^2+1}\right).\]
\item[(b)] When $Q_4$ is moving to the left of the section $\{x_4=-\chi/2\}$, we have\[V_L,\ U_L,\ W_L=O\left(\dfrac{1}{\chi}\right).\]
\end{itemize}
\label{Lm: potential}
\end{Lm}

\begin{proof}
This follows directly from equations \eqref{eq: V}, \eqref{eq: U} and \eqref{eq: W} and \eqref{LmPosEqB} in Lemma \ref{Lm: position}. For part (a), the estimate $O(\frac{1}{\chi})$ comes from $\frac{1}{|Q_{3,4}+(\chi,0)|}$ in the potentials $V_R,U_R,W_R$ and the estimate $O(\frac{\mu}{\ell_4^2+1})$ comes from the term $\frac{\mu Q_4\cdot Q_3}{|Q_4|^3}$ since $Q_4$ moves away from $Q_2$ almost linearly in $\ell_4$ according to \eqref{LmPosEqB}. Our choice of the section $\{x_4=-2\}$ excludes the collision between $Q_3$ and $Q_4$. So we put $\frac{\mu}{\ell_4^2+1}$ to stress the fact that the denominator is bounded away from zero. We do the same thing in the following proofs without mentioning it any more.
\end{proof}

\subsubsection{Estimates of gradients of potentials}
In this section, we estimate the gradients of the potentials $U,V$, which appears in the Hamiltonian equations.  
\begin{Lm}
Under the assumptions of Lemma \ref{LmQderT}
we have the following estimates for the gradients of the potentials $U,V$ as $1/\chi\ll\mu\to 0$
\begin{equation}
\begin{aligned}
&\dfrac{\partial U_R}{\partial Q_3},\ \dfrac{\partial Q_4}{\partial(G_4, g_4)}\dfrac{\partial V_R}{\partial Q_4}=O\left(\dfrac{1}{\chi^2}+\dfrac{\mu}{\ell_4^2+1}\right),\quad \dfrac{\partial V_R}{\partial Q_4}=O\left(\dfrac{1}{\chi^2}+\dfrac{\mu}{|\ell_4|^3+1}\right),\\
&\dfrac{\partial U_L}{\partial Q_3}=O\left(\dfrac{1}{\chi^2}\right),\quad \dfrac{\partial V_L}{\partial Q_4}=O\left(\dfrac{1}{\chi^2}\right),\quad \dfrac{\partial Q_4}{\partial(G_4, g_4)}\dfrac{\partial V_L}{\partial Q_4}=O\left(\dfrac{1}{\chi^2}\right).
\end{aligned}
\end{equation}
\label{Lm: gradient}
\end{Lm}
\begin{proof}
The estimates for the $\frac{\partial }{\partial Q_{3,4}}$ terms are straightforward. Indeed, we only need to use the fact $\left|\frac{d}{dx}\frac{1}{|x|^k}\right|=\frac{k}{|x|^{k+1}}$ together with the estimates in Lemma \ref{Lm: position}.

The estimates of all $\frac{\partial }{\partial (G_4,g_4)}$ terms are similar. We consider for instance
$\frac{\partial Q_4}{\partial G_4}\frac{\partial V_R}{\partial Q_4}.$
We have
\begin{equation}
\label{DotG}
\dfrac{\partial Q_4}{\partial G_4}\dfrac{\partial V_R}{\partial Q_4}=\dfrac{\partial Q_4}{\partial G_4}\dfrac{Q_4+(\chi,0)}{|Q_4+(\chi,0)|^3}
+O\left(\mu \left|\dfrac{\partial Q_4}{\partial G_4}\right| |Q_4|^{-3}\right).
\end{equation}
The second term here is $O(\mu/(\ell_4^2+1))$ due to \eqref{LmPosEqB} and Lemma~\ref{LM: 1stder}(a).
To handle the first term let  $\frac{\partial Q_4}{\partial G_4}=(\ba, \bb),$
$Q_4=(x,y).$ Note that equations \eqref{eq: delaunay4}, \eqref{eq: hypul}, \eqref{APriori34}, \eqref{LmPosEqB},
and \eqref{LBounded} show
that $x,\ell_4$ are all comparable in the sense that the ratios between any two of these qualities are bounded from above and below.
On the other hand Lemma \ref{LM: 1stder}(a) tells us that
$\ba x+\bb y=O(\ell_4). $ Since $\bb y=O(\bb)=O(\ell_4)$ we conclude that $\ba x=O(\ell_4)$ and thus $\ba=O(1).$
Thus the first term in \eqref{DotG} is
$\frac{\frac{\partial Q_4}{\partial G}\cdot Q_4 +\ba \chi}{|Q_4+(\chi,0)|^3}. $
The numerator here is $O(\chi)$ while the denominator is at least $(\chi/2)^3.$
This completes the estimate of $\frac{\partial Q_4}{\partial G_4}\frac{\partial V_R}{\partial Q_4}.$
Other derivatives are similar.
\end{proof}

Plugging the above estimates into \eqref{eq: 6 equations}
we obtain the following estimate of the Hamiltonian equations.

\begin{Lm}
Under the assumptions of Lemma \ref{LmQderT} we have the following estimates on the RHS of \eqref{eq: 6 equations} as $1/\chi\ll\mu\to 0$.
\begin{itemize}
\item[(a)] When $-\frac{\chi}{2}\leq x_4 \leq -2$
we have
\begin{equation}
\begin{aligned}
&\dfrac{dL_3}{d\ell_4},\dfrac{dG_3}{d\ell_4},\dfrac{dg_3}{d\ell_4},\dfrac{dG_4}{d\ell_4},\dfrac{dg_4}{d\ell_4}=O\left(\dfrac{1}{\chi^2}+\dfrac{\mu}{\ell_4^2+1}\right),\quad
\dfrac{d\ell_3}{d\ell_4}=-1+O(\mu).\\
\end{aligned}\nonumber
\end{equation}
\item[(b)] When $Q_4$ is moving to the left of the section $\left\{x_4=-\chi/2\right\}$, we have
\begin{equation}
\begin{aligned}
&\dfrac{dL_3}{d\ell_4},\dfrac{dG_3}{d\ell_4},\dfrac{dg_3}{d\ell_4},\dfrac{dG_4}{d\ell_4},\dfrac{dg_4}{d\ell_4}=O\left(\dfrac{1}{\chi^2}\right),\quad
\dfrac{d\ell_3}{d\ell_4}=-1+O\left(\frac{1}{\chi}\right).\\
\end{aligned}\nonumber
\end{equation}\label{Lm: orderham}
\end{itemize}
\end{Lm}
\begin{proof}
The proof is simply an application of Lemma \ref{Lm: gradient}. We only remark that in the left case, the orbit is very close to collision and the hyperbolic Delaunay coordinates becomes singular. We use Lemma \ref{LmSmallu} to show that the derivatives of the Cartesian coordinates with respect to $L_4,G_4,g_4$ are bounded. Moreover, since we treat $\ell_4$ as the new time, we never take the $\ell_4$ derivative in the RHS of the Hamiltonian equations, hence the dependence on $\ell_4$ is continuous. 
\end{proof}
In Section \ref{section: variational} we will need the following bounds on the second
derivatives to estimate the variational equations.

\begin{Lm}Under the assumptions of Lemma \ref{LmQderT}
we have the following estimates for the second derivatives.
\begin{equation}
\begin{aligned}
&\dfrac{\partial^2 U_R}{\partial Q_3^2}=O\left(\dfrac{1}{\chi^3}+\dfrac{\mu}{\ell_4^2+1}\right),\ \dfrac{\partial^2 V_R}{\partial Q_4^2}=O\left(\dfrac{1}{\chi^3}+\dfrac{\mu}{\ell_4^4+1}\right),\\
&\dfrac{\partial^2 (U_R,V_R)}{\partial Q_3\partial Q_4}=O\left(\dfrac{\mu}{|\ell_4|^3+1}\right),\\
&\dfrac{\partial^2 U_L}{\partial Q_3^2}=O\left(\dfrac{1}{\chi^3}\right),\quad \dfrac{\partial^2 V_L}{\partial Q_4^2}=O\left(\dfrac{1}{\chi^3}\right),\quad \dfrac{\partial^2 (U_L,V_L)}{\partial Q_3\partial Q_4}=O\left(\dfrac{1}{\chi^3}\right).\\
\end{aligned}
\end{equation}\label{Lm: 2ndder}
\end{Lm}
We omit the proof since it is again a direct computation.
\subsection{Proof of Lemma~\ref{Lm: position}}
\label{SSPrLmPos}
\begin{proof}[Proof of Lemma~\ref{Lm: position}.]
Let $\tau$ be the maximal time interval such that
\begin{equation}\label{eq: assump}
\frac{3}{4}|L_3(\ell_4^*)|\leq |L_3|\leq \frac{4}{3}|L_3(\ell_4^*)|,\quad
\frac{3}{4}|G_i(\ell_4^*)|\leq |G_i(\ell_4)|\leq \frac{4}{3}|G_i(\ell_4^*)|,\ i=3,4,
\end{equation}
on $[0, \tau]$ where $\ell_4^*$ is the value $\ell_4$ restricted on $\{x_4=-2\}$.
\eqref{eq: assump} implies that
$e_4=\sqrt{1+G^2_4/L^2_4}$ is bounded. We always have we have $|Q_4|\geq 2$ since $Q_4$ is to the left of the section $\{x_4=-2\}$.
Therefore \eqref{eq: W} implies that
$L_4=L_3+O(\mu)$ in the right case and $L_4=L_3+O(1/\chi)$ in the left case.
Now formula \eqref{eq: delaunay4} and Lemma~\ref{Lm: simplify} allow us replace $\sinh u,\cosh u$ by
$(1+o(1))\frac{\ell_4}{e_4}$ as $|\ell_4|\to\infty.$
\begin{equation}\label{EqQ4Linear}\begin{aligned}|Q_4|&=L_4\sqrt{L_4^2(\cosh u-e_4)^2+G_4^2\sinh^2 u}\\
&=L_4\sqrt{L_4^2\left(\cosh^2 u-2e_4\cosh u + e_4^2\right)+(L_4^2e_4^2-L_4^2)\sinh^2 u}\\
&=L^2_4\sqrt{1-2e_4\cosh u + e_4^2+e_4^2\sinh^2 u}=L_4^2(e_4\cosh u-1)
\end{aligned}
\end{equation}
This proves estimate \eqref{LmPosEqB} for $t\leq \min(\tau, \brtau)$ where $\brtau$ is the first time then $x_4$
reaches $-\frac{\chi}{2}.$
Thus for $t\leq \min(\tau, \brtau)$ the assumptions of Lemma \ref{Lm: orderham} are satisfied and hence
\begin{equation}
\label{DerStrip}
\dfrac{dL_3}{d\ell_4},\dfrac{dG_4}{d\ell_4},\dfrac{dG_3}{d\ell_4}=O\left(\dfrac{1}{\chi^2}+\dfrac{\mu}{|Q_4-Q_3|^2}\right)
\end{equation}
(note that to prove the estimates in Lemma \ref{Lm: orderham} in the right case we do not need the assumption \eqref{LmPosEqC}).
If we integrate  \eqref{DerStrip} w.r.t. $\ell_4$ on the interval of size $O(\chi)$
we find that the oscillations of $L_3,G_4,G_3$ are $O(\mu).$
Therefore $\brtau<\tau$ and we obtain the estimates of \eqref{LmPosEqB} up to the time $\brtau.$

The analysis of the cases when $Q_4$ is to the left of the section $\{x_4=-\chi/2\}$
and then it travels back from $\{x_4=-\chi/2\}$ to $\{x_4=-2\}$ is similar once we establish the bounds on the angular momentum at the
beginning of the corresponding pieces of the orbit. Let us show, for example, that at the moment when the orbit hits
$\{x_4=-\frac{\chi}{2}\}$
for the first time, the angular momentum of $Q_4$ w.r.t. $Q_1$ is $O(1).$ Indeed we have already
established that $G_{4R}=-\frac{\chi v_{4y}}{2}-y v_{4x}=O(1).$ Also \eqref{eq: assump} shows that $v=O(1)$ and so
\eqref{APriori34} implies that $y v_{4x}=O(1).$ Accordingly 
$$\chi v_{4y}=-G_{4R}-y v_{4y}=O(1)$$ and hence
$G_{4L}=G_{4R}+\chi v_{4y}=O(1)$ as claimed. The argument for the second time the orbit hits
$\{x_4=-\frac{\chi}{2}\}$ is the same. This completes the proof of part (b).


To show part (a), we notice $\frac{\partial Q_3}{\partial X}$ depends on $\ell_3,g_3$ periodically according to equation \eqref{DelEll}. So part $(a)$ follows since we have already obtained bounds on  $L_3$ and $G_3$.
\end{proof}

The next lemma gives more information about the $Q_4$ part of the orbit than Lemma~\ref{Lm: position}. It justifies the assumptions of Lemma \ref{LM: 1stder}.
\begin{Lm}\label{Lm: tilt}
Under the hypothesis of Lemma~\ref{LmQderT}, we have as $1/\chi\ll\mu\to 0$:
\begin{itemize}
\item[(a)] when $Q_4$ is moving to the right of the section $\left\{x=-\chi/2\right\}$, we have
\[\tan g_4=-\sign(u) \dfrac{G_4}{L_4}+O\left(\dfrac{\mu}{|\ell_4|+1}+\dfrac{1}{\chi}\right).\]
\item[(b)] when $Q_4$ is moving to the left of the section $\left\{x=-\chi/2\right\}$, then \[G_4,g_4=O(1/\chi).\]
\end{itemize}
\end{Lm}
\begin{proof}
We prove part (b) first. From equation \eqref{eq: Q4} we see that if $\ell_4$ is of order $\chi$ and $y=O(1)$ then
$ G_4 \cos g_4+\sign(u) L_4 \sin g_4=O(1/\chi). $
Integrating the estimates of Lemma \ref{Lm: orderham}(b) we see
that during the time $x_4\leq -\chi/2$ we have
\begin{equation}
\label{SmallOscB}
G_4=G^*+O(1/\chi), \quad L_4=L^*+O(1/\chi), \quad
g_4=g^*+O(1/\chi)
\end{equation}
where $(L^*, G^*, g^*)$
are the orbit parameters of $Q_4$ then it first hits $\{x_4=-\chi/2\}.$
It follows that both
$$ G^* \cos g^*+L^* \sin g^*=O(1/\chi), \quad\text{and}\quad
G^* \cos g^*-L^* \sin g^*=O(1/\chi) .$$
Since $L^*$ is not too small this is only possible if $G^*=O(1/\chi),$
$g^*=O(1/\chi).$ Now part (b) follows from \eqref{SmallOscB}.

The proof of part (a) is similar. Consider for example the case when $Q_4$ moves to the right.
Now \eqref{SmallOscB} has to be replaced by
\begin{equation}
\label{SmallOscA}
(G_4,L_4,g_4)=(G^*,L^*,g^*)+O\left(\dfrac{\mu}{|\ell_4|+1}+\dfrac{1}{\chi}\right), 
\end{equation}
(since we use part (a) of Lemma \ref{Lm: orderham} rather than part (b)). As before we have
\[G^* \cos g^*- L^*\sin g^*=O(1/\chi). \]
Since $\cos g^*$ can not be too small (since otherwise
$G^* \cos g^*- L^*\sin g^*\approx L^* \sin g$ would not be small) we can divide
the last equation by $L^* \cos g^*$ to get $$\tan g^*=-\dfrac{G^*}{L^*}+O\left(\dfrac{1}{\chi}\right).$$
Now part (a) follows from \eqref{SmallOscA}.
\end{proof}

\subsection{Proof of Lemma~\ref{LmGMC0}}
We begin by demonstrating that the orbits satisfying the conditions of Lemma \ref{LmGMC0}
satisfy the assumptions of Lemma \ref{Lm: orderham}.

\begin{Lm}
\label{Lm: strip}
\begin{itemize}
\item[(a)] Given $D, C$ there exist constants $\hC, \mu_0$ such that for $\mu\leq\mu_0$ the following holds.
Consider a time interval $[0,T]$ and an orbit satisfying the following conditions
\begin{itemize}
\item[(i)] $x_4(t)\in (-\chi-1, -2)$ for $t\in (0,T),$ $x_4(0)=-2,$ $x_4(T)=-\chi.$
\item[(ii)] $y_4(0)\leq C,$ $y_4(T)\leq C.$
\item[(iii)] At time $0$, $Q_3$ moves on an elliptic orbit which is completely contained in $\{x_3\geq -(2-D)\}.$
\end{itemize}
Then $|y_4(t)|\leq \hC$ for all $t\in [0, T].$
\item[(b)] The result of part (a) remains valid if $\mathrm{(i)}$ is replaced by\\
\quad $(\tilde{\mathrm{i}})$ $x_4(t)< -2$ for $t\in (0,T),$ $x_4(0)=x_4(T)=-2.$
\end{itemize}
\end{Lm}

\begin{proof}
To prove part (a) we first establish a preliminary estimate showing that $Q_4$ travels roughly in the direction of $Q_1.$

\begin{sublemma}
Given $\tilde\theta>0$ there exists $\mu_0, \chi_0$ such that the following holds for $\mu\leq \mu_0,$
$\chi>\chi_0.$
If the outgoing asymptote satisfies
\begin{equation}
\label{EqMiss}
|\pi-\theta^+_4(0)|>\tilde\theta
\end{equation}
then $Q_4$ escapes from the two center system.
\label{KeepDirection}
\end{sublemma}

\begin{proof}
We consider the case $\theta^+_4(0)<\pi-\tilde\theta,$ the other case is similar. If we disregard the influence of
$Q_1$ and $Q_3$ then $Q_4$ would move on a hyperbolic orbit and its velocity would approach
$(\sqrt{2 E_4(0)} \cos\theta_4^+(0), \sqrt{2 E_4(0)} \sin \theta_4^+(0)).$ Accordingly given $R$ we can find $\brt,$ $\mu_0$
such that uniformly over all orbits satisfying (i)-(iii) and $\theta^+_4(0)<\pi-\tilde\theta$ we have for $\mu\leq \mu_0$
$$ y_4(\brt)>R, \quad v_{4y}(\brt)>0.8\sqrt{E_4(0)} \sin \tilde\theta.$$
Let $\tilde{t}=\inf\{t>\brt: v_{4y}<\frac{\sqrt{E_4(0)}}{2} \sin \tilde\theta\}.$
We shall show that $\tilde{t}=\infty$ which implies the sublemma since for $t\in [\brt, \tilde{t}]$ we have
\begin{equation}
\label{YLinear}
y_4(t)>R+(\tilde t-\brt) \dfrac{\sqrt{E_4}}{2} \sin \tilde\theta.
\end{equation}
To see that $\tilde t=\infty$ note that \eqref{YLinear} implies that
$$ |\dot{v}_{4y}|\leq \dfrac{1}{(R+(\tilde t-\brt) \frac{\sqrt{E_4}}{2} \sin \tilde\theta)^2}$$
and so
$$ |v_{4y}(\tilde t)-v_{4y}(\brt)|\leq \int_0^\infty \dfrac{ds}{(R+s \frac{\sqrt{E_4}}{2} \sin \tilde\theta)^2}=
\dfrac{2}{R \sqrt{E_4} \sin \tilde\theta} .$$
Hence if $R$ is sufficiently large we have
$v_{4y}(\tilde t)\geq \frac{\sqrt{E_4}}{2} \sin\tilde\theta$ which is only possible if $\tilde t=\infty.$
\end{proof}
We now consider the case $|\pi-\theta_4^+|<\tilde\theta.$ Arguing as above we see that given $R,$ we can find for $\mu$ small
enough a time $\brt$ such that
$$ x_4(\bar t)<-R, \quad v_{4x}(\brt)<-0.8\sqrt{E_4(0)} \cos\tilde\theta. $$
Let $\hat t$ be the first time after $\brt$ such that $x_4=-(\chi-R).$ Arguing as in Sublemma \ref{KeepDirection} we see that
for $t\in [\brt, \hat t]$ we have
$|v_{4x}|\geq \frac{\sqrt{E_4(0)}}{2} \cos\tilde\theta.$ Hence the force from $Q_2$ and $Q_3$ is $O(1/t^2)$
and the force from $Q_1$ is $O(1/(\hat t-t)^2).$  Accordingly $v_4$ remains $O(1)$ so the energy of $Q_4$
remains bounded. Next if $|y_4(\hat t)|>R$ then the argument of Sublemma \ref{KeepDirection}
shows that
$y_4(T)>R/2$ giving a contradiction if $R>2C.$ Accordingly we have for $t\in [\hat t, T]$ that
$E_4=O(1),$ $y_4=O(1)$ and $|G_{4L}(\hat t)|=O(1)$. We point out that the $O(1)$'s here are as $\chi\to\infty$ and might depend on $R$. It remains to show that $|y_4(t)|<\hC$ for $t\in [\brt, \hat t].$
To this end let $t^*$ be the first time when $x_4=-\frac{\chi}{2}.$ We first get $E_4=O(1)$ for $t\in [t^*,\hat t]$ since by arguing as in the Sublemma we get the oscillation of $v_4$ is bounded. Next,
we have that $G_{4L}(t^*)=O(1)$
 since $\dot{G}_{4L}=O(1/\chi)$,(this estimate of $\dot G_{4L}=\ddot v_4\times x_4$ does not need any assumption on $G_{4L}$.) On other other hand, we have
$G_{4R}(t^*)=O(1)$ by integrating the equation $\dot G=O(1/\chi)$ with initial condition $G_{4R}(0)=O(1)$ provided by the assumption of the lemma.  Therefore $\chi v_{4y}(t^*)=G_{4R}-G_{4L}=O(1)$ and so $v_{4y}(t^*)=O(1/\chi).$
Since $G_{4L}(t^*)=\left(\dfrac{\chi}{2} v_{4y}-y_{4} v_{4x}\right)(t^*)$ we have $y_4(t^*)=O(1).$
Next for $t\in [t^*, \hat t]$ we have
$$ y_4(t)=y_4(t^*)+v_{4y}(t^*)(t-t^*)+
\int_{t^*}^t \int_{t^*}^u\ddot{y}_4(s) dsdu. $$
Note that
$$\ddot{y}_4(s)=O\left(\dfrac{y_4}{|Q_4-Q_1|^3}\right)=O\left(\dfrac{y_4}{(\hat t-s+R)^3}\right). $$
Combining the last two estimates we get
$$ |y_4(t)|\leq C_1+C_2 \sup_s\{|y(s)|\}\int_{t^*}^t \int_{t^*}^u\dfrac{dsdu}{(\hat t-s+R)^3} \leq C_1+C_2\left(\dfrac{1}{R}+\dfrac{1}{\chi}\right)\sup_s |y_4(s)|. $$
Here $C_1$ might depend on $R$ through the estimates of $y_4(t^*),v_{4y}(t^*)$ but $C_2$ does not. We choose $R$ large enough to get that $|y|$ is bounded on $[t^*, \hat t]$.
The argument for $[\brt, t^*]$ is the same except 
that the force from $Q_3$ is $O\left(\frac{\mu y_4}{|Q_4|^3}\right)$. This completes the proof of part (a).

To prove part (b) we note that if $|y_4(\hat t)|>R^2$ then $Q_4$ escapes by the argument of
Sublemma \ref{KeepDirection}. Hence $|y_4(\hat t)|<R^2.$ This implies (via already established
part (a) of the lemma) that $y$ is uniformly bounded on $[0, \hat t].$ The argument for $[\hat t, T]$
is the same with the roles of $Q_1$ and $Q_2$ interchanged.
\end{proof}

\begin{proof}[Proof of Lemma~\ref{LmGMC0}]
Initially we have $1/C\leq |L_3|\leq C,\ |G_3|,|G_4|\leq C$ for some constant $C>1$. We assume \eqref{eq: assump} from time $0$ to some time $\tau$. 
Due to the previous lemma, we can use Lemma~\ref{Lm: orderham} to get the estimates on the time interval $[0,\tau]$
\[\dfrac{dL_3}{d\ell_4},\dfrac{dG_3}{d\ell_4},\dfrac{dg_3}{d\ell_4},\dfrac{dG_4}{d\ell_4},\dfrac{dg_4}{d\ell_4}=O\left(\dfrac{1}{\chi^2}+\dfrac{\mu}{\ell_4^2+1}\right).\]
We integrate the equations to get $O(\mu)$ oscillations of $L_3,G_3,G_4$ so that $\tau$ can be extended to as large as $\chi$.
For part (a) of Lemma \ref{LmGMC0}, we integrate the equations of $\frac{dL_3}{d\ell_4},\frac{dG_3}{d\ell_4},\frac{dg_3}{d\ell_4},$ over time of order $\chi$ as $Q_4$ first moves away from $Q_2$ and then comes back. Therefore we get
$$O\left(2 \int_2^{\chi}\left[\dfrac{\mu }{\ell_4^2+1}+\dfrac{1}{\chi^2}\right]d\ell_4\right)=O(\mu)$$ estimate for the change of
$L_3, G_3$ and $g_3$ proving part (a).

Part (b) of Lemma \ref{LmGMC0} follows from Lemma \ref{Lm: tilt}.

For part (c), applying the bounds $1/C\leq |L_3|\leq C,\ |G_3|,|G_4|\leq C$ to equation \eqref{eq: W},  we get $1/C'<|L_4|<C'$ for some constant $C'>1$. Next, when restricted to the section $\{x_4=-\chi/2\}$, we set in \eqref{eq: Q4} $q_1=-\chi/2$ and in \eqref{eq: Q4l} $q_1=\chi/2$. We use Lemma \ref{Lm: tilt} In both the left and the right cases to get $|\ell_4|=O(\chi)$ restricted to the section $\{x_4=-\chi/2\}$. Next use the $\dot \ell_4$ equation in the Hamiltonian equation \eqref{eq: Hamiltonian eq} to get $|\dot\ell_4|>c>0$ for some constant $c$. Therefore for each piece of orbit $I,III,V$, it takes time $|t|=O(\chi)$ to get $|\ell_4|=O(\chi)$. Adding the time for the three pieces together, we get that the total time defining the global map is $O(\chi)$.  
\end{proof}
\section{Derivatives of the Poincar\'{e} map}
\label{subsection: plan}

In computing $C^1$ asymptotics of both local and global maps we will need formulas for the derivatives of Poincar\'{e} maps between
two sections. Here we give the formulas for such derivatives for the later reference.

Recall our use of notations. $X$ denotes $Q_3$ part of our system and $Y$ denotes $Q_4$ part.
Thus
\begin{equation}
\label{Sep-XY}
 X=(L_3,\ell_3,G_3,g_3), \quad Y=(G_4,g_4). 
\end{equation}
$(X,Y)^i$ will denote the orbit parameters at the initial section and $(X,Y)^f$ will denote the orbit parameters at the
final section. Likewise we denote by $\ell_4^i$ the initial ``time" when $Q_4$ crosses some section, and by $\ell_4^f$  final ``time" when $Q_4$ arrives at the next. We abbreviate the RHS of
\eqref{eq: 6 equations})
as \[X'=\mathcal{U},\quad Y'=\mathcal{V}.\]
Here $'$  is the derivative w.r.t. $\ell_4$. We also denote $Z=(X,Y)$ and $\mathcal W=(\mathcal U,\mathcal V)$ to simplify the notations further.

Suppose that we want to compute the derivative of the Poincar\'{e} map between the sections $S^i$ and $S^f.$ Assume that
on $S^i$ we have  $\ell_4=\ell_4^i(Z^i)$ and on $S^f$ we have  $\ell_4=\ell_4^f(Z^f).$ We want to compute the derivative
$\cD$ of the Poincar\'{e} map along the orbit starting from $(Z^i_*, \ell_*^i)$ and ending at $(Z^f_*, \ell_*^f).$ We have
$\cD=dF_3 dF_2 dF_1$ where
$F_1$ is the Poincar\'{e} map between $S^i$ and $\{\ell_4=\ell_*^i\},$
$F_2$ is the flow map between the times $\ell_*^i$ and $\ell_*^f,$ and
$F_3$ is the Poincar\'{e} map between $\{\ell_4=\ell_*^f\}$ and $S^f.$ We have
$F_1=\Phi(Z^i, \ell_4(Z^i), \ell_*^i)$ where $\Phi(Z, a, b)$ denotes the flow map
starting from $Z$ at time $a$ and ending at time $b.$ Since
\[ \dfrac{\partial \Phi}{\partial Z}(Z_*^i, \ell_*^i, \ell_*^i)=\Id, \quad
\dfrac{\partial \Phi}{\partial a}=-\mathcal{W} \]
we have
$dF_1=\Id-\mathcal W(\ell_4^i)\otimes \frac{D \ell_4^i}{D Z^i}.$
Inverting the time we get
\[dF_3=\left(\Id-\mathcal W(\ell_4^f)\otimes \dfrac{D \ell_4^f}{D Z^f}\right)^{-1}.\]
Finally $dF_2=\frac{DZ(\ell_*^f)}{DZ(\ell_*^i)}$ is just the fundamental solution of the variational equation between the times
$\ell_*^i$ and $\ell_*^f.$ Thus we get

\begin{equation}
\cD
=\left(\Id-\mathcal W(\ell_4^f)\otimes \dfrac{D \ell_4^f}{DZ^f}\right)^{-1}\dfrac{DZ(\ell^f_4)}{DZ(\ell^i_4)}\left(\Id-\mathcal W(\ell_4^i)\otimes \dfrac{D \ell_4^i}{D Z^i}\right).
\label{eq: formald2}
\end{equation}
Here the term $\frac{DZ(\ell^f_4)}{DZ(\ell^i_4)}$ is the fundamental solution to the variational equation from time $\ell^i_4$ to $\ell_4^f$. It does not give us the correct derivative of the Poincar\'e map since the Poincar\'e sections
are not defined by $\ell_4^{i,f}=$constant (equal time) but by $\{x_4=-\chi/2\}$ (equal space). As a result, different orbits may take different time to travel from one section to the next. The two other terms in $\mathcal D$ corresponding to $dF_1,dF_3$ are used to go from equal space section to equal time section and vice versa, which we call {\bf boundary contributions. }
\section{Variational equations}\label{section: variational}

The next step in the proof is the $C^1$ analysis of the global map. It occupies
sections \ref{section: variational}-\ref{section: switch foci}. We shall work under the assumptions of
Lemma \ref{LmDerGlob}. In particular we will use the estimates of Section \ref{section: equation} and Appendix \ref{section: appendix}.

The plan of the proof of Proposition \ref{Thm: matrices} is the following. Matrices $(I)$, $(III)$ and $(V)$ are treated in Sections
\ref{section: variational}
and \ref{section: boundary}. Namely, in Sections \ref{section: variational} we study the variational equation
while in Section \ref{section: boundary} we estimate the boundary contributions.
Finally in Section \ref{section: switch foci}
we compute matrices $(II)$ and $(IV)$ which describe the change of variables between the Delaunay coordinates
with different centers which are
used to the left and to the right of the line $x=-\frac{\chi}{2}.$

\subsection{Estimates of the coefficients}
\begin{Lm}
We have the following estimates for the RHS of the variational equation under the assumption of Lemma \ref{LmQderT}.
\begin{itemize}
\item[(a)]When $Q_4$ is moving to the right of the section $\left\{x=-\chi/2\right\}$, we have
\begin{equation}
\left[\begin{array}{c|c}
\dfrac{\partial \mathcal U_R}{\partial X}&\dfrac{\partial \mathcal U_R}{\partial Y}\\
\hline
\dfrac{\partial \mathcal V_R}{\partial X}&\dfrac{\partial \mathcal V_R}{\partial Y}
\end{array}\right]=O\left(\begin{array}{cc|c}
\frac{1}{\chi^2}&(\frac{1}{\chi^3})_{1\times 3}&(\frac{1}{\chi^2})_{1\times 2}\\
\frac{1}{\chi}&(\frac{1}{\chi^2})_{1\times 3}&(\frac{1}{\chi})_{1\times 2}\\
(\frac{1}{\chi^2})_{2\times 1}&(\frac{1}{\chi^3})_{2\times 3}&(\frac{1}{\chi^2})_{1\times 2}\\
\hline
(\frac{1}{\chi})_{2\times 1}&(\frac{1}{\chi^3})_{2\times 3}&(\frac{1}{\chi})_{2\times 2}\\
\end{array}
\right)+O\left(\dfrac{\mu}{|Q_4|^2}\right)
\nonumber
\end{equation}
In addition we have for $\xi=\frac{|Q_4|}{\chi}=\frac{|Q_4-Q_2|}{\chi}\in (0,1/2)$
\[\dfrac{\partial \mathcal{V}}{\partial Y}=
\dfrac{1}{\chi}\frac{\xi}{(1-\xi)^3} \left[
\begin{array}{cc}
\dfrac{-L^4\sign(\dot{x}_4)}{(G^2+L^2)}& L^3\\
\dfrac{ -L^5 }{(G^2+L^2)^2}& \dfrac{ L^4\sign(\dot{x}_4)}{(G^2+L^2)}\\
\end{array}
\right]+O\left(\dfrac{\mu}{\chi}+\dfrac{\mu}{|Q_4|^2}\right),
\]
\[\dfrac{\partial\mathcal{V}}{\partial L_3}=
\dfrac{1}{\chi} \frac{\xi}{(1-\xi)^3}\left(\dfrac{ G_4L_4^3  \sign(\dot{x}_4)}{(L_4^2+G_4^2)},\dfrac{G_4L_4^4}{(L^2+G_4^2)^2}\right)^T
+O\left(\dfrac{\mu}{\chi}+\dfrac{\mu}{|Q_4|^2}\right).
\]
\item[(b)] When $Q_4$ is moving to the left of the section $x=-\chi/2$, we have
\begin{equation}O\left[\begin{array}{c|c}
\dfrac{\partial \mathcal U_L}{\partial X}&\dfrac{\partial \mathcal U_L}{\partial Y}\\
\hline
\dfrac{\partial \mathcal V_L}{\partial X}&\dfrac{\partial \mathcal V_L}{\partial Y}
\end{array}\right]=
\left(\begin{array}{cc|c}
\frac{1}{\chi^2}&(\frac{1}{\chi^3})_{1\times 3}&(\frac{\mu}{\chi^2})_{1\times 2}\\
\frac{1}{\chi}&(\frac{1}{\chi^2})_{1\times 3}&(\frac{1}{\chi^2})_{1\times 2}\\
(\frac{1}{\chi^2})_{2\times 1}&(\frac{1}{\chi^3})_{2\times 3}&(\frac{\mu}{\chi^2})_{1\times 2}\\
\hline
(\frac{1}{\chi^2})_{2\times 1}&(\frac{\mu}{\chi^2})_{2\times 3}&(\frac{1}{\chi})_{2\times 2}\\
\end{array}
\right)
\nonumber
\end{equation}\label{Lm: ordervar}
In addition we have for $\xi=\frac{|Q_4-Q_1|}{\chi}\in (0,1/2)$
\[\dfrac{\partial \mathcal{V}}{\partial Y}=-
\dfrac{1}{\chi} \frac{\xi}{(1-\xi)^3}\left[\begin{array}{cc}
L^2 \sign(\dot{x}_4)& L^3\\
-L& - L^2\sign(\dot{x}_4)\\
\end{array}
\right]+O\left(\dfrac{\mu}{\chi}\right).
\]
\end{itemize}
\end{Lm}
\begin{proof}
Before going the the calculations, we remark that the variable $\ell_4$ is treated as the new time hence we do not take partial derivatives with respect to it when deriving the variational equations. We need only $C^0$ dependence on $\ell_4$
in the RHS of both the Hamiltonian equation and the variational equation, which is satisfied even if when the orbits come close to collision. We need to use Lemma \ref{LmSmallu} when taking $G_4,L_4$ partial derivatives for small $\ell_4$ in the left case to show that the first and second order derivatives of $Q_4$ with respect to $G_4,L_4$ are always bounded. 

(a) We estimate the four blocks of the derivative matrix separately. \\
$\bullet${\bf We begin with
$\frac{\partial \mathcal{U}_R}{\partial X}$ part.}

{\it We consider first the partial derivatives of $\ell_3$ since it is the largest
component of $\cU.$} Opening the brackets in the second line of \eqref{eq: 6 equations} we get
\begin{equation}
\label{EqEll3Ell4}
\dfrac{d\ell_3}{d\ell_4}
=-k+\dfrac{1}{L_3^3}W+kL_3^3\dfrac{\partial Q_3}{\partial L_3}\cdot\dfrac{\partial U}{\partial Q_3}+k^2L_3^3\dfrac{\partial Q_4}{\partial L_4}\cdot\dfrac{\partial V}{\partial Q_4}+2kW\dfrac{\partial Q_4}{\partial L_4}\cdot\dfrac{\partial V}{\partial Q_4}+
O\left(\dfrac{1}{\chi^2}+\dfrac{\mu}{|Q_4|^3}\right).
\end{equation}
Note that  by \eqref{eq: W}
\begin{equation}
\label{EqWLong}
\begin{aligned}
W_R&=k_R 3L_3^5\left(\dfrac{1}{|Q_3+(\chi,0)|}+\dfrac{1}{|Q_4+(\chi,0)|}+\dfrac{\mu Q_4\cdot Q_3}{|Q_4|^3}\right)+O\left(\dfrac{\mu}{|Q_4|^3}\right)\\
&=
O\left(\dfrac{1}{\chi}+\dfrac{\mu}{|Q_4|^2}\right)
\end{aligned}
\end{equation}
Observe that the RHS of \eqref{EqEll3Ell4} depends on $L_3$ in three ways. First, in contains several terms of the form $L_3^m.$
Second, $Q_3$ depends on $L_3$ via \eqref{eq: Q_3}. Third, $Q_4$ depends on $L_4$ via \eqref{eq: Q4} and $L_4$ depends on
$L_3$ via \eqref{eq: W}. In particular we need to
consider the contribution to $\frac{\partial }{\partial L_3}\frac{d\ell_3}{d\ell_4}$
coming from\[\dfrac{\partial L_4}{\partial L_3}\dfrac{\partial }{\partial L_4}=\dfrac{\partial L_4}{\partial L_3}\dfrac{\partial Q_4}{\partial L_4}\dfrac{\partial }{\partial Q_4}.\]
By Lemma \ref{LM: 1stder} and equation \eqref{LmPosEqB} we have $\frac{\partial Q_4}{\partial L_4}=O(|Q_4|).$
Therefore the main contribution to (2,1) entry is
 $O\left(\frac{1}{\chi}+\frac{\mu}{|Q_4|^2}\right)$ and it comes from
 $\frac{\partial W_R}{\partial Q_4}\frac{\partial Q_4}{\partial L_4}\frac{\partial L_4}{\partial L_3}$, $W_R\frac{\partial }{\partial L_3}\frac{1}{L_3^3}$ and
 $\frac{\partial L_4}{\partial L_3}\frac{\partial }{\partial L_4}\left(k^2L_3^3\frac{\partial Q_4}{\partial L_4}\cdot\frac{\partial V}{\partial Q_4}\right)$.

{\it For the $(2,2),(2,3),(2,4)$ entries, the computations are similar.}

We need to act $\frac{\partial }{\partial \ell_3},\frac{\partial }{\partial G_3}, \frac{\partial }{\partial g_3}$ on \eqref{EqEll3Ell4}.
\eqref{eq: W} and \eqref{EqWLong} show that
 the contribution coming from $\frac{\partial L_4}{\partial (\ell_3,G_3,g_3)}$ is
 $O\left(\frac{1}{\chi^2}+\frac{\mu}{|Q_4|^2}\right).$ It remains to consider the contribution coming from
 $\frac{\partial Q_3}{\partial (\ell_3,G_3,g_3)}\frac{\partial }{\partial Q_3}.$ Now the bound for
 $(2,2), (2,3)$ and $(2,4)$ entries follows directly from Lemmas ~\ref{Lm: position}, ~\ref{Lm: potential}, ~\ref{Lm: gradient},
 and~\ref{Lm: 2ndder}.

{\it The entries $(i,j),\ i\in\{1,3,4\},\ j\in\{2,3,4\}$ are done together. }

These involve second order derivatives with respect to $\ell_3,G_3,g_3$. The estimate $O(\frac{\mu}{|Q_4|^2})$ in the statement comes from the term $\frac{\mu Q_4\cdot Q_3}{|Q_4|^3}$ in $U_R$. For the term $\frac{1}{|Q_3+(\chi,0)|}=O(1/\chi)$ in $U_R$, each $Q_3$ derivative amounts to improve the estimate by multiplying $1/\chi$. Here we need to take two $Q_3$ derivatives. Moreover, $\frac{\partial Q_3}{\partial (G_3,g_3,\ell_3)},\frac{\partial^2 Q_3}{\partial (G_3,g_3,\ell_3)^2}=O(1)$ due to the periodicity. So we get the estimate in the statement. We point out that the improvement compared to the first column and second row in this block is because that we do not take $L_3$ partial derivative.

{\it Next, consider $(1,1)$ entry.} We need to estimate
\[\dfrac{\partial }{\partial L_3}\left((kL_3^3+W)\dfrac{\partial Q_3}{\partial \ell_3}\cdot\dfrac{\partial U}{\partial Q_3}\left(1+(kL_3^3+W)\dfrac{\partial Q_4}{\partial L_4}\cdot\dfrac{\partial V}{\partial Q_4}\right)\right).\]
Using the Leibniz rule we see that the leading
term comes from $\frac{\partial }{\partial L_3}\left(kL_3^3\frac{\partial Q_3}{\partial \ell_3}\cdot\frac{\partial U}{\partial Q_3}\right)$
and it is of order $O\left(\frac{1}{\chi^2}+\frac{\mu}{|Q_4|^2}\right).$ The estimates for other entries of the
$\frac{\partial \mathcal{U}_R}{\partial X}$ part
are similar to the $(1,1)$ entry. This completes the analysis of $\frac{\partial\mathcal U_R}{\partial X}.$

$\bullet$ {\bf Next, we consider $\frac{\partial \mathcal{V}_R}{\partial Y}.$}

Using the Leibniz rule again we see that the main contribution to the derivatives of $\cV$ comes from differentiating
$\left[\begin{array}{c}
L_3^3\frac{\partial Q_4}{\partial g_4}\cdot\frac{\partial V}{\partial Q_4}\\
-L_3^3\frac{\partial Q_4}{\partial G_4}\cdot\frac{\partial V}{\partial Q_4}
\end{array}
\right]$

{\it Consider the $(5,5)$ entry. }

The main contribution to this entry comes from
\[\dfrac{\partial }{\partial G_4}\left(L_3^3\dfrac{\partial Q_4}{\partial g_4}\cdot\dfrac{\partial V}{\partial Q_4}\right)=
L_3^3 \left(\dfrac{\partial^2 Q_4}{\partial G_4\partial g_4}\cdot \dfrac{\partial V}{\partial Q_4}+\dfrac{\partial Q_4}{\partial g_4}\cdot \dfrac{\partial^2 V}{\partial Q_4^2}\cdot\dfrac{\partial Q_4}{\partial G_4}\right). \]
By Lemmas~\ref{Lm: gradient} and~\ref{Lm: 2ndder} the first term is
$|Q_4|\cdot O\left(\frac{1}{\chi^2}+\frac{\mu}{|Q_4|^3}\right)=O\left( \frac{1}{\chi}+\frac{\mu}{|Q_4|^2}\right)$
and the second term is
$|Q_4|^2\cdot O\left(\frac{1}{\chi^3}+\frac{\mu}{|Q_4|^4}\right)=O\left(\frac{1}{\chi}+\frac{\mu}{|Q_4|^2}\right).$
This gives the desired upper bound of the $(5,5)$ entry. Notice that $O(1/\chi)$ term comes from
$L_3^3\frac{\partial }{\partial G_4}\left(\frac{\partial Q_4}{\partial g_4}\cdot\frac{\partial \tV}{\partial Q_4}\right)$
where
$\tV=-\frac{1}{|Q_4+(\chi, 0)|}.$ Thus we need to find the asymptotics of
\begin{equation}
\label{Key55R}
L_3^3 \dfrac{\partial }{\partial G_4}\left(\dfrac{\frac{\partial Q_4}{\partial g_4} \cdot (Q_4+(\chi,0))}{|Q_4+(\chi,0)|^3}\right) .
\end{equation}
Let $\frac{\partial Q_4}{\partial g_4}=(\ba, \bb).$ Arguing in the same way as in the estimation of \eqref{DotG} we see that $\ba=O(1).$
Accordingly the numerator in \eqref{Key55R} is $O(\chi)$ so if we differentiate the denominator of
\eqref{Key55R} the resulting fraction will be of order $O(\chi)O(\chi^{-3})=O(\chi^{-2}).$ Hence $O(1/\chi)$ term comes from
$ L_3^3
\frac{\frac{\partial }{\partial G_4}
\left(\frac{\partial Q_4}{\partial g_4} \cdot (Q_4+(\chi,0))\right)}{|Q_4+(\chi,0)|^3} . $
The numerator here equals to
\[\dfrac{\partial}{\partial G_4} \left( \dfrac{\partial Q_4}{\partial g_4} \cdot Q_4\right)+
\dfrac{\partial^2 Q_4}{\partial G_4 \partial g_4} \cdot (\chi, 0). \]
The first term vanishes due to Lemma \ref{LM: 1stder}(a) so the main contribution comes from the second term. Using Lemma \ref{Lm: 2ndderivative}
we see that
$(5,5)$ entry equals to
\[\dfrac{L_3^3 L_4^2}{\sqrt{L_4^2+G_4^2}}\dfrac{\chi \sinh u}
{|Q_4+(\chi,0)|^3}
+O\left(\dfrac{\mu}{\chi}+\dfrac{\mu}{|Q_4|^2}\right)
. \]
Recall that $L_3=L_4(1+o(1))$ (due to \eqref{eq: W}) and $\sinh u=\sign (u)\frac{|\ell_4| L_4}{\sqrt{L_4^2+G_4^2}}$
(due to \eqref{eq: hypul}).
Since Lemma \ref{Lm: position} implies that
$|Q_4|=|\ell_4|/L_4^2 (1+o(1))$
we obtain that $O(1/\chi)$-term in $(5,5)$ is asymptotic to
$ \frac{L^4 \sign(u)}{L^2+G^2} \frac{\chi |Q_4|}{(\chi-|Q_4|)^3} .$
Since $u$ and $\dot{x}_4$ have opposite signs we obtain the asymptotics of $O(1/\chi)$-term
claimed in part (a) of the Lemma \ref{Lm: ordervar}.
The analysis of other entries of $\frac{\partial \mathcal{V}_R}{\partial Y}$ is similar.

$\bullet$ {\bf Next, consider the $\frac{\partial \mathcal{U}_R}{\partial Y}$ term.}

The analysis of $(2,5)$ entry is similar to the analysis of $(2,2)$ entry except that
$\frac{\partial}{\partial G_4} \left(k^2 L_3^3 \frac{\partial Q_4}{\partial L_4} \frac{\partial V}{\partial Q_4}\right)$
contains the term
$k^2 L_3^3 \frac{\partial^2 Q_4}{\partial L_4 \partial G_4} \frac{\partial V}{\partial Q_4} $ which is of order
$O(1/\chi)$ due to Lemmas \ref{Lm: 2ndder} and \ref{Lm: 2ndderivative} and this term provides the leading contribution for large $t.$
The analysis of $(2,6)$ is similar to $(2,5).$

The estimate of the remaining entries of $\frac{\partial \mathcal{U}_R}{\partial Y}$ is similar to the analysis of $(1,1)$ entry.

$\bullet${\bf Thus to complete the proof of (a) it remains to consider $\frac{\partial \mathcal{V}_R}{\partial X}$.}
We begin with $(5,1)$ entry.
We need to act by $\frac{\partial}{\partial L_3}+\frac{\partial L_4}{\partial L_3}\frac{\partial}{\partial L_4}$ on
\[(kL_3^3+W)\frac{\partial Q_4}{\partial g_4}\cdot\frac{\partial V}{\partial Q_4}\left(1+(kL_3^3+W)\frac{\partial Q_4}{\partial L_4}
\cdot\frac{\partial V}{\partial Q_4}\right) .\]
The leading term for the estimate of $(5,1)$ comes from
\begin{equation}\nonumber
\begin{aligned}
&\left(\dfrac{\partial }{\partial L_3}+\dfrac{\partial L_4}{\partial L_3}\dfrac{\partial}{\partial L_4}\right)
\left(\dfrac{\partial Q_4}{\partial g_4}\cdot\dfrac{\partial V}{\partial Q_4}\right)\\
&=
\dfrac{\partial L_4}{\partial L_3}\dfrac{\partial }{\partial L_4}\left(\dfrac{\partial Q_4}{\partial g_4}\cdot\dfrac{\partial V}{\partial Q_4}\right)
+O\left(\dfrac{1}{\chi^2}+\dfrac{\mu}{|Q_4|^2}\right)=
O\left(\dfrac{1}{\chi}+\dfrac{\mu}{|Q_4|^2}\right)
.\end{aligned}
\end{equation}
Observe that $O(1/\chi)$ term here comes from $\frac{\partial }{\partial L_4}\left(\frac{\partial Q_4}{\partial g_4}\cdot\frac{\partial V}{\partial Q_4}\right)$
which can be analyzed in the same way as $(5,5)$ term.
The analysis of $(6,1)$ is the same as of $(5,1).$

The $(5,2)$ entry is equal to
$\left(\frac{\partial}{\partial \ell_3}+\frac{\partial L_4}{\partial \ell_3}\frac{\partial}{\partial L_4}\right)
\left[\left(\frac{\partial Q_4}{\partial g_4}\cdot\frac{\partial V}{\partial Q_4}\right)\Gamma\right] $
where
\[\Gamma=kL_3^3+W+k^2L_3^6\dfrac{\partial Q_4}{\partial L_4}\cdot\dfrac{\partial V}{\partial Q_4}+2kL_3^3W\dfrac{\partial Q_4}{\partial L_4}\cdot\dfrac{\partial V}{\partial Q_4}+W^2\dfrac{\partial Q_4}{\partial L_4}\cdot\dfrac{\partial V}{\partial Q_4} .\]
Now the estimate of the $(5,2)$ entry follows from the following estimates
\[ \Gamma=O(1), \quad \left(\dfrac{\partial Q_4}{\partial g_4}\cdot\dfrac{\partial V}{\partial Q_4}\right)=
O\left(\dfrac{1}{\chi^2}+\dfrac{\mu}{|Q_4|^2}\right), \]
\begin{equation}
\begin{aligned}
&\left(\dfrac{\partial}{\partial \ell_3}+\dfrac{\partial L_4}{\partial \ell_3}\dfrac{\partial}{\partial L_4}\right)
\left(\dfrac{\partial Q_4}{\partial g_4}\cdot\dfrac{\partial V}{\partial Q_4}\right)=\dfrac{\partial Q_4}{\partial g_4}\cdot\dfrac{\partial }{\partial \ell_3}\dfrac{\partial V}{\partial Q_4}+\dfrac{\partial L_4}{\partial \ell_3}\dfrac{\partial }{L_4}\left(\dfrac{\partial Q_4}{\partial g_4}\cdot\dfrac{\partial V}{\partial Q_4}\right)\\
&=O\left(\dfrac{\mu}{|Q_4|^2}+\left(\dfrac{1}{\chi^2}+\dfrac{\mu}{|Q_4|^2}\right)\left(\dfrac{1}{\chi}+\dfrac{\mu}{|Q_4|^2}\right)\right)
=O\left(\dfrac{1}{\chi^3}+\dfrac{\mu}{|Q_4|^2}\right),\\
\end{aligned}\nonumber\end{equation}
and
\[\left(\dfrac{\partial}{\partial \ell_3}+\dfrac{\partial L_4}{\partial \ell_3}\dfrac{\partial}{\partial L_4}\right)\Gamma=
O\left(\dfrac{1}{\chi^2}+\dfrac{\mu}{|Q_4|^2}\right) .\]
The remaining entries of $\frac{\partial \mathcal{V}}{\partial X}$ are similar to the $(5,2)$ entry.
This completes the proof of part (a).

(b)$\bullet$ {\bf The estimate of $\frac{\partial \mathcal V_L}{\partial Y}$
and $\frac{\partial \mathcal U_L}{\partial X}$
are the same as in part (a).}
However, now $|Q_4|$ is of order $\chi$ so $O(\mu/|Q_4|^2)$ is dominated by other terms.
In addition to compute the leading part we need to use part (c) Lemma \ref{Lm: 2ndderivative} rather than part (b).
Moreover, in order to be able to use the formulas of that Lemma we need to shift
the origin to $Q_1$. Therefore the coordinates of $Q_2$ become $(\chi,0)$. Then we have
\begin{equation}\label{EqTwistDiff}\dfrac{\partial \mathcal{V}_L}{\partial Y}=L_3^3
\left[\begin{array}{cc}
\dfrac{\partial^2 Q_4}{\partial G\partial g}\cdot \dfrac{(-\chi,0)}{|Q_4-(\chi,0)|^3}& \dfrac{\partial^2 Q_4}{\partial g^2}\cdot \dfrac{(-\chi,0)}{|Q_4-(\chi,0)|^3}\\
-\dfrac{\partial^2 Q_4}{\partial G^2}\cdot \dfrac{(-\chi,0)}{|Q_4-(\chi,0)|^3}&-\dfrac{\partial^2 Q_4}{\partial G\partial g}\cdot \dfrac{(-\chi,0)}{|Q_4-(\chi,0)|^3}\\
\end{array}
\right]+O\left(\dfrac{\mu}{\chi}\right).
\end{equation}
Now the asymptotic expression of $\frac{\partial \mathcal V_L}{\partial Y}$  follows directly from Lemma \ref{Lm: 2ndderivative}(c). We point out that the ``$-$" sign in front of the matrices of $\frac{\partial V}{\partial Y}$ and $\frac{\partial V}{\partial L_3}$
comes from the fact that the new time $\ell_4$ that we are using satisfies
$\frac{d\ell_4}{dt}=-\frac{1}{L_4^3}+o(1)$ as $\mu\to 0,\chi\to\infty$.

$\bullet$ {\bf Next, we consider the $\frac{\partial \mathcal{U}_L}{\partial Y}$ term.}


{\it First consider $(1,5)$. } We need to find $G_4$ derivative of
\[\left[\dfrac{\partial Q_3}{\partial \ell_3}\cdot\dfrac{\partial U}{\partial Q_3}\right](kL_3^3+W)\left(1+(kL_3^3+W)
\dfrac{\partial Q_4}{\partial L_4}\cdot\dfrac{\partial V}{\partial Q_4}\right).\]
Differentiating the first factor we get using
Lemma~\ref{Lm: 2ndder}
\begin{equation}
\label{Eq15LeftSim}
\dfrac{\partial }{\partial G_4}\left(\dfrac{\partial Q_3}{\partial \ell_3}\cdot\dfrac{\partial U}{\partial Q_3}\right)=
\dfrac{\partial Q_3}{\partial \ell_3}\cdot\dfrac{\partial^2 U}{\partial Q_3\partial Q_4}\dfrac{\partial Q_4}{\partial G_4}=
O\left(\dfrac{\mu}{\chi^2}\right).
\end{equation}
When we differentiate the product of the remaining factors then the main contribution comes from
\begin{equation}
\label{Eq15LeftKey}
\dfrac{\partial}{\partial G_4}\left( \dfrac{\partial Q_4}{\partial L_4}\cdot\dfrac{\partial V}{\partial Q_4}\right)=
\dfrac{\partial^2 Q_4}{\partial L_4 \partial G_4}\cdot\dfrac{\partial V}{\partial Q_4}+
\dfrac{\partial Q_4}{\partial L_4}\cdot \dfrac{\partial}{\partial G_4} \left(\dfrac{\partial V}{\partial Q_4}\right) .
\end{equation}
To bound the last expression we use Lemma~\ref{Lm: 2ndderivative}. Namely, the second derivative $\frac{\partial^2 Q_4}{\partial G_4\partial L_4}=O(1)+\ell_4(0,1),$ is almost vertical and $\frac{\partial V_L}{\partial Q_4}=\frac{Q_4}{|Q_4|^3}+\frac{\mu(Q_4-Q_3)}{|Q_4-Q_3|^3}$
is almost horizontal.
This shows that $\frac{\partial^2 Q_4}{\partial G_4\partial L_4}\cdot\frac{\partial V}{\partial Q_4}=\frac{1}{\chi^2}$.
The main contribution to the second summand in \eqref{Eq15LeftKey} comes from
$\frac{\partial}{\partial G_4} \left(\nabla \left(\frac{1}{Q_4}\right)\right).$
Using Lemma~\ref{LM: 1stder}, we get
\[\dfrac{\partial Q_4}{\partial L_4}\cdot \dfrac{\partial}{\partial G_4} \left(\nabla \left(\dfrac{1}{Q_4}\right)\right)=
(\ell_4(1,0)+O(1))\left(\dfrac{-\Id}{|Q_4|^3}+3\dfrac{Q_4\otimes Q_4}{|Q_4|^5}\right)(\ell_4(0,1)+O(1))=\dfrac{1}{\chi^2}.\]
 Since
$\frac{\partial Q_3}{\partial \ell_3}\cdot\frac{\partial U}{\partial Q_3}=O(1/\chi^2)$
we get the required estimate for $(1,5)$ entry.

The estimates of other $\frac{\partial \mathcal{U}_L}{\partial Y}$ terms are similar to the
estimate of $(1,5)$ entry, except for $(2,5)$ and $(2,6)$ entries which are different because
$\frac{d\ell_3}{d\ell_4}$ is larger than the other coordinates of $\cU.$

{\it Now consider $(2,5)$ entry. }We need to compute
\begin{equation}
\begin{aligned}
&-\dfrac{\partial }{\partial G_4}\left((kL_3^3+W)(\dfrac{1}{L_3^3}+\dfrac{\partial Q_3}{\partial L_3}\cdot\dfrac{\partial U}{\partial Q_3}) \left(1+(kL_3^3+W)\dfrac{\partial Q_4}{\partial L_4}\cdot\dfrac{\partial V}{\partial Q_4}\right)\right)\\
&=-\dfrac{\partial }{\partial G_4}\left(k+\dfrac{1}{L_3^3}W+kL_3^3\dfrac{\partial Q_3}{\partial L_3}\cdot\dfrac{\partial U}{\partial Q_3}+k^2L_3^3\dfrac{\partial Q_4}{\partial L_4}\cdot\dfrac{\partial V}{\partial Q_4}+2kW\dfrac{\partial Q_4}{\partial L_4}\cdot\dfrac{\partial V}{\partial Q_4}+\dfrac{1}{\chi^3}\right)\\
&=0+\dfrac{1}{\chi^2}+\dfrac{\mu}{\chi^2}+\dfrac{1}{\chi^2}+\dfrac{1}{\chi^3}+0=O\left(\dfrac{1}{\chi^2}\right)
\end{aligned}\end{equation}
where the analysis of the leading terms is similar to \eqref{Eq15LeftSim}, \eqref{Eq15LeftKey}.

$\bullet$ {\bf Finally, we consider $\frac{\partial \mathcal{V}_L}{\partial X}$.}

{\it We begin with $(5,1)$ entry. }We need to compute
\[ \left[\dfrac{\partial }{\partial L_3}+\dfrac{\partial L_4 }{\partial L_3}\dfrac{\partial }{\partial L_4}\right]
\left(\left(\dfrac{\partial Q_4}{\partial g_4}\cdot\dfrac{\partial V}{\partial Q_4}\right)\Gamma\right) \]
where
\[\Gamma=
kL_3^3+W+k^2L_3^6\dfrac{\partial Q_4}{\partial L_4}\cdot\dfrac{\partial V}{\partial Q_4}+
2kL_3^3W\dfrac{\partial Q_4}{\partial L_4}\cdot\dfrac{\partial V}{\partial Q_4}+W^2\dfrac{\partial Q_4}{\partial L_4}\cdot\dfrac{\partial V}{\partial Q_4}. \]
The main contribution to
$\left[\frac{\partial }{\partial L_3}+\frac{\partial L_4 }{\partial L_3}\frac{\partial }{\partial L_4}\right]
\left(\frac{\partial Q_4}{\partial g_4}\cdot\frac{\partial V}{\partial Q_4}\right)$
comes from
\[\dfrac{\partial L_4}{\partial L_3}\dfrac{\partial }{\partial L_4}\left(\dfrac{\partial Q_4}{\partial g_4}\cdot\dfrac{\partial V}{\partial Q_4}\right)=\dfrac{\partial L_4}{\partial L_3}\dfrac{\partial^2 Q_4}{\partial L_4\partial g_4}\cdot\dfrac{\partial V}{\partial Q_4}+\dfrac{\partial L_4}{\partial L_3}\dfrac{\partial Q_4}{\partial g_4}\cdot\dfrac{\partial^2 V}{\partial Q_4^2}\dfrac{\partial Q_4}{\partial L_4}.\]
The two summands above can be estimated by $O(1/\chi^2)$ by the argument used to bound \eqref{Eq15LeftKey}.
Next a direct calculation shows that
$$\Gamma=O(1), \quad \left[\dfrac{\partial }{\partial L_3}+\dfrac{\partial L_4 }{\partial L_3}\dfrac{\partial }{\partial L_4}\right]\Gamma=O(1)$$
while
$\left(\frac{\partial Q_4}{\partial g_4}\cdot\frac{\partial V}{\partial Q_4}\right)=O(1/\chi^2)$ by Lemma \ref{Lm: gradient}.
This gives the required bound for the $(5,1)$ entry. The bound for the $(6,1)$ entry is similar.


{\it Next, consider $(5,2)$. } It equals to
\[ \left[\dfrac{\partial }{\partial \ell_3}+\dfrac{\partial L_4 }{\partial \ell_3}\dfrac{\partial }{\partial L_4}\right]
\left(\left(\dfrac{\partial Q_4}{\partial g_4}\cdot\dfrac{\partial V}{\partial Q_4}\right)\Gamma\right) .\]
The main contribution to
$\left[\frac{\partial }{\partial \ell_3}+\frac{\partial L_4 }{\partial \ell_3}\frac{\partial }{\partial L_4}\right]
\left(\frac{\partial Q_4}{\partial g_4}\cdot\frac{\partial V}{\partial Q_4}\right)$
comes from \\
$\frac{\partial }{\partial \ell_3}
\left(\frac{\partial Q_4}{\partial g_4}\cdot\nabla\left(\frac{\mu}{|Q_4-Q_3|}\right)\right)=O\left(\frac{\mu}{\chi^2}\right).$
On the other hand the main contribution to
$\left[\frac{\partial }{\partial \ell_3}+\frac{\partial L_4 }{\partial \ell_3}\frac{\partial }{\partial L_4}\right]\Gamma$ comes
from $\frac{\partial W}{\partial \ell_3}=O\left(\frac{1}{\chi^2}\right).$
Combining this with $C^0$ bounds mentioned used in the analysis of $(5,1)$ we obtain the required estimate on the $(5,2)$ entry.
The remaining entries of $\frac{\partial \mathcal{V}_L}{\partial X}$ are similar to $(5,2).$
\end{proof}

\subsection{Estimates of the solutions}
\label{SSSol}
We integrate the variational equations to get the $\frac{\partial (X,Y)(\ell_4^f)}{\partial (X,Y)(\ell_4^i)}$ in equation~(\ref{eq: formald2}).

\begin{Lm}
Under the hypothesis of Lemma~\ref{LmQderT}
the following estimates are valid  as $1/\chi\ll\mu\to0$ 
\begin{itemize}
\item[(a)] For maps $(I)$ and $(V)$,\begin{equation}\begin{aligned}
&\dfrac{\partial (X,Y)(\ell_4^f)}{\partial (X,Y)(\ell_4^i)}
=\mathrm{Id}+O\left[\begin{array}{cc|c}
\mu& (\mu)_{1\times 3}&(\mu)_{1\times 2}\\
1& (\mu)_{1\times 3}&(1)_{1\times 2}\\
(\mu)_{2\times 1}& (\mu)_{2\times 3}&(\mu)_{2\times 2}\\
\hline
(1)_{2\times 1}&(\mu)_{2\times 3}&(1)_{2\times 2}\\
\end{array}\right].\end{aligned}\label{eq: fundIV}
\end{equation}
\item[(b)] For map $(III)$,
\begin{equation}\begin{aligned}
&\dfrac{\partial (X,Y)(\ell_4^f)}{\partial (X,Y)(\ell_4^i)}
=\mathrm{Id}+O\left[\begin{array}{cc|c}
\frac{1}{\chi}& (\frac{1}{\chi^2})_{1\times 3}&(\frac{\mu}{\chi})_{1\times 2}\\
1& (\frac{1}{\chi})_{1\times 3}&(\frac{1}{\chi})_{1\times 2}\\
(\frac{1}{\chi})_{2\times 1}&(\frac{1}{\chi^2})_{2\times 3}&(\frac{\mu}{\chi})_{2\times 2}\\
\hline
(\frac{1}{\chi})_{1\times 2}&(\frac{\mu}{\chi})_{2\times 3}&(1)_{2\times 2}
\end{array}\right].\end{aligned}\label{eq: fundIII}
\end{equation}
\item[(c)] $\frac{\partial Y(\ell_4^f)}{\partial Y(\ell_4^i)}$ and $\frac{\partial Y}{\partial L_3}$ have the same asymptotics as item (b) of Proposition \ref{PrExact}. 

\label{Lm: variation}
\end{itemize}
\end{Lm}

Parts (a) and (b) of  this lemma claim that we can integrate the estimates of Lemma \ref{Lm: ordervar} over $\ell_4$-interval of
size $O(\chi)$. 

\begin{proof}

We use the following convention. For two matrices $M_1,M_2$, by $M_1\leq M_2$, we mean the inequality for each corresponding matrix entries. Similarly, the notation $|M_1|$ means to take the absolute value in each entry of $M_1$. We use the following version of Gronwall inequality, which can be proven by 
either comparing the series obtained from Picard iterations (see below) or by applying standard comparison theorem
for the ODEs.

\blm
\label{LmLinGronwall}
Consider two linear systems
$X_1'=M_1(t) X_1$ and $X_2'=M_2(t)X_2$. Suppose that $|M_1(t)|\leq M_2(t).$ 
Then the corresponding fundamental solutions satisfy componentwise inequalities
$$\left|\Phi_1(t)\right|\leq \Phi_2(t)$$ for all $t\geq 0.$ 
\elm

Let us consider part (b) first, which is easier since the estimate (b) in Lemma \ref{Lm: ordervar} 
does not depend on $\ell_4$. Consider ODE system $X'(t)=KX$ where 
$K$ is the matrix in (b) of Lemma \ref{Lm: ordervar}.
It can be verified by straightforward computation that $(K\chi)^2\leq C K\chi$ where
$C$ is a constant independent of $\chi.$ 
thus $(K\chi)^n\leq C^n K\chi$, so we get $X(\chi)\leq \Id+e^{C}K\chi$.
Now part (b) follows from Lemma \ref{LmLinGronwall}.

Next, we work on part (a). After a rescaling $\ell_4=\chi t/2$, $t\in (0,1)$ we compare the variational equation with the
ODE 
\begin{equation}
\label{PartAComparison}
X'=K(t)X 
\end{equation}
where 
\begin{equation}
\label{PartAComparison2}
K(t):=cA\chi+\frac{c\mu\chi}{(\chi t/2)^2+1}\mathbf 1
\end{equation}
is an upper bound for $\chi$ times the estimate of Lemma \ref{Lm: ordervar}(a), 
$c$ is a large positive constant, $A$ is the constant matrix and $\mathbf 1$ is the matrix whose entries are all $1$'s. 
We can verify in the same way as the proof of part (b) that 
\begin{equation}
\label{PartBIneq}
\chi^2 A^2\leq C \chi A \quad\text{and}\quad |A|\leq \frac{C}{\chi}\mathbf{1}
\end{equation}


By Lemma \ref{LmLinGronwall}, it is enough to show that the upper bound of the fundamental solution $X(1)$ of \eqref{PartAComparison} is given by the estimate in part (a).

Solving \eqref{PartAComparison} by Picard iteration we get 
$$X(t)=\Id+\int_0^tK(s)X(s)\,ds=\Id+\int_0^tK(s)\,ds+\int_0^tK(s_1)\int_0^{s_1}K(s_0)\,ds_0\,ds_1+\cdots.$$

The terms which do not contain $\mu$ sum to $e^{\chi A}$ which is $\Id+O(\chi A)$ by the same argument
as in part (a). We claim that the remaining terms sum to $O(\mu).$ To this end let
$k(t)=\tC \max(1, \frac{\chi}{\chi^2 t^2+1}).$ 
By \eqref{PartAComparison2} and \eqref{PartBIneq} the contribution of the terms containing $\mu$
is less that $\mu Y(1)$ where $Y$ is the fundamental solution
of $\dot{Y}=k(t) Y.$ Since $\int_0^1 k(t) dt=O(1)$ we have $Y(t)=O(1)$ as claimed.

To prove part (c) we need to find the asymptotics of $\bbV.$ Consider map (I) first. $\bbV$ satisfies
$ \bbV'=\dfrac{\partial \cV}{\partial Y}\bbV. $
By already established part (a) $\bbV=O(1)$ so the above equation can be rewritten as
\[\bbV'= \dfrac{\xi L^2}{\chi (1-\xi)^3} A
 \bbV+O\left(\dfrac{\mu}{\ell_4^2+1}+\dfrac{\mu}{\chi}\right).
\]
where
$A=\left[
\begin{array}{cc}
\frac{L^2}{(G^2+L^2)}& L\\
-\frac{L^3 }{(G^2+L^2)^2}&- \frac{L^2}{(G^2+L^2)}\\
\end{array}
\right].$ Now Gronwall Lemma gives
$\bbV\approx \tbbV$ where
$\tbbV$ is the fundamental solution of
$\tbbV'= \frac{\xi L^2}{\chi (1-\xi)^3} A
 \tbbV. $ Using $\xi$ as the independent variable we get
 $\frac{d\tbbV}{d\xi}=-\frac{\xi}{(1-\xi)^3} A \tbbV.$
 Note that $\xi(\ell_4^i)=o(1),$ $\xi(\ell_4^f)=\frac{1}{2}+o(1).$
Making a further time change $d\tau=\frac{\xi d\xi}{(1-\xi)^3}$
we obtain the constant coefficient linear equation
$\frac{d\tbbV}{d\tau}=-A\tbbV.$ Observe that Tr$(A)=$det$(A)=0$
and so $A^2=0.$ Therefore
\begin{equation}
\tbbV(\sigma, \tau)=\Id-(\tau-\sigma) A.
\label{ExpNil}
\end{equation}
Since
$\tau=\frac{\xi^2}{2(1-\xi)^2}$ we have $\tau(0)=0,$ $\tau\left(\frac{1}{2}\right)=\frac{1}{2}.$
Plugging this into \eqref{ExpNil} we get the claimed asymptotics for map (I). The analysis of map (V) is similar. To analyze map (III)
we split
\[\dfrac{\partial Y(\ell_4^f)}{\partial Y(\ell_4^i)}=
\dfrac{\partial Y(\ell_4^f)}{\partial Y(\ell_4^m)}\dfrac{\partial Y(\ell_4^m)}{\partial Y(\ell_4^i)}
\]
where $\ell_4^m=\frac{\ell_4^i+\ell_4^f}{2}.$ Using the argument presented above we obtain
\[\dfrac{\partial Y(\ell_4^m)}{\partial Y(\ell_4^i)}=
\left[\begin{array}{cc} \frac{3}{2} & -\frac{L}{2} \\ & \\
                                  \frac{1}{2L} & \frac{1}{2}
                                  \end{array}\right], \quad
                                  \dfrac{\partial Y(\ell_4^f)}{\partial Y(\ell_4^m)}=
\left[\begin{array}{cc} \frac{1}{2} & -\frac{L}{2} \\ & \\
                                  \frac{1}{2L} & \frac{3}{2}
                                  \end{array}\right]. \]
Multiplying the above matrices we obtain the required asymptotics for map (III).

Next using the same argument as in analysis of
$\frac{\partial Y(\ell_4^f)}{\partial Y(\ell_4^i)}$ we obtain
$\frac{\partial Y}{\partial L_3}\approx \bbW$ where
\[\bbW'=\dfrac{\xi L^2}{\chi (1-\xi)^3}
\left[A
 \bbW+\left(-\dfrac{G L}{(L^2+G^2)},\dfrac{G L^2}{(L^2+G^2)^2}\right)^T\right] .\]
In terms of the new time this equation reads
\[\dfrac{d\bbW}{d\tau}=-
\left[A
 \bbW+\left(-\dfrac{G L}{(L^2+G^2)},\dfrac{G L^2}{(L^2+G^2)^2}\right)^T\right] .\]
Solving this equation using \eqref{ExpNil} and initial condition $(0,0)^T$, we obtain the asymptotics of $\frac{\partial Y}{\partial L_3}.$
\end{proof}

\section{Boundary contributions and the proof of Proposition~\ref{Thm: matrices}}
\label{section: boundary}
According to \eqref{eq: formald2}
we need to work out the boundary contributions in order to complete the proof of Proposition~\ref{Thm: matrices}.

\subsection{Dependence of $\ell_4$ on variables $(X,Y)$}\label{SSell_4}
To use the formula~\eqref{eq: formald2}  we need to work out
$(\mathcal{U},\mathcal{V})(\ell_4^i)\otimes\frac{\partial\ell_4^i}{\partial (X,Y)^i}$ and $(\mathcal{U},\mathcal{V})(\ell_4^f)\otimes\frac{\partial\ell_4^f}{\partial (X,Y)^f}$.
Consider $x_4$ component of $Q_4$ (see equation~(\ref{eq: Q4})).
\[x_4=-\cos g_4 (L_4^2\sinh u_4-e_4)+\sin g_4(L_4G_4\cosh u_4).\]
For fixed $x_4=-\chi/2$ or $-2$, we can solve for $\ell_4$ as a function of $L_4,G_4,g_4$. From the calculations in the Appendix~\ref{subsection: hyp}, Lemma~\ref{LM: 1stder}, and the implicit function theorem, we get
\begin{equation}
\begin{aligned}
&\mathrm{for\ the\ section\ } x_4=-\chi/2,\ \left(\dfrac{\partial \ell_4}{\partial L_4}, \dfrac{\partial \ell_4}{\partial G_4}, \dfrac{\partial \ell_4}{\partial g_4}\right)\Big|_{x_4=-\chi/2}=
(O(\chi), O(1), O(1)),\\
&\mathrm{for\ the\ section\ } x_4=-2,\ \left(\dfrac{\partial \ell_4}{\partial L_4}, \dfrac{\partial \ell_4}{\partial G_4}, \dfrac{\partial \ell_4}{\partial g_4}\right)\Big|_{x_4=-2}=
(O(1), O(1), O(1)).
\end{aligned}\nonumber\end{equation}
Explicitly, the $O(\chi)$ term is \begin{equation}\label{Eqell4/L4} \frac{\partial \ell_4}{\partial L_4}=-\frac{\partial x_4}{\partial L_4}/\frac{\partial x_4}{\partial \ell_4}= \frac{\sinh u(2\sqrt{L^2_4+G_4^2})}{\sign(u) L^3} +O(1)\end{equation} using Lemma \ref{LM: 1stder} and \ref{Lm: tilt}. This shows that the $O(\chi)$ term has always positive coefficient. 
Using equation~(\ref{eq: W}) which relates $L_4$ to $L_3$, we obtain for the section $\left\{x_4=-\chi/2\right\}$,
\begin{equation}
\begin{aligned}
&\dfrac{\partial \ell_4}{\partial (X,Y)}\Big|_{x_4=-\chi/2}=
(O(\chi), O(1/\chi),O(1/\chi),O(1/\chi),O(1), O(1)),\\
&(\mathcal{U},\mathcal{V})\Big|_{x_4=-\chi/2}=(O(1/\chi^2),-1+O(1/\chi), O(1/\chi^2)_{1\times 4})^T,\\
\end{aligned}\label{eq: l/XY1}
\end{equation}
For the section $\{x_4=-2\}$,
\begin{equation}
\begin{aligned}
&\dfrac{\partial \ell_4}{\partial L_3}\Big|_{x_4=-2}=
(O(1), O(\mu),O(\mu),O(\mu),O(1), O(1)),\\
&(\mathcal{U},\mathcal{V})\Big|_{x_4=-2}=(0,-1,0,0,0,0)^T+O(\mu).
\label{eq: l/XY2}
\end{aligned}\end{equation}
The matrix $(\mathcal{U},\mathcal{V})\otimes
\frac{\partial \ell_4}{\partial(X,Y)}\Big|_{x_4=-\chi/2}$ has rank $1$ and the only nonzero eigenvalue is $O(1/\chi)$, and $(\mathcal{U},\mathcal{V})\otimes\frac{\partial \ell_4}{\partial(X,Y)}\Big|_{x_4=-2}$ has rank $1$ and the only nonzero eigenvalue is $O(\mu)$. So the inversion appearing in~\eqref{eq: formald2} is valid.

\subsection{Asymptotics of matrices $(I),(III),(V)$ from the Proposition~\ref{Thm: matrices}}
Here we complete the computations of matrices (I), (III) and (V).

\textbf{The boundary contribution to $(I)$.}
In this case, $\ell_4^i$ stands for the section $\left\{x_4=-2\right\}$ and $\ell_4^f$ stands for the section $\left\{x_4=-\chi/2\right\}$. So we use equation~(\ref{eq: l/XY2}) to form $(\mathcal{U},\mathcal{V})(\ell_4^i)\otimes\frac{\partial\ell_4^i}{\partial (X,Y)^i}$ in equation~\eqref{eq: formald2} and equation~(\ref{eq: l/XY1}) to form $(\mathcal{U},\mathcal{V})(\ell_4^f)\otimes\frac{\partial\ell_4^f}{\partial (X,Y)^f}$.
We have
\begin{equation}
\label{eq: lf}
\begin{aligned}
&\left(\text{Id}-(\mathcal{U}, \mathcal V)(\ell_4^f)\otimes \dfrac{\partial \ell_4^i}{\partial (X,Y)^i}\right)^{-1}=\text{Id}+\sum_{k=1}^\infty\left((\mathcal{U}, \mathcal V)(\ell_4^f)\otimes \dfrac{\partial \ell_4^i}{\partial (X,Y)^i}\right)^k\\
&=\text{Id}+\left((\mathcal{U}, \mathcal V)(\ell_4^f)\otimes \dfrac{\partial \ell_4^i}{\partial (X,Y)^i}\right)\sum_{k=0}^\infty\left(\dfrac{\partial \ell_4^i}{\partial (X,Y)^i}\cdot (\mathcal{U}, \mathcal V)(\ell_4^f)\right)^k\\
&=\text{Id}+\left((\mathcal{U}, \mathcal V)(\ell_4^f)\otimes \dfrac{\partial \ell_4^i}{\partial (X,Y)^i}\right)(1+O(1/\chi)).
\end{aligned}\end{equation}
Now we use equation~\eqref{eq: formald2} and Lemma~\ref{Lm: variation} to obtain the asymptotics of the matrix $(I)$ stated in Proposition~\ref{Thm: matrices}.

\textbf{The boundary contribution to $(III)$}

This time we use equation~(\ref{eq: l/XY1}) to form both $(\mathcal{U},\mathcal{V})(\ell_4^i)\otimes\frac{\partial\ell_4^i}{\partial (X,Y)^i}$ and $(\mathcal{U},\mathcal{V})(\ell_4^f)\otimes\frac{\partial\ell_4^f}{\partial (X,Y)^f}$ in equation~\eqref{eq: formald2}.\\
The matrix $\left(\mathrm{Id}-(\mathcal{U},\mathcal{V})(\ell_4^f)\otimes\frac{\partial\ell_4^f}{\partial (X,Y)^f}\right)^{-1}$ has the same form as~(\ref{eq: lf}).
Now we use equation~\eqref{eq: formald2} and Lemma~\ref{Lm: variation} to obtain the asymptotics of the matrix $(III)$ stated in Proposition~\ref{Thm: matrices}.

\textbf{The boundary contribution to $(V)$}

This time we use equation~(\ref{eq: l/XY1}) to form $(\mathcal{U},\mathcal{V})(\ell_4^i)\otimes\frac{\partial\ell_4^i}{\partial (X,Y)^i}$ and equation~(\ref{eq: l/XY2}) to form $(\mathcal{U},\mathcal{V})(\ell_4^f)\otimes\frac{\partial\ell_4^f}{\partial (X,Y)^f}$ in equation~\eqref{eq: formald2}.\\
The matrix $\left(\mathrm{Id}-(\mathcal{U},\mathcal{V})(\ell_4^f)\otimes\dfrac{\partial\ell_4^f}{\partial (X,Y)^f}\right)^{-1}=\mathrm{Id}-(\mathcal{U},\mathcal{V})(\ell_4^f)\otimes\dfrac{\partial\ell_4^f}{\partial (X,Y)^f}(1+O(\mu))$.
Now we use equation~\eqref{eq: formald2} and Lemma~\ref{Lm: variation} to obtain the asymptotics of the matrix $(V)$ stated in Proposition~\ref{Thm: matrices}.

Now we are ready to finish the proof of Proposition \ref{PrExact}.
\begin{proof}[Proof of Proposition \ref{PrExact}]
The matrices $(I),(III),(V)$ are obtained by multiplying the solution to the variational equations (Lemma \ref{Lm: variation}) and the boundary contributions according to \eqref{eq: formald2}. By explicit calculation it can be verified that the $O(\chi)$ terms, i.e. the $(2,1)$ entries of the $(I),(III),(V)$ come from the $O(\chi)$ term in the boundary contribution, i.e. the $\frac{d\ell_4}{dL_3}$ term, which is always positive (see \eqref{Eqell4/L4} in Section \ref{SSell_4}).  This finishes the proof of part (a). Again explicit calculation shows that the estimate of part (b) comes mainly from the solution to the variational equation (Lemma \ref{Lm: variation}).
\end{proof}
\section{Switching foci}\label{section: switch foci}
Recall that we treat the motion of $Q_4$ as a Kepler motion focused at $Q_2$ when it is moving to the right of the section $\left\{x=-\chi/2\right\}$ and treat
it as a Kepler motion focused at $Q_1$ when it is moving to the left of the section $\left\{x=-\chi/2\right\}$. Therefore, we need to make a change of coordinates when $Q_4$ crosses the section $\left\{x_4=-\chi/2\right\}$. These are described by the matrices $(II)$ and $(IV)$. Under this coordinate change
the $Q_3$ part of the Delaunay variables does not change. The change of $G_4$ is given by the difference of angular momentums w.r.t. different reference points ($Q_1$ or $Q_2$). To handle it
we introduce an auxiliary variable $v_{4y}$-the $y$ component of the velocity of $Q_4.$
Relating $g_4$ with respect to the different reference points to $v_{4y}$ we complete the computation.

\subsection{From the right to the left}\label{subsection: r2l} We have
$$(II)=\dfrac{\partial(L_3, \ell_3, G_3, g_3, G_{4L},g_{4L})}{\partial (L_3, \ell_3, G_3, g_3, G_{4R},g_{4R})}\Big|_{x_4=-\chi/2}=(iii)(ii)(i)$$
where matrices $(i), (ii)$ and $(iii)$ correspond to the following coordinate changes restricted to the section $\{x_4=-\chi/2\}$.
\[(G,g)_{4R}\stackrel{(i)}{\longrightarrow} (G,v_y)_{4R}\stackrel{(ii)}{\longrightarrow}  (G,v_y)_{4L}\stackrel{(iii)}{\longrightarrow}(G,g)_{4L}.\]

\begin{proof}[Computation of matrices (i) and (iii)(ii) in Proposition \ref{Thm: matrices}]
$(i)$ is given by the relation \[
v_{4y}=
\dfrac{\frac{1}{L_{4R}}\sinh u_{4R}\sin g_{4R}+\frac{G_{4R}}{L_{4R}^2}\cos g_{4R}\cosh u_{4R}}{1-e_{4R}\cosh u_{4R}},
\quad L_{4R}=k_RL_{3}-\dfrac{W_R}{3L_3^2}.\]
where last relation follows from \eqref{eq: W}.
Recall that by Lemma \ref{Lm: tilt} $$g_{4R}=-\arctan\dfrac{G_{4R}}{L_{4R}}+O(1/\chi).$$
In addition \eqref{LRAngularMom} below and the fact that $G_{4R}$ and $G_{4L}$ are $O(1)$
implies $v_{4y}=O(\frac{1}{\chi}).$ Now
the asymptotics of (i) is obtained by a direct computation.
We compute $\frac{d v_{4y}}{d L_3}$ the other derivatives are similar but easier.
We have $\frac{d v_{4y}}{d L_3}=\frac{d v_{4y}}{d L_{4R}} \frac{\partial L_{4R}}{\partial L_3}.$ The second term is
$k_R+O(1/\chi).$ On the other hand
\begin{equation}\nonumber\begin{aligned}
\dfrac{d v_{4y}}{d L_4}&=\dfrac{\frac{\partial}{\partial L_{4R}}
\left(\frac{1}{L_{4R}}\sinh u_{4R}\sin g_{4R}+\frac{G_{4R}}{L_{4R}^2}\cos g_{4R}\cosh u_{4R}\right)}{1-e_{4R}\cosh u_{4R}}\\
&+
v_{4y} \dfrac{\frac{\partial e_{4R}}{\partial L_{4R}} \cosh u_{4R}}{1-e_{4R} \cosh u_{4R}}+
\dfrac{\partial v_{4y}}{\partial \ell_{4R}} \dfrac{\partial \ell_{4R}}{\partial L_{4R}}. \end{aligned}\end{equation}
The main contribution comes from the first term which equals
$\frac{G_{4R}}{L_{4R} (L_{4R}^2+G_{4R}^2)}+O(1/\chi).$
The second term is $O(1/\chi)$ since $v_{4R}=O(1/\chi).$ Next rewriting
\[v_{4y}=
\dfrac{\frac{1}{L_{4R}}\tanh u_{4R}\sin g_{4R}+\frac{G_{4R}}{L_{4R}^2}\cos g_{4R}}{(1/\cosh u_{4R})-e_{4R}}\]
we see
$\dfrac{\partial v_{4y}}{\partial \ell_{4R}} \dfrac{\partial \ell_{4R}}{\partial L_{4R}}=O(1/\chi^2) \times O(\chi)=O(1/\chi)$
since  $\frac{\partial \ell_{4R}}{\partial L_{4R}}=O(\chi)$ by \eqref{eq: l/XY1}.



$(ii)$ is given by
\begin{equation}
\label{LRAngularMom}
G_{4L}=G_{4R}-\chi v_{4y}, 
\end{equation}
which comes from the simple relation $v_4\times (Q_4-Q_1)=v_4\times Q_4-v_4\times Q_1$. Here $G_{4R}$ and $v_{4y}$ are independent variables so the computation of the derivative of (ii) is straightforward.

To compute the derivative of $(iii)$ we use the relation from \eqref{eq: Q4l}
\[
v_{4y}=\dfrac{-\frac{1}{L_{4L}}\sinh u_{4L}\sin g_{4L}-\frac{G_{4L}}{L_{4L}^2}\cos g_{4L}\cosh u_{4L}}{1-e_{4L}\cosh u_{4L}}\]
where $u_L<0$. Arguing the same way as for (i) and using the fact that by
Lemma \ref{Lm: tilt}, $G_L,g_L=O(1/\chi)$, $-\sinh u_L,\cosh u_L\simeq \frac{\ell_{4L}}{e_L}$
we obtain
$\delta v_{4y}=\dfrac{\delta G_{4L}}{k_R^2 L_3^2}-\dfrac{\delta g_{4L}}{k_R L_3}+ HOT.$
Hence
\[ \delta g_{4L}=\dfrac{\delta G_{4L}}{k_R L_3}-k_R L_3 \delta   v_{4y}+HOT=
\dfrac{\delta G_{4R}-\chi \delta v_{4y} }{k_R L_3}+HOT
\]
completing the proof of the lemma.
\end{proof}



\subsection{From the left to the right}
At this step we need to compute
\[(IV)=\dfrac{\partial(L_3, \ell_3, G_3, g_3, G_{4R},g_{4R})}{\partial (L_3, \ell_3, G_3, g_3, G_{4L},g_{4L})}\Big|_{x_4=-\chi/2}=(iii')(ii')(i').\]
where the matrices $(iii'),$ $(ii')$ and $(i')$ correspond to the following changes of variables restricted to the section $\{x_4=-\chi/2\}$.
\[(G,g)_L\stackrel{(i')}{\longrightarrow}
(G,v_{4y})_L\stackrel{(ii')}{\longrightarrow}(G,v_{4y})_R\stackrel{(iii')}{\longrightarrow}(G,g)_R.\]

\begin{proof}[Computation of matrices $(iii')$ and $(ii')(i')$ in Proposition \ref{Thm: matrices}]
$(i')$ is given by
\[ v_{4y}=\dfrac{-\frac{1}{L_{4L}}\sinh u_{4L}\sin g_{4L}-\frac{G_{4L}}{L_{4L}^2}\cos g_{4L}\cosh u_{4L}}{1-e_{4L}\cosh u_{4L}}<0.\]
Here $u_L>0$ and $G_{4L},g_{4L}=O(1/\chi)$.

$(ii')$ is given by
\[G_R=G_L+\chi v_{4yL}.\]
Now the analysis is similar to Subsection \ref{subsection: r2l}. In particular the main contribution to
$[(ii')(i')]_{44}$ comes from
\begin{equation}
\dfrac{\partial (G_{4R},v_{4y})}{\partial (G_{4L}, g_{4L})}=\dfrac{\partial (G_{4R},v_{4y})}{\partial (G_{4L}, v_{4y})}\dfrac{\partial (G_{4L},v_{4y})}{\partial (G_{4L}, g_{4L})}
=\left[\begin{array}{cc} 1& \chi \\
0& 1\end{array}\right]\left[\begin{array}{cc} 1& 0 \\
\frac{1}{L_3^2}+O\left(\frac{1}{\chi}\right)& \frac{1}{ L_3}+O\left(\frac{1}{\chi}\right)\end{array}\right].
\nonumber\end{equation}
The analysis of $(43)$ part is similar.

$(iii')$ is given by
\[G_R=G_R,\quad \quad v_{4y}=\dfrac{\frac{1}{L_{4R}}\sinh u_{4R}\sin g_{4R}+\frac{G_{4R}}{L_{4R}^2}\cos g_{4R}\cosh u_{4R}}{1-e_{4R}\cosh u_{4R}}<0.\]
Here $u_{4R}<0$, and by Lemma \ref{Lm: tilt}, $\tan g_{4R}= \frac{G_{4R}}{L_{4R}}+O(1/\chi).$
To get the asymptotics of the derivative
we first show that similarly to Subsection \ref{subsection: r2l}, we have
\[d v_{4y}=\left(-\dfrac{G_{4R}}{L_{3}(k_R^2L_{3}^2+G_{4R}^2)},0,0,0,\dfrac{1}{k_R^2L_{3}^2+G_{4R}^2}, \dfrac{1}{k_RL_3}\right)+O\left(\dfrac{1}{\chi},\dfrac{1}{\chi^2},\dfrac{1}{\chi^2},\dfrac{1}{\chi^2},\dfrac{1}{\chi},\dfrac{1}{\chi}\right)\]
and then take the inverse.
\end{proof}

\section{Approaching close encounter}\label{section: localappr}
\label{SSMidLand}

In this paper we choose to separate local and global maps by section $\{x_4=-2\}.$ We could have used instead $\{x_{4}=-10\},$
or $\{x_{4}=-100\}.$ Our first goal is to show that the arbitrariness  of this choice does not change the asymptotics of derivative of the local
map (we have already seen in Sections \ref{SSSol}  and \ref{section: boundary}
that it does not in change the asymptotics of the derivative of the global map).

We choose the section $\{|Q_3-Q_4|=\mu^{\kappa}\},\ 1/3<\kappa<1/2$. Outside the section the orbits are treated as perturbed Kepler motions and inside the section the orbits are treated as two body scattering. We shall estimate the errors of this approximation.
We break the orbit into three pieces: from $\{x_4=-2, \dot{x}_4>0\}$ to $\{|Q_3^--Q^-_4|=\mu^{\kappa}\}$, from $\{|Q_3^--Q^-_4|=\mu^{\kappa}\}$ to $\{|Q_3^+-Q^+_4|=\mu^{\kappa}\}$ and from $\{|Q_3^+-Q^+_4|=\mu^{\kappa}\}$ to $\{x_4=-2, \dot{x}_4<0\}$.
Here and below, we use the following convention.

{\bf Convention:} {\it A variable with superscript $-$ (reap. $+$) means its value measured on the section $|Q_3-Q_4|=\mu^\kappa$ before (resp. after) $Q_3, Q_4$ coming to close encounter.  }

In this section
we consider the two pieces of orbit outside the section $\{|Q_3-Q_4|=\mu^{\kappa}\}$. We use
Hamiltonian \eqref{eq: hamloc}. Then we convert the Cartesian coordinates to Delaunay coordinates. The resulting Hamiltonian is
\begin{equation}\label{EqHamLoc} H=-\dfrac{1}{2L_3^2}+\dfrac{1}{2L_4^2}-\dfrac{1}{|Q_4+(\chi,0)|}-\dfrac{1}{|Q_3+(\chi,0)|}-\dfrac{\mu}{|Q_3-Q_4|}.\end{equation}
The difference with the Hamiltonian \eqref{eq: Hr} is that we do not do the Taylor expansion to the potential $-\frac{\mu}{|Q_3-Q_4|}$.

The next lemma and the remark after it tell us that we can neglect those two pieces.

\begin{Lm}
Consider the orbits satisfying the conditions of Lemma \ref{LmDerLoc}.
For the pieces of orbit from $x_4=-2, \dot{x}_4>0$ to $|Q_3^--Q^-_4|=\mu^{\kappa}$ and from $|Q_3^+-Q^+_4|=\mu^{\kappa}$ to $x_4=-2, \dot{x}_4>0$,
the derivative matrices have the following form
in Delaunay coordinates 
\[\dfrac{\partial(X,Y)^-}{\partial(X,Y)|_{x_4=-2}},\ \dfrac{\partial(X,Y)|_{x_4=-2}}{\partial(X,Y)^+}
=\left[\begin{array}{ccc}
1& 0& 0_{1\times 4}\\
O(1)&2&O(1)_{1\times 4}\\
0_{4\times 1}&0_{4\times 1}&\Id_4\\
\end{array}
\right]+O\left(\mu^{1-2\kappa}+\frac{1}{\chi^3}\right).\]
\end{Lm}
\begin{proof}
The proof follows the plan in Section~\ref{subsection: plan}. We first consider the integration of the variational equation.
We treat the orbit as Kepler motions perturbed by $Q_1$ and interaction between $Q_3$ and $Q_4$. Consider first the perturbation coming from
the interaction of $Q_3$ and $Q_4.$
The contribution of this interaction to the variational equation is of order
$\frac{\mu}{|Q_3-Q_4|^3}$.
If we integrate the variational equation along an orbit such that $|Q_3-Q_4|$ goes from $-2$ to $\mu^{\kappa}$,
then the contribution has the order
\begin{equation}\label{eq: pertout}O\left(\int_{-2}^{\mu^{\kappa}}\dfrac{\mu}{|t|^3}dt\right)=O(\mu^{1-2\kappa}).\end{equation}
Similar consideration shows that the perturbation from $Q_1$ is $O(1/\chi^3)$.

On the other hand absence of perturbation, all Delaunay variables except $\ell_3$ are constants of motion. The $(2,1)$ entry is also $o(1)$ following from the same estimate as the $(2,1)$ entry of the matrix in Lemma \ref{Lm: ordervar}.
After integrating over time $O(1)$, the solutions to the variational equations have the form
\[\mathrm{Id}+O(\mu^{1-2\kappa}+1/\chi^3).\]

Next we compute the boundary contributions.
The analysis  is the same as Section~\ref{section: boundary}.
The derivative is given by formula~\eqref{eq: formald2}.
We need to work out $(\mathcal{U},\mathcal{V})(\ell_4^i)\otimes\frac{\partial\ell_4^i}{\partial (X,Y)^i}$
and $(\mathcal{U},\mathcal{V})(\ell_4^f)\otimes\frac{\partial\ell_4^f}{\partial (X,Y)^f}$.
In both cases we have \[(\mathcal{U},\mathcal{V})=(0, 1, 0,0,0,0)+O(\mu^{1-2\kappa}).\]
For the section $\{x_4=-2\}$, we use \eqref{eq: l/XY2}. For the section $\{|Q_3-Q_4|=\mu^{\kappa}\}$, we have
\begin{equation}\dfrac{\partial \ell_4}{\partial(X,Y)}=-\left(\dfrac{\partial |Q_3-Q_4|}{\partial \ell_4}\right)^{-1}\dfrac{\partial |Q_3-Q_4|}{\partial (X,Y)}=-\dfrac{(Q_3-Q_4)\cdot\frac{\partial (Q_3-Q_4)}{\partial (X,Y)}}{(Q_3-Q_4)\cdot\frac{\partial (Q_3-Q_4)}{\partial \ell_4}}\label{eq: l4}\end{equation}
We will prove in Lemma \ref{Lm: landau}(c) below that the angle formed by $Q_3-Q_4$ and $v_3-v_4$ is $O\left(\mu^{1-\kappa}\right)$
(the proof of Lemma \ref{Lm: landau} does not rely on section \ref{SSMidLand}). Thus in \eqref{eq: l4} we can replace $Q_3-Q_4$ by $v_3-v_4$
making $O\left(\mu^{1-\kappa}\right)$ error. Hence
\[\dfrac{\partial \ell_4}{\partial(X,Y)}=\dfrac{(v_3-v_4)\cdot\frac{\partial (Q_3-Q_4)}{\partial (X,Y)}}{(v_3-v_4)\cdot\frac{\partial Q_4}{\partial \ell_4}}+O(\mu^{1-\kappa}),\]
Note that
$\frac{\partial Q_4}{\partial \ell_4}$ is parallel to $v_4.$ Using the information about $v_3$ and $v_4$ from Appendix \ref{subsection: gerver}
we see that $\langle v_3, v_4\rangle\neq \langle v_4, v_4\rangle.$ Therefore
the denominator in \eqref{eq: l4} is bounded away from zero and so
\[\dfrac{\partial\ell_4}{\partial (X,Y)}=(O(1),O(1),O(1),O(1),O(1),O(1)).\]
We also need to make sure the second component $\frac{\partial \ell_4}{\partial \ell_3}$ is not close to 1, so that $\mathrm{Id}-(\mathcal{U},\mathcal{V})(\ell_4^f)\otimes\frac{\partial\ell_4^f}{\partial (X,Y)^f}$ is invertible when $|Q_3-Q_4|=\mu^{\kappa}$ serves as the final section.
In fact, due to \eqref{eq: 6 equations},
$\frac{\partial \ell_4}{\partial \ell_3}\simeq -1$. Using formula~\eqref{eq: formald2}, we get the asymptotics stated in the lemma.
\end{proof}
\begin{Rk}
Using the explicit value of the vectors $\hat\brlin_2,\ \hat\brlin_3,\ w,\ \tilde{w}$ in equations~(\ref{eq: ul}), we find that in the limit $\mu\to 0,\chi\to \infty$
\[\left(\dfrac{\partial(X,Y)^-}{\partial(X,Y)|_{x_4=-2}}\right) \Span\{w,\tilde{w}\}=\Span\{w,\tilde{w}\}\]
and \[\hat\brlin_2\left(\dfrac{\partial(X,Y)|_{x_4=-2}}{\partial(X,Y)^+}\right)=\hat\brlin_2,\quad \hat\brlin_3\left(\dfrac{\partial(X,Y)|_{x_4=-2}}{\partial(X,Y)^+}\right)=\hat\brlin_3\]
This tells us that we can neglect the derivative matrices corresponding to
the pieces of orbit from $x_4=-2, \dot{x}_4>0$ to $|Q_3^--Q^-_4|=\mu^{\kappa}$ and
from $|Q_3^+-Q^+_4|=\mu^{\kappa}$ to $x_4=-2, \dot{x}_4>0$. We thus can identify
$d\Loc$ with
\[\dfrac{\partial (L_3,\ell_3,G_3,g_3,G_4,g_4)^+}{\partial (L_3,\ell_3,G_3,g_3,G_4,g_4)^-}+O(\mu^{1-2\kappa})\]
where $(L_3,\ell_3,G_3,g_3,G_4,g_4)^\pm$ denote
 the Delaunay variables measured on the section $\{|Q_3^\pm-Q_4^\pm|=\mu^{\kappa}\}$.\label{Rk: rk}
\end{Rk}

\section{$C^0$ estimate for the local map}\label{section: localC0}
\label{ScC0Loc}
In Sections \ref{ScC0Loc} and \ref{section: local}
we consider the piece of orbit from
$|Q_3^--Q^-_4|=\mu^{\kappa}$ to $|Q_3^+-Q^+_4|=\mu^{\kappa}$.
Because of Remark~\ref{Rk: rk}, we simply write $d\Loc$ for the derivative for this piece.

\subsection{Justifying Gerver's asymptotics}

It is convenient to use the coordinates of relative motion and the motion of mass center. We define \begin{equation}v_\pm=v_3\pm v_4,\quad Q_\pm=\frac{Q_3\pm Q_4}{2}.\label{eq: defrel}
\end{equation}
Here "-" refers to the relative motion and "+" refers to the center of mass motion.
To study the relative motion, we make the following  rescaling:
\begin{equation} q_-:=Q_-/\mu, \quad \tau:=t/\mu \text{ and }v_- \text{ remains unchanged.}
\label{RelRescale}
\end{equation}
In this way, we zoom in the picture of $Q_3$ and $Q_4$ by a factor $1/\mu$.

Then we have the following lemma.
\begin{Lm}
Inside the sphere $|Q_3-Q_4|=\mu^{\kappa}, 1/3<\kappa<1/2$, as $\mu\to 0$,
\begin{itemize}
\item[(a)]the equation governing the motion of the center of mass is a Kepler motion focused at $Q_2$ perturbed by $O(\mu^{2\kappa})$, 
\begin{equation}
\dot{Q}_+=\dfrac{v_+}{2},\quad \dot{v}_+=-\dfrac{2Q_+}{|Q_+|^3}+O(\mu^{2\kappa}).\label{eq: cm}
\end{equation}
\item[(b)] In the rescaled variables, the equation governing the relative motion is a Kepler motion focused at the origin
perturbed by $O(\mu^{1+2\kappa})$,
\begin{equation}\dfrac{dq_-}{d\tau}=\dfrac{v_-}{2},\quad \dfrac{dv_-}{d\tau}=\dfrac{q_-}{2|q_-|^3}+O(\mu^{1+2\kappa}).\label{eq: rel1}\end{equation}
\end{itemize}
\label{Lm: relcm}
\end{Lm}

\begin{proof}
Note that \eqref{eq: defrel} preserves
the symplectic form.
\[dv_3\wedge dQ_3+dv_4\wedge dQ_4=dv_-\wedge dQ_-+dv_+\wedge dQ_+,\]
The Hamiltonian becomes
\begin{equation}
\label{eq: relcm}
H=\dfrac{|v_-|^2}{4}-\dfrac{\mu}{2|Q_-|}+\dfrac{|v_+|^2}{4}-\dfrac{1}{|Q_++Q_-|}-\dfrac{1}{|Q_+-Q_-|}
\end{equation}
\[\begin{aligned}
-&
\dfrac{1}{|Q_++Q_-+(\chi,0)|}-\dfrac{1}{|Q_+-Q_-+(\chi,0)|}\\
=&\dfrac{|v_-|^2}{4}-\dfrac{\mu}{2|Q_-|}+\dfrac{|v_+|^2}{4}-\dfrac{2}{|Q_+|}+\dfrac{|Q_-|^2}{2|Q_+|^3}-\dfrac{3|Q_+\cdot Q_-|^2}{2|Q_+|^5}+O(\mu^{3\kappa})+O(1/\chi),
\end{aligned}\]

where the $O(\mu^{3\kappa})$ includes the $|Q_-|^3$ and higher order terms. In the following, we drop $O(1/\chi)$ terms
since $1/\chi\ll \mu.$
So the Hamiltonian equations for the motion of the mass center part are
\[\dot{Q}_+=\dfrac{v_+}{2},\quad \dot{v}_+=-\dfrac{2Q_+}{|Q_+|^3}+O(\mu^{2\kappa})\]
proving part (a) of the lemma.


Next, we study the relative motion. From equation~(\ref{eq: relcm}), we get the equations of motion for the center of mass
\[\dot Q_-=\dfrac{v_-}{2},\quad \dot v_-=-\dfrac{\mu Q_-}{2|Q_-|^3}-\dfrac{Q_-}{|Q_+|^3}+\dfrac{3|Q_+\cdot Q_-|Q_+}{|Q_+|^5}+O(\mu^{2\kappa}),\] as $\mu\to 0$, where $O(\mu^{2\kappa})$ includes quadratic and higher order terms of $|Q_-|$.
After making the rescaling according to \eqref{RelRescale}
the equations for the relative motion part become
\begin{equation}\dfrac{dq_-}{d\tau}=\dfrac{v_-}{2},\quad \dfrac{dv_-}{d\tau}=\dfrac{q_-}{2|q_-|^3}+\dfrac{\mu^2q_-}{|Q_+|^3}-\dfrac{3\mu^2|Q_+\cdot q_-|Q_+}{|Q_+|^5}+O(\mu^{1+2\kappa}).\label{eq: rel2} \qedhere \end{equation}
\end{proof}
Lemma~\ref{Lm: relcm} implies the following $C^0$ estimate.

\begin{Lm}
\begin{itemize}
\item[(a)] We have the following equations for orbit crossing the section $\{|Q_3-Q_4|=\mu^\kappa\},\ 1/3<\kappa<1/2$ and $\mu\to 0$,
\begin{equation}
\begin{cases}
&v_3^+=\dfrac{1}{2}R(\al)(v_3^--v_4^-)+\dfrac{1}{2}(v_3^-+v_4^-)+O(\mu^{(1-2\kappa)/3}+\mu^{3\kappa-1})
,\\
&v_4^+=-\dfrac{1}{2}R(\al)(v_3^--v_4^-)+\dfrac{1}{2}(v_3^-+v_4^-)+O(\mu^{(1-2\kappa)/3}+\mu^{3\kappa-1}),\\
&Q_3^++Q_4^+=Q_3^-+Q_4^-+O(\mu^k),\\
&|Q_3^--Q_4^-|=2\mu^{\kappa},\quad |Q_3^+-Q_4^+|=2\mu^{\kappa},\\
\end{cases}\label{eq: landau}
\end{equation}
where $R(\al)=\left[\begin{array}{cc} \cos \al&  -\sin\al\\ \sin \al &\cos\al \end{array}\right]$,
\begin{equation}
\label{Alpha}
\al=\pi+2\arctan\left(\dfrac{G_{in}}{\mu \mathcal{L}_{in}}\right),\quad \mathrm{and}\ \dfrac{1}{4\mathcal{L}^{2}_{in}}=\dfrac{v_-^2}{4}-\dfrac{\mu}{2|Q_-|},\quad G_{in}= 2v_-\times Q_-.
\end{equation}
\item[(b)] We have $1/c\leq \mathcal{L}_{in}\leq c$ for some constant $c>1$. If $\al$ is bounded away from $0$ and $\pi$ by an angle independent of $\mu$ then $G_{in}=O(\mu)$
and the closest distance between $Q_3$ and $Q_4$ is bounded away from zero by $\delta \mu$ and from above by $\mu/\delta$ for some $\delta>0$ independent of $\mu$.
\item[(c)] Also if $\al$ is bounded away from $0$ and $\pi$ by an angle independent of $\mu,$ then
 the angle formed by $Q_-$ and $v_-$ is $O(\mu^{1-\kappa})$.
\item[(d)] The time interval during which the orbit stays in the sphere $|Q_-|=2\mu^{\kappa}$
is \[\Delta t=\mu\Delta\tau =O(\mu^{\kappa}) . \]
\end{itemize}
\label{Lm: landau}
\end{Lm}


\begin{proof}


In the proof, we omit the subscript \emph{in} standing for the variables \emph{inside} the sphere $|Q_-|=2\mu^\kappa$ without leading to confusion.

The idea of the proof is to treat the relative motion as a perturbation of
Kepler motion and then approximate the relative velocities by their asymptotic values for
the Kepler motion.

Fix a small number $\delta_1.$ Below we derive several estimates valid for the first $\delta_1$ units
of time the orbit spends in the set $|Q_-|\leq 2\mu^k.$ We then show that $\Delta t\ll \delta_1.$
It will be convenient to measure time from the orbit enters the set $|Q_-|<2\mu^k.$

Using the formula in the Appendix~\ref{subsection: ellip}, we decompose the Hamiltonian~(\ref{eq: relcm}) as
$H=H_{rel}+\fh(Q_+, v_+)$ where
\[H_{rel}=\dfrac{\mu^2}{4L^2}+\dfrac{|Q_-|^2}{2|Q_+|^2}-\dfrac{|Q_+\cdot Q_-|^2}{2|Q_+|^5}+O(\mu^{3\kappa}), \ \mathrm{as}\ \mu\to 0,\]
and $\fh$ depends only on $Q_+$ and $v_+.$

Note that $H$ is preserved and $\dot{\fh}=O(1)$ which implies that $\frac{L}{\mu}$ is $O(1)$ and
moreover that ratio does not change much for $t\in [0, \delta_1].$
Using the identity
$\frac{\mu^2}{4L^2}=\frac{v_-^2}{4}-\frac{\mu}{2|Q_-|}$ we see that initially
$\frac{L}{\mu}$ is uniformly bounded from below for the orbits from Lemma \ref{LmLMC0}.
Thus there is a constant $\delta_2$ such that for $t\in [0, \delta_1]$
we have $\delta_2\mu\leq L(t) \leq \frac{\mu}{\delta_2}. $

Expressing the Cartesian variables via Delaunay variables (c.f. equation~\eqref{eq: delaunay4}
in Section~\ref{subsection: hyp})
we have up to a rotation by $g$
\begin{equation}
\begin{aligned}
&q_1=\dfrac{1}{\mu} L^{2}(\cosh u-e), \quad q_2=\dfrac{1}{\mu}LG\sinh u,\\
&O(\mu^\kappa)=|Q_-|=\sqrt{|q_1|^2+|q_2|^2}=\dfrac{L^2}{\mu}(e\cosh u-1),
\end{aligned}
\label{eq: delaunayscattering}
\end{equation}
following from the same calculation as \eqref{EqQ4Linear} with $\ell$ and $u$ related by $u-e\sinh u=\ell.$ 
This gives 
\begin{equation}
\label{SectionEll}
\ell=O(\mu^{\kappa-1}).
\end{equation}
 Next
\begin{equation}\label{EqDotl} \dot{\ell}=-\dfrac{\partial H}{\partial L}=-\dfrac{\mu^2}{2L^3}
-\dfrac{\partial H_{rel}}{\partial Q_-}
\dfrac{\partial Q_-}{\partial L}=
-\dfrac{\mu^2}{2L^3}+O(\mu^\kappa)O(\mu^{\kappa-1})=-\dfrac{\mu^2}{2L^3}+O(\mu^{2\kappa-1}). \end{equation}
Since the leading term here is at least
$\frac{\delta_2^3}{2\mu}$ while $\ell=O(\mu^{\kappa-1})$ we obtain part (d) of the lemma.
In particular the estimates derived above are valid for the time the orbits
spend in $|Q_-|\leq 2\mu^\kappa.$
Next, without using any control on $G$ (using the inequality $\left|\frac{\partial e}{\partial G}\right|=\frac{1}{L}\frac{G/L}{e}\leq \frac{1}{L}$), we have
\begin{equation}
\label{DerG}
\dot{G}=\dfrac{\partial H}{\partial Q_-}\dfrac{\partial Q_-}{\partial g}=O(|Q_-|^2)=O(\mu^{2\kappa}), \quad\dot{L}=\dfrac{\partial H}{\partial Q_-}\dfrac{\partial Q_-}{\partial \ell}=O(\mu^{\kappa+1}),
\end{equation}
\begin{equation}
\label{Derg}
 \dot{g}=
\dfrac{\partial H}{\partial Q_-} \dfrac{\partial Q_-}{\partial G}=
O(\mu^\kappa)O(\mu^{\kappa-1})=O(\mu^{2\kappa-1}).
\end{equation}
Integrating over time $\Delta t=O(\mu^{\kappa})$ we get
the oscillation of 
$g$ and $\arctan\frac{G}{L}$ are $O(\mu^{3\kappa-1}).$

We are now ready to derive the first two equations of \eqref{eq: landau}. It is enough to show $v_-^+=R(\al) v_-^-+O(\mu^{(1-2\kappa)/3}+\mu^{3\kappa-1})$ where $\al=2\arctan \frac{G}{L}$ is the angle formed by the two asymptotes of the Kepler hyperbolic motion. We first have $|v_-^+|=|v_-^-|+O(\mu^\kappa)$ using the total energy conservation. It remains to show the expression of $\al$.
Let us denote till the end of the proof $\phi=\arctan\frac{G}{L},$
$\gamma=\frac{(1/2)-\kappa}{3}.$
Recall (see \eqref{eq: delaunay4}) that for $v_-=(p_1,p_2),$
\begin{equation}
\label{EqDelMom}
 p_1=\tp_1 \cos g+\tp_2 \sin g, \quad p_2=-\tp_1 \sin g+\tp_2 \cos g \text{ where}
 \end{equation}
\[ \tp_1=\dfrac{\mu}{L}\dfrac{\sinh u}{1-e\cosh u}, \quad
 \tp_2=\dfrac{ \mu G}{ L^{2}}\dfrac{\cosh u}{1-e\cosh u}. \]
Consider two cases.

(I) $G\leq \mu^{\kappa+\gamma}.$ In this case on the boundary of the sphere $|Q_-|=2\mu^\kappa$ we have
$\ell>\delta_3 \mu^{-\gamma}$ for some constant $\delta_3.$ Thus
{\small\[\dfrac{p_2}{p_1}=\dfrac{\frac{\mu G}{L^2}\cosh u \cos g+\frac{\mu}{L}\sinh u\sin g}{-\frac{\mu G}{L^2}\cosh u \sin g+\frac{\mu}{L}\sinh u\cos g}=\dfrac{\frac{G}{L}\pm \tan g}{\pm 1-\frac{G}{L}\tan g}+O(e^{-2|u|})=
\tan(g\pm \phi)+O(\mu^{2\gamma}). \]}
where the plus sign is taken if $u>0$ and the minus sign is taken if $u<0.$
Since $\arctan$ is globally Lipschitz,
this completes the proof in case (I) by choosing $\al=2\phi$.

(II) $G>\mu^{\kappa+\gamma}.$ In this case $\frac{G}{L}\gg 1$ and so it suffices to show that
$\frac{p_2}{p_1}$ (or $\frac{p_1}{p_2}$) changes little during the time the orbit is inside the sphere.
Consider first the case where $|g^-|>\frac{\pi}{4}$ so $\sin g$ is bounded from below. Then
\[\dfrac{p_2}{p_1}=\cot g+O(\mu^{1-(\kappa+\gamma)})\]
proving the claim of part (a) in that case. The case $|g^-|\leq \frac{\pi}{4}$ is similar but we need to consider
$\frac{p_1}{p_2}.$ This completes the proof in case (II).

Combining equation \eqref{eq: cm} and Lemma \ref{Lm: relcm}(c) we obtain
\begin{equation}
\label{CmFixed}
Q_+^+=Q_+^-+O(\mu^{\kappa}).
\end{equation}
We also have $Q_-^+=Q_-^-+O(\mu^{\kappa})$ due to to the definition of the sections $\{|Q_-^\pm|=2\mu^{\kappa}\}$.
This proves the last two equation in \eqref{eq: landau}.  Plugging \eqref{CmFixed} into \eqref{eq: cm} we see that
\[ v_+^+=v_+^-+O(\mu^{\kappa}).\]
This completes the proof of part (a).

The first claim of part (b) has already been established. The estimate of $G$ follows from the formula for $\al.$
The estimate of the closest distance follows from the fact that if $\alpha$ is bounded away from $0$ and $\pi$ then
the $Q_-$ orbit of $Q_-(t)$ is a small perturbation of Kepler motion and for Kepler motion the closest distance is of order $G.$ We integrate the $\dot G$ equation \eqref{DerG} over time $O(\mu^\kappa)$ to get the total variation $\Delta G$ is at most $\mu^{3\kappa}$, which is much smaller than $\mu$. So $G$ is bounded away from 0 by a quantity of order $O(\mu)$.

Finally part (c) follows since we know $G=\mu^\kappa|v_-|\sin\measuredangle(v_-,Q_-)=O(\mu).$
\end{proof}

\subsection{Proof of Lemma \ref{LmLMC0} and \ref{Lm: boundQ3}}
With the help of Lemma \ref{Lm: landau}, we are ready to prove Lemma \ref{LmLMC0} and \ref{Lm: boundQ3}.
\begin{proof}[Proof of Lemma~\ref{LmLMC0}]
Since we assume the outgoing asymptote $\bar\theta^+$ is close to $\pi$, we get that the orbit under consideration has to intersect the section $|Q_3-Q_4|=\mu^{\kappa}$ and also achieve $|Q_3-Q_4|=O(\mu)$ Lemma \ref{Lm: landau}. With the same initial $E_3,e_3,g_3,e_4$, we determine a solution of the Gerver's map.  It follows from \eqref{EqHamLoc} that the equations of motion outside the section $|Q_3-Q_4|=\mu^\kappa$ is a $O(\mu^{1-2\kappa})$ perturbation of the Kepler motion. We get that the $v^{-}_{3,4},Q_{3,4}^-$ at collision in Gerver's case is close to those values measured on the section $|Q_3-Q_4|=\mu^{\kappa}$ in the $\mu>0$ case. Here we note that the coordinates change between Cartesian and Delaunay outside the section $|Q_3-Q_4|=\mu^{\kappa}$ is not singular. 
Letting $\mu=0$ in the first two equations of \eqref{eq: landau} we obtain the equations
of elastic collisions. Namely, both the kinetic energy and momentum conservations hold
$$|v_3^+|^2+|v_4^+|^2=|v_3^-|^2+|v_4^-|^2,\quad v_3^++v_4^+=v_3^-+v_4^- .$$
On the other hand, the Gerver's map $\Ger$ in Lemma~\ref{LmLMC0} is also
defined through elastic collisions. If we could show that the rotation angle $\al$ in the $\mu>0$ case is close to Gerver's case, we then could show that the outgoing information $v^{+}_{3,4},Q_{3,4}^+$ are close in both cases. We then complete the proof using the fact that the orbit outside $|Q_3-Q_4|=\mu^\kappa$ is a small perturbation of the Kepler motion after running the orbit till the section $\{x_4=-2\}$. By converting $v_4^+,Q_4^+$ into Delaunay coordinates, we can express the outgoing asymptote $\bar\theta^+$ as a function of $v_4^+,Q_4^+$ therefore a function of $\al,v_3^-,v_4^-,Q_3^-,Q_4^-$ using \eqref{eq: landau} where $\mu=0$ corresponds to Gerver's case. To compare the angle $\al$, it is enough to show that  the outgoing asymptote $\bar\theta^+$ as a function of $\al$ has non degenerate derivative so that we can apply the implicit function theorem to solve $\al$ as a function of $\bar\theta^+$ and the initial conditions. In fact we have $\frac{d\bar\theta^+}{d\al}=\brlin \cdot \mathbf{u}$ up to a multiplicative non vanishing factor $c$, which is non vanishing due to Lemma \ref{Lm: local4}. Here the vectors $\brlin$ and $\mathbf{u}$ are in Lemma \ref{LmDerLoc} and \ref{LmDerGlob} with subscripts omitted. See item (2) of Remark \ref{RkPhysics} for the derivation of $d\theta^+=c\brlin$ and Corollary \ref{Cor} for $\frac{\partial}{\partial \al}=\mathbf u$. So the assumption $|\bar\theta^+-\pi|\leq\tilde\theta$ implies that $\al$ in \eqref{eq: landau} is $\tilde\theta$-close to its value in Gerver's case. 
\end{proof}

\begin{proof}[Proof of Lemma \ref{Lm: boundQ3}]
We follow the same argument as in the proof of Lemma \ref{LmLMC0} to get that the orbit of $Q_3$ is a small
deformation of Gerver's $Q_3$ ellipse. So we only need to prove this lemma in Gerver's setting. Since the $Q_3$ ellipse has semimajor $1$ in Gerver's case, the distance from the apogee to the focus is strictly less than $2$.
Therefore we can find some $D>0$ such that $|Q_3|\leq 2-2D$ in the Gerver case. Next we know from the Sublemma \ref{KeepDirection} and its proof that $Q_4$ moves away almost linearly (the oscillation of $v_4$ is small). We then integrate the $\frac{dL_3}{d\ell_4}$ equation to get that the oscillation of $L_3$ is $O(\mu). $ 
\end{proof}



\section{Consequences of $C^0$ estimates}
\label{ScCons}

Here we obtain corollaries $C^0$ estimates for the local and global maps. Namely, in subsection \ref{subsection: nocollision}
we show that the orbits we construct are collision free. In subsection \ref{SSAngMom} we show that the angular momentum can be prescribed
freely during the consecutive iterations of the inductive scheme, that is, we prove Sublemma \ref{LmHitTarget}.

\subsection{Avoiding collisions}\label{subsection: nocollision}
Here we exclude the possibility of collisions. The possible collisions may occur for the pair $Q_3,Q_4$ and the pair $Q_1,Q_4$.
The fact that there is no collision between $Q_4$ and $Q_1$ is a consequence of the following result.
\begin{Lm}\label{Lm: nocollision}
If an orbit satisfies the conditions of Lemma \ref{Lm: position} and there is a collision between $Q_4$ and $Q_1$ then we have
$\brG_4+G_4=O(\mu)$ where $G_4$ and $\brG_4$ denote the angular moment of $Q_4$ before and after the application of the global map respectively.
\end{Lm}
\begin{proof}
We write the  equations of motion
as $\mathbf Y'=\mathbf{V},$ where $\mathbf Y=(L_3,G_3,g_3;G_4,g_4)$ and $\mathbf{V}$ is the RHS of the Hamiltonian equations
\eqref{eq: Hamiltonian eq}.

We run the orbit coming to a collision backward so that we can compare it to the orbit exiting collision. We can still use the hyperbolic Delaunay coordinates to estimate the variational equation for collisional orbits as explained at the beginning of the proof of Lemma \ref{Lm: ordervar}. 
We shall use the subscript {\it in} to refer to
the orbit coming to collision with time direction reversed
the subscript {\it out}
for  the orbit exiting collision.

We have \[(\mathbf Y_{in}-\mathbf Y_{out})'=
O\left(\left\Vert\dfrac{\partial \mathbf{V}}{\partial \mathbf Y}\right\Vert\right)\;\;
\left(\mathbf Y_{in}-\mathbf Y_{out}\right)+ O\left(\dfrac{\mu}{|Q_4-Q_3|^2}\right)\]
where the last term comes from the $\frac{\mu}{|Q_4-Q_3|}$ term in the potential $V_L$.
We integrate this estimate for $\ell_4$ starting from the collision and ending when the outgoing orbit hits the section
$\left\{x_4=-\chi/2\right\}$.
The initial condition is $\mathbf Y_{in}-\mathbf Y_{out}=0$ since $L_3,G_4, g_4$ assume the same values
before and after the $Q_4$-$Q_1$ collision.  Next,
$\left\Vert\frac{\partial \mathbf{V}}{\partial \mathbf Y}\right\Vert=O\left(\frac{1}{\chi}\right)$
(this is proven in Lemma \ref{Lm: ordervar}(b)). Now the estimates
$$\int_{\ell_4^i}^{\ell_4^f}\dfrac{\partial \mathbf{V}}{\partial \mathbf Y}d\ell_4=O(1),
\quad \int_{\ell_4^i}^{\ell_4^f}O\left(\dfrac{\mu}{|Q_4-Q_3|^2}\right)d\ell_4=O(\mu/\chi) $$
and the Gronwall Lemma imply that
\begin{equation}\mathbf Y_{in}(\ell_4^f)-\mathbf Y_{out}(\ell_4^f)=O(\mu/\chi). \label{eq: Yinout}\end{equation}

Next we estimate the angular momentum of $Q_4$ w.r.t. $Q_2$.  We have
\begin{equation}
\label{AMLR}
G_{4R}=G_{4L}+v_4\times (-\chi,0)=G_{4L}+v_{4y}\chi,
\end{equation}
where $v_{4y}$ is the $y$ component of the velocity of $Q_4$ at the time the orbit hits
the section $\left\{x_4=-\chi/2\right\}$. Using the equation \eqref{eq: Q4} in the Appendix~\ref{subsection: hyp}
and Lemma \ref{Lm: tilt} we see that for the orbits of interest
$$ v_{4y}=\dfrac{k}{L_4^2}(L_4 \sin g_4-G_4\cos g_4)+O\left(\dfrac{1}{\chi^2}\right). $$
Now \eqref{eq: Yinout} shows that
$v_{4y,in}-v_{4y,out}=O(\mu/\chi)$, where we need to use \eqref{eq: W} to get that the difference of $L_4$ is also $O(\mu/\chi)$ from other variables when restricted to the section $\left\{x_4=-\chi/2\right\}$.  Hence \eqref{AMLR} implies that
$G_{4R,in}-G_{4R,out}=O(\mu).$
Finally the proof of Lemma \ref{Lm: position}
shows that the
angular momentum of $Q_4$ with respect to $Q_2$ changes by $O(\mu)$ during the time the orbits moves from the section
$\{x_4=-\chi/2\}$ to the section $\{x_4=-2\}.$
\end{proof}
Now we exclude the possibility of collisions between $Q_3$ and $Q_4$.
Note that $Q_3$ and $Q_4$ have two potential collision points corresponding to two intersections of the ellipse of $Q_3$
and the branch of the hyperbola utilized by $Q_4.$ See Fig 1 and 2 in Section \ref{SSGer}.
Now it follows from Lemma~\ref{Lm: landau}(b) that $Q_3$ and $Q_4$ do not collide near the intersection where they have
the close encounter. We need also to rule out the collision near the second intersection point.
This was done by Gerver in \cite{G2}. Namely he
shows that the time for $Q_3$ and $Q_4$ to move from one crossing point to the other are different.
As a result, if $Q_3$ and $Q_4$ come to the correct intersection points nearly simultaneously,
they do not collide at the wrong points. To see that the travel times are different
recall that by second Kepler's law
the area swiped by the moving body in unit time is a constant for the two-body problem. In terms of Delaunay coordinates, this fact is given by the equation $\dot\ell=\pm\frac{1}{L^3}$ where $-$ is for hyperbolic motion and $+$ for elliptic. In our case, we have
$L_3\approx L_4$ when $\mu\ll 1, \chi\gg 1.$ Therefore in order to collide 
$Q_3$ and $Q_4$ must swipe nearly the same area within the unit time.
We see from Fig 1 and Fig 2, the area swiped by $Q_4$ is a proper subset of that by $Q_3$ between the two crossing points.
Therefore the travel time for $Q_4$ is shorter.

\subsection{Choosing angular momentum}
\label{SSAngMom}
\begin{proof}[Proof of the Sublemma~\ref{LmHitTarget}.]
The idea is to apply the strong expansion of the Poincar\'{e} map in a neighborhood of the collisional orbit studied in Lemma \ref{Lm: nocollision}. Notice Delaunay coordinates regularize double collisions and our estimate of $d\Glob$ holds also for collisional orbits.
 
{\bf Step 1.} We first show that there is a collisional orbit as $\ell_3$ varies. 
The proof of Lemma \ref{Lm: nocollision} shows that $Q_4$ nearly returns back to its initial position.
Sublemma \ref{KeepDirection} shows that if after the application of the local map we have
$\theta_4^+(0)=\pi-\tilde\theta$ then the orbit hits the line $x_4=-\chi$ so that its $y_4$ coordinate is a large positive number
and if $\theta_4^+(0)=\pi+\tilde\theta$ then the orbit hits the line $x_4=-\chi$ so that its $y_4$ coordinate is a large negative number.
Therefore due to the Intermediate Value Theorem it suffices to show that our surface $S_j,\ j=1,2,$ contains points $\bx_1, \bx_2$ such that
$\theta_4^+(\bx_1)=\pi-\tilde\theta,$ $\theta_4^+(\bx_2)=\pi+\tilde\theta.$  We have the expression $\theta^+_4=g^+_4-\arctan\frac{G^+_4}{L^+_4}$. By direct calculation we find $d\theta^+=L_4^+\hat\brlin$ (see also item (2) of Remark \ref{RkPhysics}). Since $TS_j\subset \mathcal K_j$ and the cone $\mathcal K_j$ is centered at the plane $span\{\brv_{3-j},\brrv_{3-j}\}$. Note that $\brrv_{3-j}\to\tw=\frac{\partial}{\partial\ell_3}$. We get using Lemma \ref{Lm: local4}\[d\theta^+\cdot(d\Loc \brrv_{3-j})=L_4^+\hat\brlin_j\cdot\left(\dfrac{1}{\mu}(\hat{\mathbf u}_j (\hat\lin_j \tw)+o(1))+O(1)\right)=c_j(\bx)/\mu,\ c_j(\bx_j)\neq 0.\]
So it is enough to vary $\ell_3$ in a $O(\mu)$ neighborhood of a point whose outgoing asymptotes satisfies the assumption of Lemma \ref{LmDerLoc}. We choose $\tilde\theta\ll 1$ but independent of $\mu$ such that the assumption of Lemma \ref{LmDerLoc} and Sublemma \ref{KeepDirection} is satisfied.

{\bf Step 2.} We show that there exists $\ell_3$ such that $\bre_4(\cP(S(\ell_3,\te_4)))$ is close to $e_4^{**}.$
We fix $\te_4$ then $\mathcal P(S(\cdot,\te_3))$ becomes a function of one variable $\ell_3$. Suppose the collisional orbit in Step 1 occurs at $\ell_3=\hat\ell_3$. As we vary $\ell_3$, 
the same calculation as in Step 1 gives  $\hat\brlin_j\cdot(d\Loc \brrv_{3-j})=\bar c_j(\bx)/\mu,\ \bar c_j(\bx_j) \neq 0$
and that $\brv_j$ contains nonzero $\partial/\partial e_4$ component.
Therefore the projection of $\mathcal{P}=\mathbb G\circ\mathbb L$ to the $e_4$ component, i.e.
$\bre_4(\ell_3, \te_4)$ as a function of $\ell_3$ is strongly expanding with
derivative bounded from below by $\frac{\brc\chi^2}{\mu}$
provided that the assumptions of Lemma \ref{Lm: position}
are satisfied (for the orbits of interest this will always be the case according to Lemma \ref{Lm: strip}).
Considering the map $\bre_4(\ell_3, \te_4)$ is not injective, we study $\bar G_4(\ell_3, \te_4)$ instead of $\bre_4(\ell_3,\te_4)$ using the relation $e=\sqrt{1+(G/L)^2}$. We have the same strong expansion for $\bar G_4(\ell_3, \te_4)$ since our estimates of the $d\Loc,d\Glob$ are done using $G_4$ instead of $e_4$.
Thus it follows from the strong expansion of the map $\bar G_4(\ell_3, \te_4)$ that a $R$-neighborhood of
$G_4^{**}$ (corresponding to $e^{**}_4$) is covered if $\ell_3$ varies in a $\frac{R\mu}{\brc\chi^2}$-neighborhood of $\hat\ell_3$. Taking $R$ large we can ensure
that $\bar G_4$ changes from a large negative number to a large positive number.
Then we use the intermediate value theorem to find $e_4$ such that $|\bar G_4-G_4^{**}|<KK'\delta,$ hence $|\bar e_4-e_4^{**}|<KK'\delta.$

{\bf Step 3.} We show that for the orbit just constructed $\cP(S(\tilde\ell_3, \te_4))\in U_2(\delta).$
By Lemma \ref{LmGMC0}, we get $\theta_4^+=O(\mu)$. Therefore by Lemma \ref{LmLMC0}
$\Loc(\te_4, \ell_3)$  has $(E_3, e_3, g_3)$ close to
$\Ger_{\te_4, 2, 4}(E_3(\te_4, \ell_3), e_3(\te_4, \tilde\ell_3), g_3(\te_4, \tilde\ell_3)).$
It follows that
$$|E_3-E_3^{**}|<KK'\delta/2, \quad |e_3-e_3^{**}|<KK'\delta/2, \quad |g_3-g_3^{**}|<KK'\delta/2. $$
Next Lemma \ref{LmGMC0} shows that after the application of $\Glob$, $(E_3, e_3, g_3)$ change little and
$\theta_4^-$ becomes $O(\mu).$
\end{proof}

\section{Derivative of the local map}
\label{section: local}

\subsection{Justifying the asymptotics}

Here we give the proof of Lemma \ref{LmDerLoc}. 
Our goal is to show that the main contribution to the derivative
comes from differentiating the main term in Lemma \ref{Lm: landau}.

\begin{proof}[Proof of Lemma~\ref{LmDerLoc}]
Since the transformation from Delaunay to Cartesian variables
is symplectic and the norms of the transformation matrices are independent of $\mu$,
it is sufficient to prove the lemma in terms of Cartesian coordinates.  To go to the coordinates system used in Lemma~\ref{LmDerLoc}, we only need to multiply the Cartesian derivative matrix by $O(1)$ matrices, namely,
by $\frac{\partial (L_3,\ell_3,G_3,g_3,G_4,g_4)^+}{\partial(Q_3,v_3,Q_4,v_4)^+}$ on the left and
by $\frac{\partial(Q_3,v_3,Q_4,v_4)^-}{\partial (L_3,\ell_3,G_3,g_3,G_4,g_4)^-}$ on the right.
This does not change the form of the $d\Loc$ stated in Lemma~\ref{LmDerLoc}.

As before we use the formula~(\ref{eq: formald2}).
We need to consider the integration of the variational equations and also the boundary contribution. 

{\it Recall that the subscripts $-$ and $+$ mean relative motion and mass center motion respectively, and the superscripts $-$ and $+$ mean incoming and outgoing respectively. In the following, we are most interested in the relative motion, so we drop the \text{subscript} $-$ of $Q_-,v_-,\mathcal L_-,G_-,g_-$ for simplicity and without leading to confusion. }

{\bf Step 1, the Hamiltonian equations, the variational equations and the boundary contributions.}
 
It is convenient to use the variable $\mathcal L=L/\mu$. Lemma \ref{Lm: landau} gives that $1/c<\mathcal L<c$ and $\mu/c\leq G\leq c\mu$ for some $c>1$ if the rotation angle $\al$ is bounded away from $0$ and $\pi$. We also have $g,Q_+,v_+,v=O(1)$ and $Q=O(\mu^\kappa)$. 
From the Hamiltonian \eqref{eq: relcm}, we have $\dot\ell=-\frac{1}{2\mu \mathcal{L}^{3}}+O(\mu^{2\kappa})$ (see \eqref{EqDotl}). Using $\ell$ as the time variable we get from \eqref{eq: relcm} that the equations of motion take the following form
(recall that $\ell=O(\mu^{\kappa-1})$ due to \eqref{SectionEll}):

\begin{equation}
\begin{cases}
\dfrac{\partial \mathcal L}{\partial \ell}=-2  \mathcal L^{3}\dfrac{\partial H}{\partial \ell }
\left(1-2\mathcal L^{3}\dfrac{\partial H}{\partial \mathcal L }+\dots \right)(1+O(\mu^{2\kappa+1}))= O(\mu^{1+\kappa}),\\
\dfrac{\partial G}{\partial \ell}=-2\mu \mathcal  L^{3}\dfrac{\partial H}{\partial g}
\left(1-2\mathcal L^{3}\dfrac{\partial H}{\partial \mathcal L }+\dots \right)(1+O(\mu^{2\kappa+1}))=O(\mu^{1+2\kappa}),\\
\dfrac{\partial g}{\partial \ell}=2\mu \mathcal L^{3}\dfrac{\partial H}{\partial G}
\left(1-2\mathcal  L^{3}\dfrac{\partial H}{\partial \mathcal L }+\dots \right)(1+O(\mu^{2\kappa+1}))=O(\mu^{2\kappa}),\\
\dfrac{dQ_+}{d\ell}=-\dfrac{v_+}{2}(2\mu \mathcal{L}^{3})(1+O(\mu^{2\kappa+1}))=O(\mu)\\
\dfrac{dv_+}{d\ell}=\left(\dfrac{2Q_+}{|Q_+|^3}+O(\mu^{2\kappa})\right)(2\mu \mathcal{L}^{3})(1+O(\mu^{2\kappa+1}))=O(\mu).
\end{cases}\label{eq: hamrel}
\end{equation}
where $\dots$ denote the higher order terms. The estimates of the last two equations are simple. In the first three equations, the main contribution in $H$ is coming from $|Q|^2$ and $|Q_+\cdot Q|^2$, both of which are $O(\mu^{2\kappa})$. We have the estimate 
$$\left|\left(\frac{\partial}{\partial \mathcal L},\frac{\partial}{\partial \ell},\frac{\partial}{\partial G},\frac{\partial}{\partial g}\right) Q\right|=O(\mu^\kappa, \mu, \mu^{\kappa-1},\mu^\kappa)$$
using \eqref{eq: delaunayscattering} for $Q=(q_1,q_2)$ up to a rotation by $g$. 
 In fact, the $\frac{\partial}{\partial \ell}$ amounts to dividing by the scale of $\ell$, i.e. $\mu^{-1+\kappa}$. The derivatives $\frac{\partial}{\partial \mathcal L},\frac{\partial}{\partial g}$ do not change the order of magnitude. Finally since $G=O(\mu)$, the $\frac{\partial}{\partial G}$ amounts to dividing by $\mu.$
Next  we analyze the variational equations. This estimate is much easier than that of the global map part. 
The same rules as those used to obtain \eqref{eq: hamrel} apply here. 
\begin{equation}
\dfrac{d}{d\ell}\left[\begin{array}{c}
\delta \mathcal L\\
\delta G\\
\delta g\\
\delta Q_+\\
\delta v_+\\
\end{array}\right]=O\left(\begin{array}{ccccc}
\mu^{1+\kappa}& \mu^{\kappa}&\mu^{1+\kappa}&\mu^{1+\kappa}&0 \\
\mu^{1+2\kappa}&\mu^{2\kappa}&\mu^{1+2\kappa}&\mu^{1+2\kappa}&0\\
\mu^{2\kappa}&\mu^{2\kappa-1}&\mu^{2\kappa}& \mu^{2\kappa}&0\\
\mu&\mu^{2\kappa+1}&\mu^{2\kappa+2}&\mu^{2\kappa+2}&\mu\\
\mu&\mu^{2\kappa+1}&\mu^{2\kappa+2}&\mu&0\\
\end{array}\right)
\left[\begin{array}{c}
\delta \mathcal L\\
\delta G\\
\delta g\\
\delta Q_+\\
\delta v_+\\
\end{array}\right].
\label{eq: varrel}
\end{equation}
We need to integrate this equation over time $\mu^{\kappa -1}$. As we did in the proof of Lemma \ref{Lm: variation}, we have Gronwall inequality for linear systems (Lemma \ref{LmLinGronwall}).
Recall also that the ``$\leq$" for matrices means ``$\leq$" entry-wise. 

Thus we compare  the solution to the variational equation with a constant linear 
ODE of the form $X'=AX.$
Its solution has form $X(\mu^{\kappa-1})=\sum_{n=0}^\infty \frac{(A\mu^{\kappa-1})^n}{n!}$. 
We will show that
\begin{equation}
\label{TwoTermPicard}
(A\mu^{\kappa-1})^3\leq C_3((A\mu^{\kappa-1})+(A\mu^{\kappa-1})^2).
\end{equation} 
Then we have 
$$(A\mu^{\kappa-1})^n\leq C_n((A\mu^{\kappa-1})+(A\mu^{\kappa-1})^2), \quad C_n=C_3(1+C_3)^n.$$ 
Hence  $X(\mu^{\kappa-1})\leq \mathrm{Id}+C((A\mu^{\kappa-1})+(A\mu^{\kappa-1})^2)$. We next integrate the variational equations over time $O(\mu^{\kappa-1})$ to get the estimate of its fundamental solution. 
{\em From now on, we fix $\kappa=2/5\in (1/3,1/2).$} 
We get the following bound for the fundamental solution of the variational equation in the case of 
$\kappa=2/5$, in which case \eqref{TwoTermPicard} holds and so
two steps of Picard iteration are enough 
\begin{equation}
\Id_7+
O\left(\begin{array}{ccc|cc}
\mu^{2\kappa}& \mu^{2\kappa-1}&\mu^{2\kappa}&\mu^{2\kappa}&\mu^{3\kappa}\\
\mu^{3\kappa}&\mu^{3\kappa-1}&\mu^{3\kappa}&\mu^{3\kappa}&\mu^{4\kappa}\\
\mu^{3\kappa-1}&\mu^{3\kappa-2}&\mu^{3\kappa-1}&\mu^{3\kappa-1}&\mu^{4\kappa-1}\\
\hline
\mu^\kappa&\mu^{3\kappa-1}&\mu^{3\kappa}&\mu^{2\kappa}&\mu^{\kappa}\\
\mu^\kappa&\mu^{3\kappa-1}&\mu^{3\kappa}&\mu^{\kappa}&\mu^{2\kappa}\\
\end{array}\right).\label{eq: fundrel}
\end{equation}
This calculation can either be done by hand or using computer. 

Next, we compute the boundary contribution using the formula~(\ref{eq: formald2}).
In terms of the Delaunay variables inside the sphere $|Q|=2\mu^{\kappa}$, we have
\begin{equation}\dfrac{\partial \ell}{\partial (\mathcal{L},G,g,Q_+,v_+)}=-\left(\dfrac{\partial |Q|}{\partial\ell}\right)^{-1}\dfrac{\partial |Q|}{\partial (\mathcal{L},G,g,Q_+,v_+)}=(O(\mu^{\kappa-1}),O(\mu^{\kappa-2}),0,0,0).\label{eq: boundaryrel}\end{equation}
Indeed, due to \eqref{eq: delaunayscattering}
we have
$\frac{\partial |Q|}{\partial g}=0$, $\frac{\partial |Q|}{\partial\ell}=O(\mu)$, $\frac{\partial |Q|}{\partial \mathcal{L}}=O(\mu^{\kappa})$
and $\frac{\partial |Q|}{\partial G}=O(\mu^{\kappa-1}).$ Combining this with \eqref{eq: hamrel} we get
\begin{equation}
\label{EqBCLocRel}
\begin{aligned}
&\left(\dfrac{\partial }{\partial \ell}(\mathcal{L},G,g,Q_+,v_+)\right)\otimes\dfrac{\partial \ell}{\partial (\mathcal{L},G,g,Q_+,v_+)}\\
&=O(\mu^{1+\kappa},\mu^{1+2\kappa},\mu^{2\kappa},\mu,\mu)\otimes O(\mu^{\kappa-1},\mu^{\kappa-2},0,0,0).
\end{aligned}
\end{equation}

{\bf Step 2, the analysis of the relative motion part.}

The structure of $d\Loc$ comes mainly from the relative motion part, on which we now focus. We neglect the $Q_+,v_+$ part and will study them in the last step. 

{\it Substep 2.1, the strategy. }

Using \eqref{eq: formald2} we obtain the derivative matrix
\begin{equation}
\begin{aligned}
&\dfrac{\partial (\mathcal{L},G,g)^+}{\partial (\mathcal{L},G,g)^-}=
\left(\Id_3+O\left(\begin{array}{ccc}
\mu^{2\kappa}& \mu^{2\kappa-1}&0\\
\mu^{3\kappa}&\mu^{3\kappa-1}&0\\
\mu^{3\kappa-1}&\mu^{3\kappa-2}&0
\end{array}\right)\right)^{-1} \times \\
&\left(\Id_3+O\left(\begin{array}{ccc}
\mu^{2\kappa}& \mu^{2\kappa-1}&\mu^{2\kappa}\\
\mu^{3\kappa}&\mu^{3\kappa-1}&\mu^{3\kappa}\\
\mu^{3\kappa-1}&\mu^{3\kappa-2}&\mu^{3\kappa-1}
\end{array}\right)\right)
\left(\Id_3-O\left(\begin{array}{ccc}
\mu^{2\kappa}& \mu^{2\kappa-1}&0\\
\mu^{3\kappa}&\mu^{3\kappa-1}&0\\
\mu^{3\kappa-1}&\mu^{3\kappa-2}&0
\end{array}\right)\right)\\
&=\Id_3+O\left(\begin{array}{ccc}
\mu^{2\kappa}& \mu^{2\kappa-1}&\mu^{2\kappa}\\
\mu^{3\kappa}&\mu^{3\kappa-1}&\mu^{3\kappa}\\
\mu^{3\kappa-1}&\mu^{3\kappa-2}&\mu^{3\kappa-1}
\end{array}\right):=\Id_3+P.\label{eq: dervar}
\end{aligned}
\end{equation}

For the position variables $q$, we are only interested in the angle $\Theta:=\arctan\left(\frac{q_2}{q_1}\right)$ since the length $|(q_1,q_2)|=2\mu^\kappa$ is fixed when restricted on the sphere.

We split the derivative matrix as follows:
\begin{equation}
\dfrac{\partial (\Theta,v)^+}{\partial (\Theta,v)^-}=\dfrac{\partial (\Theta,v)^+}{\partial (\mathcal{L},G,g)^+}\dfrac{\partial (\mathcal{L},G,g)^+}{\partial (\mathcal{L},G,g)^-}\dfrac{\partial (\mathcal{L},G,g)^-}{\partial (\Theta,v)^-}=
\label{eq: dersplitting3}
\end{equation}
$$ \dfrac{\partial (\Theta,v)^+}{\partial (\mathcal{L},G,g)^+}\dfrac{\partial (\mathcal{L},G,g)^-}{\partial (\Theta,v)^-}+
\dfrac{\partial (\Theta,v)^+}{\partial (\mathcal{L},G,g)^+} P \dfrac{\partial (\mathcal{L},G,g)^-}{\partial (\Theta,v)^-}=I+II. $$
In the following, we prove 

{\bf Claim: }\begin{equation}\label{EqI&II}
I=\frac{1}{\mu}O(1)_{1\times 3}\otimes \frac{\partial G^-}{\partial (\Theta,v)^-}+O(1),\ II=\frac{1}{\mu}O(\mu^{3\kappa-1})_{1\times 3}\otimes \frac{\partial G^-}{\partial (\Theta,v)^-}+O(\mu^{3\kappa-1}).\end{equation}
We will give the expressions of $O(1)$ terms explicitly. 

{\it Substep 2.2, the estimate of $I$ in the splitting \eqref{eq: dersplitting3}.}

Using equations \eqref{eq: delaunayscattering} and \eqref{EqDelMom}
we obtain
\begin{equation}
\label{LocRelCartDel}
\dfrac{\partial (\Theta,v)^+}{\partial (\mathcal{L},G,g)^+}=
O\left(\begin{array}{ccc}
1&\mu^{-1}&1\\
1&\mu^{-1}&1\\
1&\mu^{-1}&1
\end{array}\right).
\end{equation}
Next, we consider the first term in \eqref{eq: dersplitting3}.
\begin{equation}
I =\dfrac{\partial (\Theta,v)^+}{\partial \mathcal{L}^{+}}\otimes \dfrac{\partial \mathcal{L}^{-}}{\partial (\Theta,v)^-}
+\dfrac{\partial (\Theta,v)^+}{\partial G^{+}}\otimes \dfrac{\partial G^{-}}{\partial (\Theta,v)^-}+\dfrac{\partial (\Theta,v)^+}{\partial g^{+}}\otimes \dfrac{\partial g^{-}}{\partial (\Theta,v)^-}.
\label{eq: dersplitingid}
\end{equation}
Using the expressions
$\frac{1}{4\mathcal{L}^{2}}=\frac{v^2}{4}-\frac{\mu}{2|Q|},\quad G=v\times Q=|v|\cdot|Q|\sin\measuredangle (v,Q)$,
we see that
\begin{equation}
\label{RelLocDelCart}
\dfrac{\partial \mathcal{L}^{-}}{\partial (\Theta,v)^-}=O(1),\quad \dfrac{\partial G^{-}}{\partial (\Theta,v)^-}=(O(\mu^\kappa),O(\mu^\kappa)).
\end{equation}
It only remains to get the estimate of $\frac{\partial g^{-}}{\partial (\Theta,v)^-}.$ Next, we claim that
 \begin{equation}
\dfrac{\partial g^-}{\partial (\Theta,v)^-}=
\left[\frac{\partial}{\partial G^-}\arctan\left(\frac{G^-}{\mu \mathcal{L}}\right)\right]
\dfrac{\partial G^-}{\partial (\Theta,v)^-}+O(1)=O(1/\mu)\dfrac{\partial G^-}{\partial (\Theta,v)^-}+O(1).\label{eq: derg}
\end{equation}
We use the fact
\[\dfrac{p_2}{p_1}=\dfrac{\sin g \sinh u\pm\frac{G}{\mu \mathcal{L}}\cos g\cosh u}{\cos g \sinh u\mp\frac{G}{\mu \mathcal{L}}\sin g\cosh u}=\dfrac{\tan g\pm\frac{G}{\mu \mathcal{L}}}{1\mp\frac{G}{\mu \mathcal{L}} \tan g} +e^{-2|u|}E(G/\mu \mathcal{L},g,u),\]
where $E$ is a smooth function satisfying
$\frac{\partial E}{\partial g}=O(1)$ as $\ell\to\infty$.
Therefore we get
\[g=\arctan\left(\dfrac{p_2}{p_1}-e^{-2|u|}E(G/\mu \mathcal{L},g)\right)\mp\arctan\dfrac{G}{\mu \mathcal{L}}\ \mathrm{as}\ \ell\to\infty.\]
We choose the $+$ when considering the incoming orbit parameters.
Thus
\begin{equation}\nonumber\begin{aligned}
&\dfrac{\partial g}{\partial (\Theta,v)}\left(1+O(e^{-2|u|})\right)=\dfrac{\partial \arctan\frac{p_2}{p_1}}{\partial (\Theta,v)}+\dfrac{\partial\arctan\frac{G}{\mu \mathcal{L}}}{\partial \mathcal{L}}\dfrac{\partial \mathcal{L}}{\partial (\Theta,v)}\\
&+\left(\frac{\partial\arctan\frac{G}{\mu \mathcal{L}}}{\partial G}+O(e^{-2|u|}/\mu)\right)\dfrac{\partial G}{\partial (\Theta,v)} +O(e^{-2|u|}) \end{aligned}\end{equation}
proving \eqref{eq: derg}. 

Plugging \eqref{LocRelCartDel}, \eqref{RelLocDelCart} and \eqref{eq: derg} back to \eqref{eq: dersplitingid} we get the estimate of $I$ in \eqref{EqI&II}. More explicitly, 
$I=\frac{1}{\mu}\mathbf U\otimes \frac{\partial G^-}{\partial (\Theta,v)^-}+\mathbf B$, where
{\small\begin{equation}
\begin{aligned}
&\mathbf U=\left(\mu \dfrac{\partial (\Theta,v)^+}{\partial G^{+}}+\mu\dfrac{\partial\arctan\frac{G^-}{\mu \mathcal{L}^{-}}}{\partial G^-}\dfrac{\partial (\Theta,v)^+}{\partial g^{+}}+O(e^{-2|u|})\right)\\
&\mathbf B=\left(\dfrac{\partial (\Theta,v)^+}{\partial \mathcal{L}^{+}}\otimes\dfrac{\partial \mathcal{L}^{-}}{\partial (\Theta,v)^-}+\dfrac{\partial (\Theta,v)^+}{\partial g^{+}}\otimes \left(\dfrac{\partial \arctan\frac{p_2^-}{p^-_1}}{\partial (\Theta,v)}+\dfrac{\partial\arctan\frac{G^-}{\mu \mathcal{L}^{-}}}{\partial \mathcal{L}^{-}}\dfrac{\partial \mathcal{L}^{-}}{\partial (\Theta,v)}\right)\right)+O(e^{-2|u|}).
\end{aligned}\label{eq: local2bp}
\end{equation}}
{\it Substep 2.3, the estimate of $II$ in the splitting \eqref{eq: dersplitting3}.}

Now we study the second term in \eqref{eq: dersplitting3}
\begin{equation}
\begin{aligned}
II&=O\left(\begin{array}{ccc}
1&\mu^{-1}&1\\
1&\mu^{-1}&1\\
1&\mu^{-1}&1
\end{array}\right)\cdot O
\left(\begin{array}{ccc}
\mu^{2\kappa}& \mu^{2\kappa-1}&\mu^{2\kappa}\\
\mu^{3\kappa}&\mu^{3\kappa-1}&\mu^{3\kappa}\\
\mu^{3\kappa-1}&\mu^{3\kappa-2}&\mu^{3\kappa-1}
\end{array}\right)
\dfrac{\partial (\mathcal{L},G,g)^-}{\partial (\Theta,v)^-}\\
&=O\left(\begin{array}{ccc}
\mu^{3\kappa-1}& \mu^{3\kappa-2}&\mu^{3\kappa-1}\\
\mu^{3\kappa-1}&\mu^{3\kappa-2}&\mu^{3\kappa-1}\\
\mu^{3\kappa-1}&\mu^{3\kappa-2}&\mu^{3\kappa-1}
\end{array}\right)
\dfrac{\partial (\mathcal{L},G,g)^-}{\partial (\Theta,v)^-}\\
&=\mu^{3\kappa-1}\left[O(1)_{1\times 3}\otimes \dfrac{\partial \mathcal{L}^{-}}{\partial (\Theta,v)^-} +O(\mu^{-1})_{1\times 3}\otimes \dfrac{\partial G^{-}}{\partial (\Theta,v)^-}
+O(1)_{1\times 3}\otimes \dfrac{\partial g^{-}}{\partial (\Theta,v)^-}\right]
\end{aligned}\label{eq: addpert}
\end{equation}
where we use that $\mu^{2\kappa}<\mu^{3\kappa-1}$ and $\mu^{2\kappa-1}<\mu^{3\kappa-2}$
since $\kappa<1/2$.
The first summand in \eqref{eq: addpert}
is $O(\mu^{3\kappa-1})$. Applying \eqref{eq: derg}, we get the estimate of $II$ in \eqref{EqI&II}. 

{\it Substep 2.4, going from $\Theta$ to $Q$.}

We use the variable $\Theta$ for the relative position $Q$ and we have $\frac{\partial G^-}{\partial (\Theta,v)^-}=O(\mu^\kappa)$. 
To obtain $\frac{\partial (Q,v)^+}{\partial (Q,v)^-}$, we use that
$$Q=2\mu^\kappa(\cos\Theta,\sin\Theta)=(x,y),\quad \Theta=\arctan\frac{y}{x}.$$
So we have the estimate $\frac{\partial Q^+}{\partial (\mathcal L,G,g)^+}=O(\mu^\kappa)\frac{\partial \Theta^+}{\partial (\mathcal L,G,g)^+}=O(\mu^{\kappa-1})$. To get $\frac{\partial -}{\partial Q^-},$ we use the transformation from polar coordinates to Cartesian, $\frac{\partial -}{\partial Q^-}=\frac{\partial -}{\partial (r,\Theta)^-}\frac{\partial(r,\Theta)^-}{\partial Q^-}$, where $r=|Q^-|=2\mu^\kappa$. 
Therefore we have $\frac{\partial r^-}{\partial Q^-}=0,\ \frac{\partial -}{\partial Q^-}=\frac{1}{\mu^\kappa}\frac{\partial -}{\partial \Theta^-}(-\sin\Theta^-,\cos\Theta^-).$ So we have the estimate $\frac{\partial G^-}{\partial Q^-}=O(1)$, and $\frac{\partial \mathcal{L}^-}{\partial Q^-}=\frac{\partial \mathcal{L}^-}{\partial \Theta^-}=0$ since in the expression $\frac{1}{4\mathcal{L}^{2}}=\frac{v^2}{4}-\frac{\mu}{2|Q|}$, the
angle $\Theta$ plays no role. Finally, we have $\frac{\partial }{\partial Q_-^-}\arctan \frac{p_2^-}{p^-_1}=0$. Applying these estimates to the concrete expressions of $\mathbf U,\mathbf B$, and \eqref{eq: addpert} for the $O(\mu^{3\kappa-1})$ remainder, 
so we get \begin{equation}\label{EqQvpm}\dfrac{\partial (Q,v)^+}{\partial (Q,v)^-}=\dfrac{1}{\mu}(O(\mu^\kappa)_{1\times 2}, O(1)_{1\times 2})\otimes (O(1)_{1\times 2},O(\mu^\kappa)_{1\times 2})+O(1)_{4\times 4}.\end{equation}
In particular, the estimate of $\mathbf B+O(\mu^{3\kappa-1})=O(1)$ instead of $O(\mu^{-\kappa})$ is due to
$\frac{\partial \mathcal{L}^-}{\partial Q^-}=\frac{\partial \mathcal{L}^-}{\partial \Theta^-}=0$ and $\frac{\partial p^-_-}{\partial Q^-}=\frac{\partial p^-_-}{\partial \Theta^-}=0.$

{\bf Step 3, the contribution from the motion of the mass center.}

{\it Substep 3.1, the decomposition.}

Consider the following decomposition 
\begin{equation}
\begin{aligned}
\cD&:=\dfrac{\partial(\Theta,v,Q_+,v_+)^+}{\partial(\Theta,v,Q_+,v_+)^-}=\dfrac{\partial(\Theta,v;Q_+,v_+)^+}{\partial(\mathcal{L},G,g;Q_+,v_+)^+}\dfrac{\partial(\mathcal{L},G,g;Q_+,v_+)^+}{\partial (\mathcal{L},G,g;Q_+,v_+)(\ell^f) }\\
&\dfrac{\partial (\mathcal{L},G,g;Q_+,v_+)(\ell^f)}{\partial (\mathcal{L},G,g;Q_+,v_+)(\ell^i)} \dfrac{\partial (\mathcal{L},G,g;Q_+,v_+)(\ell^i)}{\partial (\mathcal{L},G,g;Q_+,v_+)^-}\dfrac{\partial (\mathcal{L},G,g;Q_+,v_+)^-}{\partial (\Theta,v;Q_+,v_+)^-}\\
&:=\left[\begin{array}{cc}
M&0\\
0&\Id_4
\end{array}\right]
\left[\begin{array}{cc}
A&0\\
B&\Id_4
\end{array}\right]
\left[\begin{array}{cc}
C&D\\
E&F
\end{array}\right]
\left[\begin{array}{cc}
A'&0\\
B'&\Id_4
\end{array}\right]
\left[\begin{array}{cc}
N&0\\
0&\Id_4
\end{array}\right]\\
&=\left[\begin{array}{cc}
MACA'N+MADB'N&MAD\\
(BC+E)A'N+(BD+F)B'N&BD+F
\end{array}\right].
\end{aligned}\label{eq: matrices}
\end{equation}
Each of the above matrix is $7\times 7$. 

{\it Substep 3.2, the estimate of each block. }

The matrix $M=\frac{\partial(\Theta,v)^+}{\partial(\mathcal L,G,g)^+}$ is given by \eqref{LocRelCartDel} and
$N=\frac{\partial(\mathcal L,G,g)^-}{\partial(\Theta,v)^-}$ is given by  \eqref{LocRelCartDel}, \eqref{RelLocDelCart}, \eqref{eq: derg}  $$M=O\left(\begin{array}{ccc}
1& \mu^{-1}&1\\
1& \mu^{-1}&1\\
1& \mu^{-1}&1
\end{array}\right),\quad
N=\left(\begin{array}{cccc}
O(1)_{1\times 3}\\
\frac{\partial G^-}{\partial (\Theta,v)^-}\\
O(\frac{1}{\mu})\frac{\partial G^-}{\partial (\Theta,v)^-}+O(1)\\
\end{array}\right).$$
$C,D,E,F$ form the matrix \eqref{eq: fundrel}, the fundamental solution of the variational equation, 
$$\left(\begin{array}{c|c}
C&D\\
\hline
E&F\end{array}\right)=	
\Id_7+
O\left(\begin{array}{ccc|cc}
\mu^{2\kappa}& \mu^{2\kappa-1}&\mu^{2\kappa}&(\mu^{2\kappa})_{1\times 2}&(\mu^{3\kappa})_{1\times 2}\\
\mu^{3\kappa}&\mu^{3\kappa-1}&\mu^{3\kappa}&(\mu^{3\kappa})_{1\times 2}&(\mu^{4\kappa})_{1\times 2}\\
\mu^{3\kappa-1}&\mu^{3\kappa-2}&\mu^{3\kappa-1}&(\mu^{3\kappa-1})_{1\times 2}&(\mu^{4\kappa-1})_{1\times 2}\\
\hline
(\mu^\kappa)_{2\times 1}&(\mu^{3\kappa-1})_{2\times 1}&(\mu^{3\kappa})_{2\times 1}&(\mu^{2\kappa})_{2\times 2}&(\mu^{\kappa})_{2\times 2}\\
(\mu^\kappa)_{2\times 1}&(\mu^{3\kappa-1})_{2\times 1}&(\mu^{3\kappa})_{2\times 1}&(\mu^{\kappa})_{2\times 2}&(\mu^{2\kappa})_{2\times 2}\\
\end{array}\right).
$$
  $A,B,A',B'$ are given by \eqref{EqBCLocRel}, boundary contributions, 
	$$\left[\begin{array}{c|c}
A&0\\
\hline
B&\Id_4
\end{array}\right],\left[\begin{array}{c|c}
A'&0\\
\hline
B'&\Id_4
  \end{array}\right]=\mathrm{Id}_7+
O(\mu^{1+\kappa},\mu^{1+2\kappa},\mu^{2\kappa};\mu_{1\times4 })
\otimes O(\mu^{\kappa-1},\mu^{\kappa-2},0;0_{1\times 4}).
$$

{\it Substep 3.3, the estimate of the first block $MACA'N+MADB'N$ in $\mathcal D$.}
By \eqref{eq: dervar}
$$ACA'=\Id_3+P=\Id_3+O\left(\begin{array}{ccc}
\mu^{2\kappa}& \mu^{2\kappa-1}&\mu^{2\kappa}\\
\mu^{3\kappa}&\mu^{3\kappa-1}&\mu^{3\kappa}\\
\mu^{3\kappa-1}&\mu^{3\kappa-2}&\mu^{3\kappa-1}
\end{array}\right)$$
(Recall that \eqref{eq: dervar} is the part of $\frac{\partial (\mathcal{L},G,g)^+}{\partial (\mathcal{L},G,g)^-}$
without considering the motion of the mass center),
and by \eqref{EqI&II} and \eqref{eq: local2bp}
\begin{equation}\label{CD-Main}
MACA'N=M(\Id_3+P)N=\frac{1}{\mu}\left(\mathbf{U}+O\left(\mu^{3\kappa-1}\right)\right) 
\otimes \frac{\partial G^-}{\partial (\Theta,v)^-}
+\mathbf B+O\left(\mu^{3\kappa-1}\right).
\end{equation} 
Indeed, using the notation of \eqref{eq: dersplitting3}, we have $I=MN$ and $II=MPN$.
The estimates of $I$ and $II$ are given in  \eqref{EqI&II}.

Next we claim that
\begin{equation}\label{CD-Error}
  MADB'N=
  O\left(\mu^{3\kappa-2}\right)\frac{\partial G^-}{\partial (\Theta,v)^-}+
O\left(\mu^{3\kappa-1}\right)
\end{equation}
so it can be absorbed into the errror terms of \eqref{CD-Main}.
To this end we split $N=N_1+N_2,$ $A=\mathrm{Id}+A_2$ where
$$
N_1=\left(
\begin{array}{c} 0_{1\times 3} \\
\frac{\partial G^-}{\partial (\Theta,v)^-}\\
O(\frac{1}{\mu})\frac{\partial G^-}{\partial (\Theta,v)^-}
\end{array}\right), \quad
N_2=\left(\begin{array}{c} O(1)_{1\times 3} \\ 0_{1\times 3} \\
  O(1)_{1\times 3} \end{array}\right), $$
$A_2=
O(\mu^{1+\kappa},\mu^{1+2\kappa},\mu^{2\kappa})
\otimes O(\mu^{\kappa-1},\mu^{\kappa-2},0).
$ Thus
$MADB'N=MDB'N+MA_2 DB'N.$ Let us work on the first term.
A direct computation shows that
$$ DB'=O\left(\begin{array}{ccc} \mu^{3\kappa} & \mu^{3\kappa-1} & 0 \\
  \mu^{4\kappa} & \mu^{4\kappa-1} & 0 \\
  \mu^{4\kappa-1} & \mu^{4\kappa-2} & 0 \end{array} \right), $$
and 
$MDB'=O\left(\mu^{4\kappa-1}_{3\times 1}, \mu^{4\kappa-2}_{3\times 1}, 0_{3\times 1}\right). $
Now it is easy to see that $MDB' N_1$ can be absorbed into the first term in \eqref{CD-Error}   
and $MDB' N_2$ can be absorbed into the second term. The key is that $N_1$ has rank one and 
the second row of $N_2$ is zero. The analysis of $MA_2DB'N$ is even easier since a direct 
computation shows that $DB'$ dominates
$A_2 DB'$ componentwise.
This proves \eqref{CD-Error} and shows that 
$MACA'N+MADB'N$ has the same asymptotics as \eqref{CD-Main}.


{\it Substep 3.4, estimate of the remaining blocks in $\mathcal D$.}

The following estimates are obtained by a direct computation
\begin{equation}\nonumber
\begin{aligned}
BD+F&=
O(\mu_{1\times 4})\otimes O(\mu^{\kappa-1},\mu^{\kappa-2},0)O\left(\begin{array}{cc}
(\mu^{2\kappa})_{1\times 2}&(\mu^{3\kappa})_{1\times 2}\\
(\mu^{3\kappa})_{1\times 2}&(\mu^{4\kappa})_{1\times 2}\\
(\mu^{3\kappa-1})_{1\times 2}&(\mu^{4\kappa-1})_{1\times 2}\\
\end{array}\right)
+\Id_4\\
&
+O\left(\begin{array}{cc}
(\mu^{2\kappa})_{2\times 2}&(\mu^{\kappa})_{2\times 2}\\
(\mu^{\kappa})_{2\times 2}&(\mu^{2\kappa})_{2\times 2}\\
\end{array}\right)=
\Id_4+O\left(\mu^{\kappa}\right)_{4\times 4}.\\
BC+E&=
O(\mu_{1\times 4})\otimes O(\mu^{\kappa-1},\mu^{\kappa-2},0)O\left(\begin{array}{ccc}
\mu^{2\kappa}& \mu^{2\kappa-1}&\mu^{2\kappa}\\
\mu^{3\kappa}&\mu^{3\kappa-1}&\mu^{3\kappa}\\
\mu^{3\kappa-1}&\mu^{3\kappa-2}&\mu^{3\kappa-1}
\end{array}\right)\\
&+\left(
(\mu^\kappa)_{4\times 1},(\mu^{3\kappa-1})_{4\times 1},(\mu^{3\kappa})_{4\times 1}
\right)=
O\left(
(\mu^{\kappa})_{4\times 1}, (\mu^{4\kappa-2})_{4\times 1},(\mu^{4\kappa-1})_{4\times 1}
\right).
\end{aligned}\end{equation}
Accordingly using \eqref{RelLocDelCart} and \eqref{eq: derg} for $N$, and arguing the same way 
as on substep 3.3 we get
\begin{equation}\label{eq: dercm}(BC+E)A'N+(BD+F)B'N=\dfrac{1}{\mu}[O(\mu^{\kappa})]_{1\times 4}\otimes \dfrac{\partial G^-}{\partial(\Theta,v)_-^-}+O(\mu^\kappa).
 \end{equation}
Finally, we have  $MAD=[O(\mu^{3\kappa-1})]_{3\times 4}$.\\

{\it Substep 3.5, completing the asymptotics of $\cD$.}

Substeps 3.1--3.4 above can be summarized as follows
\begin{equation}\label{eq: derloc}
\cD=\end{equation}
$$\dfrac{1}{\mu}(\mathbf{U}+O(\mu^{3\kappa-1}); O(\mu^{\kappa})_{1\times 4})\otimes \left(\dfrac{\partial G^-}{\partial(\Theta,v)_-^-};0_{1\times 4}\right)+\left(\begin{array}{c|c}
\mathbf B&0\\
\hline
0&\Id_4
\end{array}\right)+O\left(\mu^{3\kappa-1}\right).$$
Finally, when we use the coordinates $(Q_-,v_-)$ instead of $(\Theta_-,v_-)$ as we did in Substep 2.4, we get
$\mathbf U+O(\mu^{3\kappa-1})=O(\mu^{\kappa}_{1\times 2}, 1_{1\times 2})$ and $\mathbf B+O\left(\mu^{3\kappa-1}\right)=O(1)$, and $\frac{\partial G^-}{\partial(Q,v)_-^-}= O( 1_{1\times 2}
,\mu^{\kappa}_{1\times 2})$ in terms of the coordinates $(Q_-,v_-,Q_+,v_+)$.
Hence, similarly to \eqref{EqQvpm}, we get $$\frac{\partial(Q_-,v_-,Q_+,v_+)^+}{\partial(Q_-,v_-,Q_+,v_+)^-}=\frac{1}{\mu}O(\mu^{\kappa}_{1\times 2}, 1_{1\times 2}; O(\mu^{\kappa})_{1\times 4})\otimes \left( 1_{1\times 2}
,\mu^{\kappa}_{1\times 2};0_{1\times 4}\right)+O(1).$$
This is the structure of $d\Loc$ stated in the lemma. 

It remains to obtain the asymptotics of the leading terms in Lemma \ref{LmDerLoc}. 
Below we use the Delaunay variables $(L_3,\ell_3,G_3,g_3,G_4,g_4)^\pm$ as the orbit parameters \textit{outside} the sphere $|Q_-|=2\mu^\kappa$ and add a subscript $in$ to the Delaunay variables \textit{inside} the sphere. We relate $C^0$ estimates of Lemma~\ref{Lm: landau} to the $C^1$
estimates obtained above. Namely consider the following equation which is obtained by
discarding the $o(1)$ errors in~\eqref{eq: landau}
\begin{equation}
\label{LandauLead}
Q_-^+=0,\ v_-^+=R(\al)v_-^-,\quad Q_+^+=Q_+^-,\ v_+^+=v_+^-,
\end{equation}
where $\al$ is given by \eqref{Alpha}. We have the following corollary saying that $d\Loc$ can be obtained by taking derivative directly in \eqref{LandauLead}.
\begin{Cor}
\label{Cor}
Under the assumption of Lemma \ref{LmDerLoc}, the derivative of the local map has the following form
\begin{equation}
\label{DLocMore}
d\Loc=\dfrac{1}{\mu}(\hat{\mathbf{u}}_j+O(\mu^{\kappa}))\otimes \lin_j+ \hat B_j+O(\mu^{3\kappa-1}),
\end{equation}
where $\hat{\mathbf{u}}_j, \lin_j$ and $\hat B_j$
are computed from \eqref{LandauLead} and the variables are evaluated at the $j$-th Gerver's collision point, $j=1,2$.
In particular,

\begin{equation}
\begin{aligned}
& \hat{\mathbf{u}}_j=\dfrac{\partial (L_3,\ell_3,G_3,g_3,G_4,g_4)^+}{\partial (Q_3,v_3,Q_4,v_4)^+}\dfrac{\partial (Q_3,v_3,Q_4,v_4)^+}
{\partial (Q_-,v_-,Q_+,v_+)^+} \dfrac{\partial (Q_-,v_-,Q_+,v_+)^+}{\partial \al}\left(\mu\dfrac{\partial \al}{\partial G_{in}}\right), \\
&\lin_j=\dfrac{\partial G_{in}}{\partial (Q_-,v_-,Q_+,v_+)^-}\dfrac{\partial (Q_-,v_-,Q_+,v_+)^-}{\partial (Q_3,v_3,Q_4,v_4)^-}\dfrac{\partial (Q_3,v_3,Q_4,v_4)^-}{\partial (L_3,\ell_3,G_3,g_3,G_4,g_4)^-}.
\end{aligned}\label{Equl}
\end{equation}
As $1/\chi\ll\mu\to 0$, we have that $\lin_j$ is a continuous function of $(L_3,\ell_3,G_3,g_3,G_4,g_4)^-$, and $\hat{\mathbf u}_j$ is a continuous function of both $(L_3,\ell_3,G_3,g_3,G_4,g_4)^-$ and $\al$. 
\end{Cor}

\begin{proof}
We begin by computing the rank 1 terms in the expression for $\cD.$
To get \eqref{Equl} we need to multiply the vector by
$\frac{\partial (L_3,\ell_3,G_3,g_3,G_4,g_4)^+}{\partial (Q_3,v_3,Q_4,v_4)^+}\frac{\partial (Q_3,v_3,Q_4,v_4)^+}
{\partial (Q_-,v_-,Q_+,v_+)^+}$ 
and the linear functional by $\frac{\partial (Q_-,v_-,Q_+,v_+)^-}{\partial (Q_3,v_3,Q_4,v_4)^-}\frac{\partial (Q_3,v_3,Q_4,v_4)^-}{\partial (L_3,\ell_3,G_3,g_3,G_4,g_4)^-}.$ 

For the map \eqref{LandauLead} we have 
\[\dfrac{\partial (Q_+,v_+)^+}{\partial (Q_+,v_+)^-}=\Id_4,\ \dfrac{\partial (Q_+,v_+)^+}{\partial (Q_-,v_-)^-}=\dfrac{\partial (Q_-,v_-)^+}{\partial (Q_+,v_+)^-}=0, \ \dfrac{\partial(Q_+,v_+)^+}{\partial \al}=\dfrac{\partial G_{in}}{\partial (Q_+,v_+)^-}=0\]
which agrees with the corresponding blocks in \eqref{eq: derloc} up to an $o(1)$ error as $\mu\to 0$. It remains to compare $\frac{\partial (Q_-,v_-)^+}{\partial (Q_-,v_-)^-}$.  

Now the expression for $\lin_j$ follows from \eqref{CD-Main}.

Differentiating \eqref{LandauLead} we get 
$\frac{\partial(Q_-,v_-)^+}{\partial \al}=\left(0,\frac{\partial v_-^+}{\partial\al}\right).$
Thus to get the expression of $\hat{\mathbf u}$ in \eqref{Equl}, it is enough  to show 
(cf.  \eqref{eq: local2bp}) that for the map \eqref{LandauLead} we have
\begin{equation}\label{EqCompareal}
\dfrac{\partial v_-^+}{\partial\al}\left(\dfrac{\partial \al}{\partial G_{in}}\right)=\left(\dfrac{\partial v_-^+}{\partial G^{+}}+\dfrac{\partial\arctan\frac{G^-}{\mu \mathcal{L}^{-}}}{\partial G^-}\dfrac{\partial v_-^+}{\partial g^{+}}\right),\quad G_{in}=G_-.
\end{equation}
Write $v_-^+=\bbV(G^+, \mu\cL, g^+)$ where $G^+$ and $g^+$ depend on $G^-$ as follows.
First, $G^+=G^-.$ Second, \eqref{eq: delaunay4}
gives
$$ \arctan\left(\frac{v_2^+}{v_1^+}\right)\sim g^+-\arctan\left(\frac{G^+}{\mu\cL}\right), \quad
\arctan\left(\frac{v_2^-}{v_1^-}\right)\sim g^--\arctan\left(\frac{G^-}{\mu\cL}\right), $$
$$\arctan\left(\frac{v_2^+}{v_1^+}\right)\sim \arctan\left(\frac{v_2^-}{v_1^-}\right)+\alpha
$$ 
where $\sim$ means that the difference between the LHS and the RHS is $O\left(e^{-2u}\right).$
Thus $g^+\sim g^-+\alpha$ and so
$$ \frac{\partial v_-^+}{\partial G^-}=\frac{\partial \bbV}{\partial G^+}+
\frac{\partial \bbV}{\partial g^+} \frac{\partial g^+}{\partial G^-}\sim
\frac{\partial \bbV}{\partial G^+}+
\frac{\partial \bbV}{\partial g^+} \frac{\partial \alpha}{\partial G^-}
$$
proving \eqref{EqCompareal}.


To complete the proof of the corollary
 we have to show that the formula for $\hat B$ is obtained by taking the derivatives of \eqref{LandauLead}
with respect to variables different from $G^-.$ This is done by comparing \eqref{Equl} with 
\eqref{eq: local2bp} similarly to the derivation of \eqref{Equl}. 
\end{proof}



It remains to show that the RHS of \eqref{Equl} has  the dependence on $\bx, \theta_4^+$
required by Lemma \ref{LmDerLoc}.
To this end we note that the variable $\al$ can be solved using the implicit function theorem as a function of the outgoing asymptote $\bar\theta^+_4$ in the limit $\mu\to 0$ ( see the proof of Lemma~\ref{LmLMC0}). The proof of Lemma \ref{LmDerLoc} is now complete. 
\end{proof}

The next corollary says that the small remainders in \eqref{eq: landau} is also $C^1$ small if the derivative is taken along a correct direction, i.e. the direction with small change of $G_{in}^-$. 

\begin{Cor}\label{cor: 2}
Let $\gm(s):(-\eps,\eps)\to\R^6 $ be a $C^1$ curve such that $\Gamma=\gm'(0)=O(1)$ and
$\frac{\partial G^-_{in}\circ\gamma(0)}{\partial s}=\frac{\partial G_{in}^-}{\partial\Gamma}=O(\mu)$ then when taking derivative with respect to $s$ in equations
\[\begin{cases}
&|v_3^+|^2+|v_4^+|^2=|v_3^-|^2+|v_4^-|^2+o(1),\\
&v_3^++v_4^+=v_3^-+v_4^-+o(1),\\
&Q_3^++Q_4^+=Q_3^-+Q_4^-+o(1),
\end{cases}
\]
obtained from equation \eqref{eq: landau},
the $o(1)$ terms are small in the $C^1$ sense as $\mu\to 0$.
\end{Cor}
\begin{proof}
For the motion of the mass center, it follows from Corollary \ref{Cor} that
$$\dfrac{\partial (Q_+,v_+)^+}{\partial (Q_-,v_-,Q_+,v_+)^-}=\dfrac{1}{\mu}\dfrac{\partial (Q_+,v_+)^+}{\partial\al}\otimes \lin+(0_{4\times 4},\Id_{4\times 4})+o(1).$$

We already obtained that $\frac{\partial (Q_+,v_+)^+}{\partial\al}=O(\mu^{\kappa})$
(see equation \eqref{eq: derloc}).
Due to Corollary \ref{Cor}
our assumption that $\frac{\partial G^-_{in}}{\partial s}=O(\mu)$ implies that
\begin{equation}
\label{SmallLin}
\lin \cdot\Gamma=O(\mu)
\end{equation}
which suppresses the $1/\mu$ term.
This proves the last two identities of the corollary.

To derive the first equation it is enough to show $\frac{d}{ds}(|v_-^+|^2-|v_-^-|^2)=o(1)$
since we already have the required estimate for the velocity of the mass center.
We use the fact that RHS \eqref{eq: relcm} is the same
in incoming and outgoing variables (superscripts $+$ and $-$ respectively).
In \eqref{eq: relcm}, the terms involving only $Q_+,v_+$ are handled using the result of the previous paragraph. The term $-\frac{\mu}{|Q_-|}$ vanishes when taking derivative since $|Q_-|=2\mu^\kappa$ is constant.  All the remaining terms have
$Q_-$ to the power 2 or higher. We have $\frac{\partial Q_-^-}{\partial s}=O(1)$ since $\Gamma=O(1)$.
We also have $\frac{\partial Q_-^+}{\partial s}=O(1)$ due to \eqref{SmallLin}.
Therefore after taking the $s$ derivative, any term involving $Q_-$ is of order $O(\mu^\kappa)$. This completes the proof of the energy conservation part.
\end{proof}

\subsection{Proof of the Lemma~\ref{Lm: local4}}
In this section we work out the $O(1/\mu)$ term in the local map.

\begin{proof} The proof is relies on a numerical computation.

\textbf{Before collision, $\lin=\frac{\partial G_{in}}{\partial -}$.}
According to Corollary \ref{Cor} we can differentiate the asymptotic expression of Lemma \ref{Lm: landau}.
We have
$\left(\frac{\partial G_{in}}{\partial G_4^-},\frac{\partial G_{in}}{\partial g_4^-}\right)=$
\[-(v_3^--v_4^-)\times\left(\dfrac{\partial }{\partial G_4^-},\dfrac{\partial }{\partial g_4^-}\right)Q_4-(v_3^--v_4^-)\times\left(\dfrac{\partial Q_4}{\partial \ell_4^-}\right)\cdot\left(\dfrac{\partial \ell_4^-}{\partial G_4^-},\dfrac{\partial \ell_4^-}{\partial g_4^-}\right)
+O(\mu^{\kappa}+\mu^{1-2\kappa}),\] where $O(\mu^{\kappa})$ comes from $\left(\frac{\partial }{\partial -}(v_3^--v_4^-)\right)\times (Q_3-Q_4)$ and $O(\mu^{1-2\kappa})$ comes from $\frac{\partial Q_4}{\partial L_4^-}\frac{\partial L_4^-}{\partial -}$ where $L_4$ is solved from the Hamiltonian \eqref{EqHamLoc} $H=0.$

We need to eliminate $\ell_4$ using the relation $|Q_3-Q_4|=\mu^{\kappa}$.
\[\left(\dfrac{\partial \ell_4^-}{\partial G_4^-},\dfrac{\partial \ell_4^-}{\partial g_4^-}\right)=-\left(\dfrac{\partial |Q_3-Q_4|}{\partial \ell_4^-}\right)^{-1}\left(\dfrac{\partial |Q_3-Q_4|}{\partial G_4^-},\dfrac{\partial |Q_3-Q_4|}{\partial g_4^-}\right)\]
\[=-\dfrac{(Q_3-Q_4)\cdot\left(\frac{\partial Q_4}{\partial G_4^-},\frac{\partial Q_4}{\partial g_4^-}\right)}{(Q_3-Q_4)\cdot\frac{\partial Q_4}{\partial \ell_4^-}}=-\dfrac{(v_3^--v_4^-)\cdot\left(\frac{\partial Q_4}{\partial G_4^-},\frac{\partial Q_4}{\partial g_4^-}\right)}{(v_3^--v_4^-)\cdot\frac{\partial Q_4}{\partial \ell_4^-}}+O(\mu^{1-\kappa}).\]
Here we replaced $Q_3^--Q_4^-$ by $v_3^--v_4^-$ using the fact that the two vectors form
an angle of order $O(\mu^{1-\kappa})$ by Lemma \ref{Lm: landau}(c).
Therefore
\[\left(\dfrac{\partial G_{in}}{\partial G_4^-},\dfrac{\partial G_{in}}{\partial g_4^-}\right)=-(v_3^--v_4^-)\times\left(\dfrac{\partial }{\partial G_4^-},\dfrac{\partial }{\partial g_4^-}\right)Q_4\]\[+(v_3^--v_4^-)\times\dfrac{\partial Q_4}{\partial \ell_4^-}\left(\dfrac{(v_3^--v_4^-)\cdot\left(\frac{\partial Q_4}{\partial G_4^-},\frac{\partial Q_4}{\partial g_4^-}\right)}{(v_3^--v_4^-)\cdot\frac{\partial Q_4}{\partial \ell_4^-}}\right)
+O(\mu^{\kappa}+\mu^{1-2\kappa}).\]
Similarly, we get 
\[\dfrac{\partial G_{in}}{\partial \ell_3^-}=(v_3^--v_4^-)\times\dfrac{\partial Q_3}{\partial \ell_3^-}+(v_3^--v_4^-)\times\dfrac{\partial Q_4}{\partial \ell_4^-}\left(\dfrac{(v_3^--v_4^-)\cdot\frac{\partial Q_3}{\partial \ell_3^-}}{(v_3^--v_4^-)\cdot\frac{\partial Q_4}{\partial \ell_4^-}}\right)
+O(\mu^{\kappa}+\mu^{1-2\kappa}).\]
We use {\sc Mathematica} and the data in the Appendix~\ref{subsection: numerics}
to work out $\frac{\partial G_{in}}{\partial -}.$ The results are :  for the first collision,
$\hat\lin_1=[*,-0.8,*,*,3.42,-2.54],$ and for the second collision: $\hat\lin_2=[*,-0.35,*,*,3.44,-0.47]$. We can check directly that $\hat\lin_i\cdot w_{3-i}\neq 0$ and $\hat\lin_i\cdot \tw\neq 0$ for $i=1,2$ using \eqref{eq: ul}.

\textbf{After collision, $\hat{\mathbf{u}}=\frac{\partial -}{\partial \al}$.}
In equation~(\ref{eq: landau}), we let $\mu\to 0$.
Recall \eqref{Sep-XY}.
Applying the implicit function theorem to \eqref{eq: landau} with $\mu=0$ we obtain
\begin{equation}
\begin{aligned}
&\left(\dfrac{\partial(Q_3^+,v^+_3,Q_4^+,v_4^+) }{\partial (X^+,Y^+)}+\dfrac{\partial(Q_3^+,v_3^+,Q_4^+,v_4^+) }{\partial \ell_4^+}\otimes\dfrac{\partial \ell_4^+}{\partial (X^+,Y^+)}\right)\cdot\dfrac{\partial (X^+,Y^+)}{\partial\al}\\
&=\dfrac{1}{2}\left(0,0,R\left(\frac{\pi}{2}+\al\right)(v^-_3-v^-_4),0,0,-R\left(\frac{\pi}{2}+\al\right)(v^-_3-v^-_4)\right)^T\\
&=\dfrac{1}{2}\left(0,0,R\left(\frac{\pi}{2}\right)(v^+_3-v^+_4),0,0,-R\left(\frac{\pi}{2}\right)(v^+_3-v^+_4)\right)^T.
\end{aligned}\nonumber
\end{equation}
where $R(\pi/2+\al)=\frac{dR(\al)}{d\al}$ and $\frac{\partial \ell_4^+}{\partial (X^+,Y^+)}$ is given
by \eqref{eq: l4}.
Again we use {\sc Mathematica} to work out the $\frac{\partial -}{\partial \al}.$ 
The results are: for the first collision  $\hat{\mathbf{u}}_1=[-0.49, *,*,*, -0.20, -0.64]$ and for the second collision $\hat{\mathbf{u}}_2=[-1.00, *,*,*, 0.34, -0.50]$. We can check directly that $\brlin_i\cdot \hat{\mathbf{u}}_i\neq 0$ for $i=1,2$ using \eqref{eq: ul}.

To obtain a symbolic sequence with any order of symbols $3,4$ as claimed in the main theorem, we notice that the only difference is that the outgoing relative velocity changes sign $(v_3^+-v_4^+)\to -(v_3^+-v_4^+)$. So we only need to send $\hat{\mathbf{u}}\to -\hat{\mathbf{u}}$.
\end{proof}

\subsection{Proof of the Lemma~\ref{Lm: local3}}\label{subsection: local3}
In this section, we prove Lemma \ref{Lm: local3}, which guarantees the non degeneracy condition Lemma \ref{LmNonDeg} (see the proof of Lemma \ref{LmNonDeg}). Since we have already obtained $\lin$ and $\mathbf u$ in $d\Loc$ and $\brlin,\brrlin,\brv, \bar{\brv}$ in $d\Glob$, one way to prove Lemma \ref{LmNonDeg} is to work out the matrix $B$ explicitly using Corollary  \ref{Cor} on a computer. In that case, the current section is not necessary. However, in this section, we use a different approach, which simplifies the computation and has several advantages. The first advantage is that this treatment has clear physical and geometrical meaning. Second, we use the same way to control the shape of the ellipse in Appendix \ref{SSEllipse}. Third, this method gives us a way to deal with the singular limit $d\Loc$ as $\mu\to 0$.

Recall that Lemmas ~\ref{LmDerLoc} and~\ref{LmDerGlob} give the following form for the derivatives of local map and global maps
\[d\Loc=\dfrac{1}{\mu}\mathbf u_j\otimes \lin_j+ B+O(\mu^{\kappa}),\quad d\Glob=\chi^2\brv_j\otimes\brlin_j+\chi\brrv_j\otimes\brrlin_j+O(\mu^2\chi),\]
where $j=1,2$ standing for the first or second collision and we have absorbed the $o(1)$ errors into the vectors. Moreover, in the limit $1/\chi\ll\mu\to 0$ followed by $\dt\to0$ and $\tilde\theta\to 0$, \[span\{\brv_j,\brrv_j\}\to span\{w_j,\tilde{w}\},\quad \lin_j\to \hat\lin_j,\ \brlin_j\to \hat\brlin_j,\ \brrlin_j\to \hat\brrlin_j,\quad j=1,2.\]
We first prove an abstract lemma that reduces the study of the local map of the $\mu>0$ case to $\mu=0$ case. It shows that we can find a direction in span$\{\brv,\brrv\}$, along which the directional derivative of $d\Loc$ is not singular. 

\begin{Lm} Consider $\bx\in U_j(\dt)$, j=1,2 and $|\bar\theta_4^+-\pi|<\tilde\theta$ as in Lemma \ref{LmDerLoc}. 
Suppose the vector $\tilde{\Gamma}_{\mu}\in span\{\brv_{3-j}, \brrv_{3-j}\}\subset T_{\bx} U_j(\dt)$ satisfies
$\brlin_{j}(d\Loc \tilde{\Gamma}_{\mu})=0$ and $\Vert\tilde{\Gamma}_{\mu}\Vert_\infty=1.$
Then we have 
\begin{enumerate}
\item[(a)]$\lin_j(\tilde\Gamma_{\mu})=O(\mu)$ as $\mu\to 0$,
\item[(b)] the limits $\lim_{\mu\to 0}\tilde{\Gamma}_{\mu}$ and $\lim_{\mu\to 0} d\Loc \tilde{\Gamma}_{\mu} $ exist, and $\lim_{\mu\to 0}\tilde{\Gamma}_{\mu}$ is continuous in $\bx$ and $\lim_{\mu\to 0} d\Loc \tilde{\Gamma}_{\mu} $ is continuous in $\bx$ and $\bar\theta_4^+$,
\item[(c)]$\displaystyle\hat\brlin_j(\lim_{\dt,\tilde\theta\to0}\lim_{\mu\to 0} d\Loc \tilde{\Gamma}_{\mu})=0$.
\end{enumerate}
\label{Lm: limit}
\end{Lm}
\begin{proof}
Denote
$\Gamma'_\mu=\lin_j(\brv_{3-j})\brrv_{3-j}-\lin_j(\brrv_{3-j})\brv_{3-j}\in Ker \lin_j$ and let $v_\mu$ be a vector in $\Span(\brv_{3-j}, \brrv_{3-j})$ such that
$v_\mu\to v$ as $\mu\to 0$ and $\lin_j(v_\mu)=1.$
Suppose that
\[\tilde\Gamma_\mu=a_\mu v_\mu+b_\mu \Gamma'_\mu\]
then
\begin{equation}
\label{GammaBasis}
d\Loc(\tilde\Gamma_\mu)=\dfrac{a_\mu}{\mu} \lin_j(v_\mu) \mathbf u_j+a_\mu B_j(v_\mu)+b_\mu B_j \Gamma'_\mu+o(1) .
\end{equation}
So $\brlin_j(d\Loc(\tilde\Gamma_\mu))=0$ implies that
\begin{equation}
\label{EqAmu}
a_\mu=-\mu \dfrac{b_\mu \brlin_j(B_j\Gamma'_\mu)+o(1)}{\lin_j(v_\mu) \brlin_j(\mathbf u_j)+\mu\brlin _jB_j(v_\mu)}.
\end{equation}
The denominator is not zero since $\lin_j(v_\mu)=1$ and $\brlin_j(\mathbf u_j)\neq 0$ using Lemma \ref{Lm: local4}.
Therefore $a_\mu=O(\mu).$ 
Hence $\tilde\Gamma_\mu=b_\mu \Gamma'_\mu+O(\mu)$ and $\lin_j(\tilde\Gamma_\mu)=O(\mu).$ The continuous dependence on variables in part (b) follows from part (a) of Lemma \ref{LmDerLoc} and \ref{LmDerGlob}. Now the remaining statements of
the lemma follow from equations \eqref{GammaBasis} and \eqref{EqAmu}.
\end{proof}
To compute the numerical values it is more convenient for us to work with polar coordinates.
We need the following quantities.
\begin{Def}
\begin{itemize}
\item $\psi$: polar angle, related to $u$ by $\tan\frac{\psi}{2}=\sqrt{\frac{1+e}{1-e}}\tan\frac{u}{2}$ for ellipse. We choose the positive $y$ axis as the axis $\psi=0$. 
$E$: energy; $e:$~eccentricity;
$G$: angular momentum, $g$: argument of periapsis.
\item The subscripts $3,4$ stand for $Q_3$ or $Q_4$.
The superscript $\pm$ refers to before or after collision. Recall that all quantities are evaluated on the sphere
$$|Q_3-Q_4|=\mu^\kappa.$$
\end{itemize}
\end{Def}
Recall the formula $r=\frac{G^2}{1-e\cos\psi}$ for conic sections in which the perigee lies on the axis $\psi=\pi$.
In our case we have
\begin{equation}
\label{PolarGen}
\begin{cases}&r_3^\pm=\dfrac{(G^\pm_3)^2}{1-e^\pm_3\sin(\psi_3^\pm+g^\pm_3)}+o(1),\\
&r_4^\pm=\dfrac{(G^\pm_4)^2}{1-e_4^\pm\sin(\psi_4^\pm-g^\pm_4)}+o(1).
\end{cases}
\end{equation}
$o(1)$ terms are small when $\mu\to 0$ (recall that we always assume that $\chi\gg 1/\mu$).\label{Lm: polar}



\begin{Lm}\label{Lm: polarsection}
Under the assumptions of Corollary \ref{cor: 2} we have
\[\dfrac{d r_3^+}{d s}=\dfrac{d r_4^+}{d s}+o(1),\quad \dfrac{d r_3^-}{d s}=\dfrac{d r_4^-}{d s}+o(1),
\quad\dfrac{d \psi_3^+}{d s}=\dfrac{d \psi_4^+}{d s}+o(1),\quad\dfrac{d \psi_3^-}{d s}=\dfrac{d \psi_4^-}{d s}+o(1).\]
Moreover in \eqref{PolarGen} the $o(1)$ terms are also $C^1$ small when taking the $s$ derivative.
\end{Lm}
\begin{proof}
To prove the statement about \eqref{PolarGen}, we use the Hamiltonian \eqref{eq: hamloc}. The $r_{3,4}$
obey the Hamiltonian system \eqref{eq: hamloc}.
The estimate \eqref{eq: pertout} shows the $\frac{-\mu}{|Q_3-Q_4|}$ gives small perturbation to the variational equations. The two $O(1/\chi)$ terms in \eqref{eq: hamloc} are also small. This shows that the perturbations to Kepler motion is $C^1$ small.

Next we consider the derivatives $\frac{\partial  r_{3,4}^\pm}{\partial s}$.
We consider first the case of ``$-$". From the condition
$|\vec r_3-\vec r_4|=\mu^\kappa$, for the Poincar\'{e} section we get
\[(\vec r_3-\vec r_4)\cdot \dfrac{d }{d s}(\vec r_3-\vec r_4)=0.\]
This implies $(\vec r_3-\vec r_4)\perp \frac{d }{d s}(\vec r_3-\vec r_4)$.

We also know the angular momentum for the relative motion is
$$G_{in}=(\dot{\vec r}_3-\dot{\vec r}_4)\times (\vec r_3-\vec r_4)=O(\mu), $$
which implies $\dot{\vec r}_3-\dot{\vec r}_4$ is almost parallel to $\vec r_3-\vec r_4$.
The condition $\frac{\partial G^-_{in}}{\partial s}=O(\mu)$ reads
\[\left(\dfrac{d }{d s}(\dot{\vec r}_3-\dot{\vec r}_4)\right)\times (\vec r_3-\vec r_4)+(\dot{\vec r}_3-\dot{\vec r}_4)\times \left(\dfrac{d }{d s}(\vec r_3-\vec r_4)\right)=O(\mu).\]
Since the first term is $O(\mu^\kappa)$ due to our choice of the Poincare section we 
see that
$$(\dot{\vec r}_3-\dot{\vec r}_4)\times \left(\dfrac{d }{d s}(\vec r_3-\vec r_4)\right)=o(1). $$
Since $\frac{d }{d s}(\vec r_3-\vec r_4)$ is almost perpendicular to $(\dot{\vec r}_3-\dot{\vec r}_4)$ by the analysis
presented above
we get $\frac{d }{d s}(\vec r_3-\vec r_4)=o(1)$.
Taking the radial and angular part of this vector identity and using that $r_4=r_3+o(1),$ $\psi_4=\psi_3+o(1)$ we get "$-$" part of
the lemma.

To repeat the above argument for ``+" variables, we first need to establish $\frac{\partial G^+_{in}}{\partial s}=O(\mu).$ Indeed, using equations \eqref{eq: dervar} and \eqref{eq: matrices} we get
\begin{align*}
\dfrac{\partial G^+_{in}}{\partial \psi}&=\dfrac{\partial G^+_{in}}{\partial(\mathcal{L},G_{in},g,Q_+,v_+)^-}\dfrac{\partial(\mathcal{L},G_{in},g,Q_+,v_+)^-}{\partial \psi}\\
&=O(\mu^{3\kappa},1, \mu^{3\kappa},\mu^{3\kappa}_{1\times 2},\mu^{3\kappa}_{1\times 2})\cdot O(1,\mu,1,1_{1\times 2},1_{1\times 2})=O(\mu).
\end{align*}

It remains to show $\left(\frac{d }{d s}(\dot{\vec r}_3-\dot{\vec r}_4)\right)=O(1)$ in the $``+"$ case. Since we know it is true in the ``-" case, the ``+" case follows, because the directional derivative of the local map $d\Loc\Gamma$ is bounded due to our choice of
$\Gamma$.
\end{proof}
We are now ready to describe the computation of Lemma \ref{Lm: local3}.
The reader may notice that the computations in the proofs of Lemmas \ref{Lm: local3} and \ref{LmGer} are quite similar.
Note however that Lemma \ref{Lm: local3} describes the {\it subleading} term for the derivative of the local map.
By contrast the {\it leading} term can not be understood in terms of the Gerver map since it comes from the possibility of
varying the closest distance between $Q_3$ and $Q_4$ and this distance is assumed to be zero in Gerver's model.

We will use the following set of equations which follows
from \eqref{LandauLead}.
\begin{equation}\label{eq: polarcollision1}
E_3^++E_4^+=E_3^-+E_4^-, \end{equation}
\begin{equation}\label{eq: polarcollision2}
G_3^++G_4^+=G_3^-+G_4^-, \end{equation}
\begin{equation}\label{eq: polarcollision3}
\dfrac{e_3^+}{G_3^+}\cos(\psi_3^++g_3^+)+\dfrac{e_4^+}{G_4^+}\cos(\psi_4^--g_4^{-})=\dfrac{e_3^-}{G_3^-}\cos(\psi_3^-+g_3^-)+\dfrac{e_4^-}{G_4^-}\cos(\psi_4^--g_4^{-}),
\end{equation}
\begin{equation}\label{eq: polarcollision4}
\dfrac{(G^+_3)^2}{1-e^+_3\sin(\psi_3^++g^+_3)}=\dfrac{(G_3^-)^2}{1-e_3^-\sin(\psi_3^-+g^-_3)},
\end{equation}
\begin{equation}\label{eq: polarcollision5}
\psi^+_3=\psi^-_3,
\end{equation}
\begin{equation}\label{eq: polarcollision6}
\dfrac{(G^+_3)^2}{1-e^+_3\sin(\psi_3^++g^+_3)}=\dfrac{(G^+_4)^2}{1-e_4^+\sin(\psi_4^+-g^{+}_4)},
\end{equation}
\begin{equation}\label{eq: polarcollision7}
\dfrac{(G_3^-)^2}{1-e_3^-\sin(\psi_3^-+g^-_3)}=\dfrac{(G_4^-)^2}{1-e_4^-\sin(\psi_4^--g_4^{-})}, \end{equation}
\begin{equation}\label{eq: polarcollision8}
\psi_4^-=\psi_3^-,
\end{equation}
\begin{equation}\label{eq: polarcollision9}
\psi_4^+=\psi_3^+.
\end{equation}
In the above equations we have dropped $o(1)$ terms for brevity.
We would like to emphasize that the above approximations hold not only in $C^0$ sense but also
in $C^1$ sense when we take the derivatives along the directions satisfying the conditions of Corollary \ref{cor: 2}.
\eqref{eq: polarcollision1} is the approximate conservation of the energy, \eqref{eq: polarcollision2} is the approximate
conservation of the angular momentum and \eqref{eq: polarcollision3} follows from the approximate momentum conservation
(see the derivation
of \eqref{eq: polarcollision3B} in Appendix \ref{SSEllipse}).
The possibility of differentiating these equations
is justified in Corollary \ref{cor: 2}. The remaining equations reflect the fact
that $Q_3^\pm$ and $Q_4^\pm$ are all close to each other. The possibility of differentiating these equations
is justified by Lemma \ref{Lm: polarsection}.

We set the total energy to be zero. So we get $E_4^\pm=-E_3^\pm$. This eliminates $E_4^\pm$. Then we also eliminate
$\psi_{4}^\pm$ by setting them to be equal $\psi_3^\pm$.

\begin{proof}[Proof of the Lemma~\ref{Lm: local3}]
Lemma \ref{Lm: limit} and Corollary \ref{Cor} show that the assumption of Lemma \ref{Lm: local3} implies that the direction $\Gamma$ along which we take the directional derivative satisfies $\frac{\partial G_{in}}{\partial \Gamma}=O(\mu)$. So we can directly take derivatives in equations \eqref{eq: polarcollision1}-\eqref{eq: polarcollision7}.
Recall that we need to compute $dE_3^+(d\Loc \Gamma)$ where
$\Gamma\in Ker\mathbf l_j\cap$span$\{w_{3-j},\tilde w\}$.  \eqref{eq: ul} tells us that in
in Delaunay coordinates we have
\begin{equation}
\label{wtw}
\tw=(0,1,0,0,0,0), \quad  w=(0,0,0,0,1,a) \text{ where }a=\frac{-L_4^{-}}{(L^{-}_4)^2+ (G_4^{-})^2} .
\end{equation}
The formula $\tan\frac{\psi}{2}=\sqrt{\frac{1+e}{1-e}}\tan\frac{u}{2}$ which relates $\psi$ to $\ell$ through $u$
shows that \eqref{wtw} also holds if we use $(L_3, \psi_3, G_3, g_3, G_4, g_4)$ as coordinates.
Hence $\Gamma$ has the form $(0,1,0,0,c,ca)$.
To find  the constant $c$ we use \eqref{eq: polarcollision7}.

Note that the expression $dE_3^+(d\Loc \Gamma)$
does not involve $d\psi_3^+.$ Therefore we can eliminate $\psi_3^+$ from consideration by setting
$\psi_3^+=\psi_3^-=\psi$ (see \eqref{eq: polarcollision5}).
Let $\bL$ denote the projection
of our map to $(L_3, G_3, g_3, G_4, g_4)$ variables. Thus we need to find
$dE_3^+(d\bL \Gamma).$
To this end write the remaining equations (\eqref{eq: polarcollision2}, \eqref{eq: polarcollision3}, \eqref{eq: polarcollision4}, and
\eqref{eq: polarcollision6})
formally as $\mathbf{F}(Z^+,Z^-)=0$, where in $Z^+=(E_3^+,G_3^+,g_3^+,G_4^+,g_4^+)$ and $Z^-=(E_3^-,\psi,G_3^-,g_3^-,G_4^-,g_4^-)$.

We have
\[\dfrac{\partial\mathbf F}{\partial Z^+}
d\bL \Gamma+\dfrac{\partial\mathbf F}{\partial Z^-}\Gamma=0.\]

However, $\frac{\partial\mathbf F}{\partial Z^+}$ is not invertible since $\mathbf F$ involves
only four equations of $\mathbf F$ while $Z^+$ has 5 variables.
To resolve this problem we use that by definition of $\Gamma$ we have
  $\bar{\mathbf l}\cdot \frac{\partial Z^+}{\partial\psi}=0$, where $\bar{\mathbf l}=\left(\frac{ G_4^+/L_4^{+}}{(L_4^{+})^2+(G_4^{+})^2},0,0,0,\frac{-1}{(L_4^{+})^2+(G_4^{+})^2},\frac{1}{L_4^{+}}\right)$ by \eqref{eq: ul}.
Thus we get
\[\left[\begin{array}{c}\bar{\mathbf l}\\
\dfrac{\partial\mathbf F}{\partial Z^+}
\end{array}\right] d\bL\Gamma
=-\left[\begin{array}{c}0\\
\dfrac{\partial\mathbf F}{\partial Z^-}\Gamma
\end{array}\right]\]
and so
\[d\bL\Gamma
=-\left[\begin{array}{c}\bar{\mathbf l}\\
\dfrac{\partial\mathbf F}{\partial Z^+}
\end{array}\right]^{-1}\left[\begin{array}{c}0\\\dfrac{\partial\mathbf F}{\partial Z^-}\Gamma
\end{array}\right].\]
We use computer to complete the computation.
We only need the entry $\frac{\partial E_3^+ }{\partial \psi}$ to prove Lemma~\ref{Lm: local3}.
It turns out this number is $1.855$ for the first collision and $-1.608$ for the second collision. Neither is zero
as needed.
\end{proof}

\appendix

\section{Delaunay coordinates}\label{section: appendix}
\subsection{Elliptic motion}\label{subsection: ellip}
The material of this section could be found in \cite{Al}.
Consider the two-body problem with Hamiltonian
\[H(P,Q)=\dfrac{|P|^2}{2m}-\dfrac{k}{|Q|},\quad (P,Q)\in \R^4.\]
This system is integrable in the Liouville-Arnold sense when $H<0$.
So we can introduce the action-angle variables $(L,\ell, G,g)$
in which the Hamiltonian can be written as
\[H(L,\ell, G,g)=-\dfrac{mk^2}{2L^2},\quad (L,\ell, G,g)\in T^*\T^2.\]
The Hamiltonian equations are
\[\dot{L}=\dot{G}=\dot{g}=0,\quad \dot{\ell}=\dfrac{mk^2}{L^3}.\]
We introduce the following notation
$E$-energy, $M$-angular momentum, $e$-eccentricity, $a$-semimajor axis, $b$-semiminor axis.
Then we have the following relations which explain the physical and geometrical meaning of the Delaunay coordinates.
\[a=\dfrac{L^2}{mk}, \ b=\dfrac{LG}{mk},\ E=-\dfrac{k}{2a},\ M=G,\ e=\sqrt{1-\left(\dfrac{G}{L}\right)^2}.\]
Moreover, $g$ is the argument of periapsis and $\ell$ is called the mean anomaly, and $\ell$ can be related to the polar angle $\psi$ through
the equations
\[    \tan\dfrac \psi 2 = \sqrt{\dfrac{1+e}{1-e}}\cdot\tan\dfrac u 2,\quad u-e\sin u=\ell.\]
We also have the Kepler's law $\frac{a^3}{T^2}=\frac{1}{(2\pi)^2}$ which relates the semimajor axis 
$a$ and the period $T$ of the ellipse.

Denoting particle's position by $(q_1, q_2)$ and its momentum $(p_1,p_2)$ we have the following formulas
in case $g=0.$
\[\begin{cases}
& q_1=a(\cos u-e),\\
& q_2=a\sqrt{1-e^2}\sin u,
\end{cases}
\quad\begin{cases}
& p_1=-\sqrt{mk}a^{-1/2}\dfrac{\sin u}{1-e \cos u},\\
& p_2=\sqrt{mk}a^{-1/2}\dfrac{\sqrt{1-e^2}\cos u}{1-e \cos u},
\end{cases}\]
where $u$ and $l$ are related by $u-e\sin u=\ell$.

Expressing $e$ and $a$ in terms of Delaunay coordinates we obtain the following
\begin{equation}
\label{DelEll}
\begin{aligned}
q_1=\dfrac{L^2}{mk}\left(\cos u-\sqrt{1-\dfrac{G^2}{L^2}}\right), \quad &
q_2=\dfrac{LG}{mk} \sin u.
\\
p_1=-\dfrac{mk}{L}\dfrac{\sin u}{1-\sqrt{1-\dfrac{G^2}{L^2}} \cos u}, \quad &
p_2=\dfrac{mk}{L^2}\dfrac{G\cos u}{1-\sqrt{1-\dfrac{G^2}{L^2}}\cos u} .
\end{aligned}
\end{equation}

Here $g$ does not enter because the argument of perihelion is chosen to be zero. In general case, we need to rotate the $(q_1,q_2)$ and
$(p_1,p_2)$ using the matrix
$\left[\begin{array}{cc}
\cos g& -\sin g\\
\sin g& \cos g
\end{array} \right].
$

Notice that the equation \eqref{DelEll} describes an ellipse with one focus at the origin and the other focus on the negative $x$-axis. We want to be consistent with \cite{G2}, i.e. we want $g=\pi/2$ to correspond to the
``vertical" ellipse with one focus at the origin and the other focus on the positive $y$-axis (see Appendix~\ref{subsection: numerics}).
Therefore we rotate the picture clockwise. So we use the Delaunay coordinates which are related to the Cartesian ones through the equation
\begin{equation}
\begin{aligned}
q_1=&\dfrac{1}{mk}\left(L^2\left(\cos u-\sqrt{1-\dfrac{G^2}{L^2}}\right)\cos g+LG\sin u\sin g\right) , \\
q_2=&\dfrac{1}{mk} \left(-L^2 \left(\cos u-\sqrt{1-\dfrac{G^2}{L^2}}\right)\sin g+LG\sin u\cos g\right).\\
\end{aligned}\label{eq: Q_3}
\end{equation}
\subsection{Hyperbolic motion}\label{subsection: hyp}
The above formulas can also be used to describe hyperbolic motion, where we need
to replace ``$\sin\to\sinh,$ $\cos\to \cosh$"(c.f.\cite{Al, F}). Namely, we have for $g=0$
\begin{equation}
\begin{aligned}
q_1=\dfrac{L^2}{mk}
\left(\cosh u-\sqrt{1+\frac{G^2}{L^2}}\right), \quad
& q_2=\dfrac{LG}{mk} \sinh u, \\
p_1=-\dfrac{mk}{L}\dfrac{\sinh u}{1-\sqrt{1+\frac{G^2}{L^2}} \cosh u}, \quad
& p_2=-\dfrac{mk}{L^2}\dfrac{G\cosh u}{1-\sqrt{1+\frac{G^2}{L^2}}\cosh u}.
\end{aligned}\label{eq: delaunay4}
\end{equation}
where $u$ and $l$ are related by
\begin{equation}u-e\sinh u=\ell, \text{ where } e=\sqrt{1+\left(\dfrac{G}{L}\right)^2}.
\label{eq: hypul}
\end{equation}
This hyperbola is symmetric w.r.t. the $x$-axis, opens to the right and the particle moves counterclockwise on it when $u$ increases ($\ell$ decreases) in the case when the angular momentum 
$G=p\times q<0$. The angle $g$ is defined to be the angle measured from the positive $x$-axis to the symmetric axis. There are two such angles that differ by $\pi$ depending on the orientation of the symmetric axis. This $\pi$ difference disappears in the symplectic form and the Hamiltonian equation, so it does not matter which angle to choose.

When the particle moves to the right of $x=-\frac{\chi}{2}$ line we have a hyperbola opening to the left and the particle moves
counter-clockwise. To get the picture studied in \cite{G1}, we rotate \eqref{eq: delaunay4} by $\pi+g.$
In this case, we choose $g$ to be the angle measured from the positive $x$-axis to the symmetric axis pointing to the perigee. 
Thus we have
\begin{equation}
\begin{aligned}
q_1=&\dfrac{1}{mk}\left(-\cos g L^2(\cosh u-e)+\sin g LG\sinh u\right), \\
q_2=&\dfrac{1}{mk}\left(-\sin g L^2(\cosh u-e)-\cos g LG\sinh u \right),\\
P=&\dfrac{mk}{1-e\cosh u}\left(\dfrac{1}{L}\sinh u\cos g-\dfrac{G}{L^2}\sin g\cosh u,\right.\\
& \left.\dfrac{1}{L}\sinh u\sin g+\dfrac{G}{L^2}\cos g\cosh u\right).
\end{aligned}\label{eq: Q4}
\end{equation}
If  the incoming asymptote is horizontal, (see the arrows in Figure 1 for ``incoming" and ``outgoing"), then the particle comes from the left, and
as $u$ tends to $-\infty$,
the $y$-coordinate is bounded and $x$-coordinate is negative.
In this case we have $\tan g=\frac{G}{L}$, $g\in(-\pi/2,0).$ We use $u<0$ to refer to this piece of orbit.

If the outgoing asymptote is horizontal, then the particle escapes to the left, and
as $u$ tends to $+\infty$, the $y$-coordinate is bounded and $x$-coordinate is negative.
In this case we have $\tan g=-\frac{G}{L}, g\in(0,\pi/2)$. We use $u>0$ to refer to this piece of orbit.

The above two cases can be unified as $\tan g=-\sign(u)\frac{G}{L}$ with $G<0,L>0$. 

When the particle $Q_4$ is moving to the left of the section $\{x=-\chi/2\}$, we treat the motion as hyperbolic motion focused at $Q_1$.
We move the origin to $Q_1$. The hyperbola opens to the right. The particle $Q_4$ moves on the hyperbola counterclockwise with negative angular momentum $G$, we then rotate by angle $g$ and $g$ is the angle measured from the positive $x$-axis to the symmetric axis pointing to the opening of the hyperbola. The orbit has the following parametrization
\begin{equation}
\begin{aligned}
q_1=&\dfrac{1}{mk}\left(\cos g L^2(\cosh u-e)-\sin g LG\sinh u\right), \\
q_2=&\dfrac{1}{mk}(  \sin g L^2(\cosh u-e)+\cos g LG\sinh u),\\
P=&\dfrac{mk}{1-e\cosh u}\left(-\dfrac{1}{L}\sinh u\cos g+\dfrac{G}{L^2}\sin g\cosh u, \right.\\
&\left.-\dfrac{1}{L}\sinh u\sin g-\dfrac{G}{L^2}\cos g\cosh u\right).
\end{aligned}\label{eq: Q4l}
\end{equation}

In the left case the orbits we consider have $G$ is close to zero, i.e. the system is close to the double collision. 
In this case, the hyperbolic Delaunay coordinates are singular when $\ell$ is close to zero. Indeed
when we set $e=1$ in \eqref{eq: hypul}, we find $\ell=u^3+h.o.t.$ Hence $u$ as a function $\ell$ in a neighborhood of $0$ is only $C^0$ but not $C^1$.  One can verify that for $G=0$ and $\ell\neq 0$ the hyperbolic Delaunay coordinates still give a symplectic transformation, so we only have singular behavior when $G$ and $\ell$ are both close to zero. To control this singular behavior, we need the following estimates.

\begin{Lm} \label{LmSmallu}In the hyperbolic Delaunay coordinates, as $G\to 0$, $u\to 0$ and $L$ being close to $1$, we have the following estimates of the first order derivatives 
\[\left|\dfrac{\partial u}{\partial G}\right|\leq 2,\quad \left|\dfrac{\partial u}{\partial L}\right|\leq 2|G|\]
and the second order derivatives
\[\left|\dfrac{\partial Q}{\partial u}\frac{\partial^2 u}{\partial G^2}\right|\leq 4,\quad \left|\dfrac{\partial Q}{\partial u}\frac{\partial^2 u}{\partial L^2}\right|\leq 4G^2,\quad \left|\dfrac{\partial Q}{\partial u}\frac{\partial^2 u}{\partial G\partial L}\right|\leq 4|G|.\]
\end{Lm}
\begin{proof}
For the first order derivatives, it follows from \eqref{eq: hypul} that 
$$
\frac{\partial u}{\partial G}-e\cosh u\dfrac{\partial u}{\partial G}=\sinh u\dfrac{\partial e}{\partial G}.
$$
We have $\frac{\partial e}{\partial G}=\frac{G}{eL^2}$ and $\frac{\partial e}{\partial L}=\frac{-G^2}{eL^3}$. 
Hence we get for small $G$ and $u$
$$
\left|\frac{\partial u}{\partial G}\right|=\left|\frac{\sinh u\frac{\partial e}{\partial G}}{1-e\cosh u}\right|\sim \left|\frac{2uG}{G^2+u^2}\right|\leq 1.
$$
To get $\frac{\partial u}{\partial L}$, we replace $G$ by $L$ in the above expression we get $$\left|\frac{\partial u}{\partial L}\right|=\left|\frac{\sinh u\frac{\partial e}{\partial L}}{1-e\cosh u}\right|\sim \left|\frac{2uG^2}{G^2+u^2}\right|\leq G. $$
Next, we work on second order derivatives.  We have $$
 \dfrac{\partial^2 u}{\partial G^2}-2\dfrac{\partial e}{\partial G}\dfrac{\partial u}{\partial G}\cosh u-e\sinh u\left(\frac{\partial u}{\partial G}\right)^2-\dfrac{\partial^2 e}{\partial G^2}\sinh u-e\cosh u\frac{\partial^2u}{\partial G^2}=0
 $$ 
 which gives 
 $$
 \dfrac{\partial^2 u}{\partial G^2}=
 \frac{1}{1-e\cosh u}\left( 2\dfrac{\partial e}{\partial G}\dfrac{\partial u}{\partial G}\cosh u
 +e\sinh u\left(\frac{\partial u}{\partial G}\right)^2
 +\dfrac{\partial^2 e}{\partial G^2}\sinh u\right)\sim \dfrac{G+u}{G^2+u^2}
 $$
for small $u$ and $G$ by substituting $\sinh u\sim u,\ \cosh u\sim 1,\ \frac{\partial e}{\partial G}\sim G$ and $\frac{\partial u}{\partial G}\sim 1$. 
On the other hand, we have $$\frac{\partial Q}{\partial u}=\dfrac{\partial }{\partial u}(L^2\cosh u,LG\sinh u)=(L^2\sinh u,LG\cos u)\sim(u,G),$$
where we choose $g=0$ in $Q$ since a rotation by $g$ does not change the Euclidean norm. 
When we consider $\frac{\partial Q}{\partial u}\frac{\partial^2 u}{\partial G^2}$, we get
 $$\left|\dfrac{\partial Q}{\partial u}\frac{\partial^2 u}{\partial G^2}\right|\leq \dfrac{(|u|+|G|)^2}{u^2+G^2}\leq 2$$

To get $\frac{\partial Q}{\partial u}\frac{\partial^2 u}{\partial L^2}$, we need to replace in the expression of $\frac{\partial^2 u}{\partial G^2}$ everywhere $G$ by $L$, which gives us the estimate $\frac{\partial^2 u}{\partial L^2}\sim \frac{G+u}{G^2+u^2}G^2$.  To get $\frac{\partial Q}{\partial u}\frac{\partial^2 u}{\partial L\partial G}$, we have 
$$
 \dfrac{\partial^2 u}{\partial L\partial G}=\frac{1}{1-e\cosh u}\left( \left(\dfrac{\partial e}{\partial L}\dfrac{\partial u}{\partial G}+\dfrac{\partial e}{\partial G}\dfrac{\partial u}{\partial L}\right)\cosh u-e\sinh u\frac{\partial u}{\partial G}\frac{\partial u}{\partial L}-\dfrac{\partial^2 e}{\partial L\partial G}\sinh u\right)$$
 which is estimated as $G\frac{G+u}{G^2+u^2}.$ This completes the proof. 
\end{proof}

\subsection{Large $\ell$ asymptotics: auxiliary results}
In the remaining part of Appendix~\ref{section: appendix}
we obtain estimates onthe first and second order derivatives of $Q$ w.r.t. the hyperbolic Delaunay variables
$(L,\ell,G,g)$ which are needed in our proof.
The next lemma allows us to simplify the computations.
Since the hyperbolic motion approaches a linear motion, this lemma shows that,
we can replace $u$ by $\ln (\mp \ell/e)$ when taking first and second order derivatives.

\begin{Lm} \label{Lm: simplify}
Let $u$ be the function of $\ell, G$ and $L$ given by \eqref{eq: hypul}.
Then we can approximate $u$
by $\ln (\mp\ell/e)$ in the following sense.

\[u\mp\ln\dfrac{\mp\ell}{e}=O(\ln|\ell|/\ell),\quad \dfrac{\partial u}{\partial\ell}=\pm 1/\ell+O(1/\ell^2),\]
\[\left(\dfrac{\partial }{\partial L},\dfrac{\partial }{\partial G}\right)\left(u\pm\ln e\right)=O(1/|\ell|),\quad \left(\dfrac{\partial }{\partial L},\dfrac{\partial }{\partial G}\right)^2\left(u\pm\ln e\right)=O(1/|\ell|),\]

Here the first sign is taken if $u>0$ and the second sign is taken then $u<0.$
The estimates above are uniform as long as $|G|\leq K,$ $1/K\leq L \leq K,$ $\ell>\ell_0$
and the implied constants in $O(\cdot)$ depend only on $K$ and $\ell_0.$
\end{Lm}

\begin{proof}

We see from formula~\eqref{eq: hypul} that $\sinh u\simeq\cosh u= -\frac{\ell}{e}+O(\ln|\ell|)$ when $u>0$ and $\sinh u\simeq-\cosh u\simeq -\frac{\ell}{e}+O(\ln|\ell|)$ when $u<0$ and $|u|$ large enough. This proves $C^0$ estimate.

Now we consider the first order derivatives. We assume that $u>0$ to fix the notation.
Differentiating \eqref{eq: hypul} with respect to $\ell$ we get
\[\dfrac{\partial u}{\partial \ell}-e\cosh u\dfrac{\partial u}{\partial \ell}=1,\quad
\dfrac{\partial u}{\partial \ell}=1/\ell+O(1/\ell^2).\]
Next, we differentiate \eqref{eq: hypul} with respect to $L$ to obtain
\[\dfrac{\partial u}{\partial L}-\dfrac{\partial e}{\partial L}\sinh u-e\cosh u\dfrac{\partial  u}{\partial L}=0.\]
Therefore,
\[\dfrac{\partial  u}{\partial L}=\dfrac{\sinh u}{1-e\cosh u}\dfrac{\partial e}{\partial L}=-\dfrac{1}{e} \dfrac{\partial e}{\partial L}+O(e^{-|u|})=-\dfrac{\partial}{\partial L}\ln(e)+O(1/|\ell|).\]
The same argument holds for $\frac{\partial }{\partial G}$. This proves $C^1$ part of the Lemma.

Now we consider second order derivatives. We take $\frac{\partial^2 }{\partial L^2}$ as example. Combining
\begin{equation}
\dfrac{\partial^2 u}{\partial L^2}-\dfrac{\partial^2 e}{\partial L^2}\sinh u-2\cosh u\dfrac{\partial e}{\partial L}\dfrac{\partial u}{\partial L}-e\cosh u\dfrac{\partial^2 u}{\partial L^2}-e\sinh u\left(\dfrac{\partial u}{\partial L}\right)^2=0.
\nonumber\end{equation}
with $C^1$ estimate proven above we get
\[\dfrac{\partial^2 u}{\partial L^2}=-\dfrac{1}{e} \dfrac{\partial^2 e}{\partial L^2}
-\dfrac{2\partial e}{e\partial L}\dfrac{\partial u}{\partial L}+
\left(\dfrac{\partial u}{\partial L}\right)^2+O\left(\dfrac{1}{\ell}\right)\]
\[=
-\dfrac{1}{e} \dfrac{\partial^2 e}{\partial L^2}+\left(\dfrac{1}{e} \dfrac{\partial e}{\partial L}\right)^2+O\left(\dfrac{1}{\ell}\right)
=\dfrac{\partial^2}{\partial L^2}\ln e+O\left(\dfrac{1}{\ell}\right). \]

This concludes the $C^2$ part of the lemma.
\end{proof}

In the estimate of the derivatives presented in the next two subsections we shall often use the following facts.
Let $f=\ln e.$ Then
\begin{equation}
\label{Ecc1D}
f_G=\dfrac{G}{L^2+G^2},\quad f_L=-\dfrac{G^2}{L(L^2+G^2)},
\end{equation}
\begin{equation}
\label{Ecc2D}
(f)_{GG}=\dfrac{L^2-G^2}{(L^2+G^2)^2},
\quad
f_{LG}=-\dfrac{2GL}{(L^2+G^2)^{2}} .
\end{equation}

\subsection{First order derivatives}\label{subsubsection: 1stderivative}
In the following computations, we assume for simplicity that $m=k=1.$ To get the general case we only need to divide
positions by $mk.$
\begin{Lm}
Under the same conditions as in Lemma \ref{Lm: simplify}
we have the following result for the first order derivatives\begin{itemize}
\item[(a)]$\qquad
\left|\dfrac{\partial Q}{\partial \ell}\right|=O(1),\quad \left|\dfrac{\partial Q}{\partial (L,G,g)}\right|=O(\ell),
\quad\dfrac{\partial Q}{\partial g}\cdot Q=0,$\[\dfrac{\partial Q}{\partial G}\cdot Q=O_{C^2(L, G, g)}(\ell). \]
\item[(b)] 
If in addition we have $\left|g+\sign(u) \arctan\frac{G}{L}\right|\leq C/|\ell|$  then
we have the following bounds for \eqref{eq: Q4}
{\small\[\frac{\partial Q}{\partial G}=-\frac{L^2\sinh u}{\sqrt{L^2+G^2}}\left(0,1\right)+O(1),\
\frac{\partial Q}{\partial L}=\sinh u\left(-2 \sqrt{L^2+G^2},\frac{GL}{\sqrt{L^2+G^2}}\right)+O(1).\]}
\item[(c)] If in addition to the conditions of Lemma \ref{Lm: simplify} we have $G,g=O(1/\chi)$ and $\ell=O(\chi),$
then
we have the following bounds for \eqref{eq: Q4l}
\[\dfrac{\partial Q}{\partial G}=\sinh u(0, L)+O(1),\quad \dfrac{\partial Q}{\partial L}=\sinh u(2L,0)+O(1).\]
\end{itemize}
\label{LM: 1stder}
\end{Lm}
\begin{Rk}The assumptions of the lemma and the next lemma hold in our situation due to Lemma~\ref{Lm: tilt}.
\end{Rk}
\begin{proof}
We write the position variables in \eqref{eq: Q4} as $$q=(L^2\cosh u,LG\sinh u)-L^2e(1,0)=\cosh u L(L,G\sign(u))+O(1)$$ and $Q$ is obtained by rotating $q$ by angle $\pi+g$ in case $(b)$ and by angle $g$ in case (c). 

Using Lemma~\ref{Lm: simplify}, 
we obtain
\begin{equation}
\begin{aligned}
\frac{\partial q}{\partial G}&=-\sign(u)\cdot f_G(L^2\sinh u,LG\cosh u)+L\sinh u(0,1)+O(1)\\
&=\frac{G}{L^2+G^2}(-L^2,-\sign(u)LG)\cosh u+\sign(u)L \cosh u(0,1)+O(1)\\
&=\frac{L^2\cosh u}{L^2+G^2}(-G,\sign (u) L)+O(1),\\
\frac{\partial q}{\partial L}&=-\sign(u)\cdot f_L(L^2\sinh u,LG\cosh u)+(2L\cosh u,G\sinh u)+O(1)\\
&=\frac{-G^2}{L(L^2+G^2)}(-L^2,-\sign(u)LG)\cosh u+(2L,G\sign(u))\cosh u+O(1)\\
&=(L,0)\cosh u+\frac{L^2+2G^2}{L^2+G^2}(L,\sign (u)G)\cosh u+O(1).
\end{aligned}\nonumber
\end{equation}
Now the estimates on $\frac{\partial Q}{\partial G}$ and $\frac{\partial Q}{\partial L}$ follow since,
by \eqref{eq: hypul},
$\cosh u$ and $\sinh u$ are $O(\ell).$ The estimates on $\frac{\partial Q}{\partial g}$ 
follow since $Q$ is obtained from $q$ by a rotatation.
Also 
$$ \frac{\partial q}{\partial l}=-\sign(u) \frac{\partial u}{\partial l} 
\left(L^2 sinh u, \;\; LG \cosh u\right) $$
so the estimate of $\frac{\partial Q}{\partial g}$ follows from Lemma \ref{Lm: simplify}.

To prove the last estimate of part (a) we observe that
\begin{equation}
\begin{aligned}
&Q\cdot \frac{\partial Q}{\partial G}=q\cdot \frac{\partial q}{\partial G}= \cosh u L(L,G\sign u)\cdot 
\frac{L^2}{L^2+G^2}(-G,\sign (u) L)+O(\ell)=O(\ell).\end{aligned}\nonumber
\end{equation}
Next, we work on (b). First consider $g=-\sign(u)\arctan\frac{G}{L}$, $G<0$. 
Then $\frac{\partial Q}{\partial G}$ is a rotation of $\frac{\partial q}{\partial G}$ by $\pi+g$. 
We see from above that $\frac{\partial q}{\partial G}$ is a vector  with polar angle $\sign(u)\arctan\frac{L}{-G}=\sign(u)(\frac{\pi}{2}-\arctan\frac{-G}{L}) $. So after rotating by angle $g+\pi$, finally we get that $\frac{\partial Q}{\partial G}$ has polar angle $\pi+\sign (u)\frac{\pi}{2}=-\sign (u)\frac{\pi}{2}$, we get 
$$\frac{\partial Q}{\partial G}=- \sinh u\frac{L^2}{\sqrt{L^2+G^2}}(0,1)+O(1).$$
When $g$ is in a $1/|\ell|$ neighborhood of $-\sign(u)\arctan\frac{G}{L}$, we get the same estimate by absorbing the error into $O(1)$. 
By the same argument, we get that 
$$\frac{\partial Q}{\partial L}= \left(-2\sqrt{L^2+G^2}\cosh u,\;\;\sinh u  \frac{LG}{\sqrt{L^2+G^2}}\right)+O(1).$$
Part (c) follows directly from the formulas for 
$\frac{\partial q}{\partial G},\frac{\partial q}{\partial L},$ since both $g$ and $\arctan\frac{G}{L}$ are $O(1/\chi).$
\end{proof}

\subsection{Second order derivatives}\label{subsubsection: 2ndderivative}
The following bounds of the second order derivatives are used in estimations of the variational equation.
\begin{Lm}
We have the following information for the second order derivatives of $Q_4$ w.r.t. the Delaunay variables.
\begin{itemize}
\item[(a)] Under the conditions of Lemma \ref{LM: 1stder}(a) we have
\[\dfrac{\partial^2 Q}{\partial g^2}=-Q,\quad \dfrac{\partial^2 Q}{\partial g\partial G}\perp \dfrac{\partial Q}{\partial G},
\quad \left(\dfrac{\partial}{\partial G},\dfrac{\partial}{\partial g}\right)\left(\dfrac{\partial |Q|^2}{\partial g}\right)=(0,0),\]\[\dfrac{\partial^2Q}{\partial G^2}=O(\ell),\quad \dfrac{\partial^2 Q}{\partial L^2}=O(\ell),\quad
\dfrac{\partial^2 Q}{\partial G \partial L}=O(\ell).\]
\item[(b)] Under the conditions of Lemma \ref{LM: 1stder}(b) we have
we have
\begin{equation}
\begin{aligned}
\dfrac{\partial^2 Q}{\partial G^2}&=\dfrac{L^2}{(L^2+G^2)^{3/2}}(L\cosh u, 2G \sinh u)+O(1),\\
\dfrac{\partial^2 Q}{\partial g\partial G}&=\left(\dfrac{L^2\sinh u}{\sqrt{L^2+G^2}},0\right)+O(1),\\
\dfrac{\partial^2 Q}{\partial g\partial L}&=\left(-\dfrac{GL\sinh u}{\sqrt{L^2+G^2}},-2 \sqrt{L^2+G^2}  \cosh u\right)+O(1),\\
\dfrac{\partial^2 Q}{\partial G\partial L}&=\dfrac{-L}{(L^2+G^2)^{3/2}}\left(LG \cosh u, (L^2+3G^2) \sinh u\right)+O(1).
\end{aligned}\nonumber\end{equation}
\item[(c)] Under the conditions of Lemma \ref{LM: 1stder}(c) we have
\begin{equation}
\begin{aligned}
&\dfrac{\partial^2 Q}{\partial G^2}=-\cosh u(1,0)+O(1), \quad
\dfrac{\partial^2 Q}{\partial g\partial G}=-L\sinh u(1,0)+O(1), \\
& \dfrac{\partial^2 Q}{\partial g\partial L}=L\sinh u(0,2)+O(1), \quad
\dfrac{\partial^2 Q}{\partial G\partial L}=\cosh u(0,1)+O(1).\\
\end{aligned}\nonumber\end{equation}
\end{itemize}
\label{Lm: 2ndderivative}
\end{Lm}
\begin{proof}
The estimates of $\frac{\partial^2Q}{\partial G^2},$ $\frac{\partial^2Q}{\partial L^2},$ and
$\frac{\partial^2Q}{\partial G \partial L}$   
follows from similar estimates on the derivatives of $q.$ The estimates on the second derivatives of
$q$ follow by straightforward differentiation of \eqref{eq: delaunay4}
using Lemma~\ref{LM: 1stder}. The other estimates of part (a) follow since $Q$ depends on $g$ via a rotation.

Next we prove parts (b) and (c). Again we first work on $q$ then rotate by $g+\pi$ for (b) and by $g$ for (c),
\begin{equation*}\begin{aligned}
\frac{\partial^2 q}{\partial G^2}&=\left(\left(\frac{L^2}{L^2+G^2}\right)_G\cosh u+\frac{L^2 \sinh u u_G}{L^2+G^2}\right)(-G,\sign (u) L)+\frac{L^2\cosh u}{L^2+G^2}(-1,0)+O(1)\\
&=\cosh u\left(\frac{-3L^2G}{(L^2+G^2)^2}\right)(-G,\sign (u) L)+\frac{L^2\cosh u}{L^2+G^2}(-1,0)+O(1)
\end{aligned}\end{equation*}
\begin{equation*}\begin{aligned}
\frac{\partial^2 q}{\partial L\partial G}&=\left(\left(\frac{L^2}{L^2+G^2}\right)_L\cosh u+\frac{L^2 \sinh u u_L}{L^2+G^2}\right)(-G,\sign (u) L)+\frac{L^2\cosh u}{L^2+G^2}(0,\sign(u))+O(1)\\
&=\cosh u\left(\frac{3LG^2}{(L^2+G^2)^2}\right)(-G,\sign (u) L)+\frac{L^2\cosh u}{L^2+G^2}(0,\sign(u))+O(1)
\end{aligned}\end{equation*}
After rotating by angle $\pi+g$ with $g=-\sign (u)\cdot\arctan \frac{G}{L}$, we get 
\begin{equation*}\begin{aligned}
\frac{\partial^2 Q}{\partial G^2}&=\sinh u\frac{3L^2G}{(L^2+G^2)^{3/2}}(0,1)+\frac{L^2\cosh u}{(L^2+G^2)^{3/2}}(L,-\sign(u)G)+O(1)\\
&=\frac{L^2}{(L^2+G^2)^{3/2}}(L\cosh u,2G\sinh u)+O(1)
\end{aligned}\end{equation*}
\begin{equation*}\begin{aligned}
\frac{\partial^2 Q}{\partial L\partial G}
&=\sinh u\frac{-3LG^2}{(L^2+G^2)^{3/2}}(0,1)+\frac{L^2\sinh u}{(L^2+G^2)^{3/2}}(-\sign(u)G,-L)+O(1)\\
&=\frac{-L}{(L^2+G^2)^{3/2}}(LG\cosh u,(L^2+3G^2)\sinh u)+O(1).
\end{aligned}\end{equation*}
This gives the estimates on $\frac{\partial^2 Q}{\partial G^2}$
and $\frac{\partial^2 Q}{\partial L\partial G}$ in part (b). The estimates of part (c) are similar.
The estimates of $\frac{\partial^2 Q}{\partial L\partial g}$ and $\frac{\partial^2 Q}{\partial G\partial g}$
follow easily from parts (b) and (c) of Lemma~\ref{LmDerGlob}.
\end{proof}

\section{Gerver's mechanism}\label{section: gerver}
\subsection{Gerver's result in \cite{G2}}\label{subsection: gerver}
We summarize the result of \cite{G2} in the following table. Recall that the Gerver scenario deals with the limiting case
$\chi\to\infty, \mu\to 0$. Accordingly $Q_1$ disappears at infinity and there is no interaction between $Q_3$ and $Q_4.$
Hence both particles perform Kepler motions.
The shape of each Kepler orbit is characterized by energy, angular momentum and the argument of periapsis.
In Gerver's scenario, the incoming and outgoing asymptotes of the hyperbola are always horizontal and the semimajor of the ellipse is always vertical. So we only need to describe on the energy and angular momentum.
\begin{center}
\begin{tabular}{|c|c|c|c|c|}
\hline
  & 1st collision  &@$(-\eps_0\eps_1,\eps_0+\eps_1)$  & 2nd collision & $@(\eps_0^2,0)$\\
\hline
   & $Q_3$ & $Q_4$ & $Q_3$ & $Q_4$ \\
\hline
energy  & $-\frac{1}{2}$ & $\frac{1}{2}$ & $-\frac{1}{2}\rightarrow -\frac{\eps_1^2}{2\eps_0^2}$ & $\frac{1}{2}\rightarrow \frac{\eps^2_1}{2\eps_0^2}$\\
\hline
{\small angular\,momentum}& $\eps_1\rightarrow -\eps_0$ & $p_1\rightarrow -p_2$ & $-\eps_0$ & $\sqrt{2}\eps_0$ \\
\hline
eccentricity & $\eps_0\rightarrow \eps_1$ &  & $\eps_1\rightarrow \eps_0$& \\
\hline
semimajor & $1$ & $-1$ & $1\rightarrow \left(\frac{\eps_0}{\eps_1}\right)^2$ &  $1\rightarrow -\frac{\eps^2_1}{\eps_0^2}$\\
\hline
semiminor & $\eps_1\rightarrow \eps_0$ & $p_1\rightarrow p_2$ & $\eps_0\rightarrow \frac{\eps_0^2}{\eps_1}$ & $\sqrt{2}\eps_0 \rightarrow \sqrt{2}\eps_1$\\
\hline
\end{tabular}
\end{center}

Here \[p_{1,2}=\dfrac{-Y\pm \sqrt{Y^2+4(X+R)}}{2},\quad R=\sqrt{X^2+Y^2}.\]
and $(X,Y)$ stands for the point where collision occurs (the parenthesis after $@$ in the table). We will call the two points the Gerver's collision points.

In the above table $\eps_0$ is a free parameter and $\eps_1=\sqrt{1-\eps_0^2}.$

At the collision points, the velocities of the particles are the following.

For the first collision,
\[v_3^-=\left(\dfrac{-\eps_1^2}{\eps_0\eps_1+1}, \dfrac{-\eps_0}{\eps_0\eps_1+1}\right),\quad v_4^-=\left(1-\dfrac{Y}{Rp_1}, \dfrac{1}{Rp_1}\right).\]
\[v_3^+=\left(\dfrac{\eps_0^2}{\eps_0\eps_1+1}, \dfrac{\eps_1}{\eps_0\eps_1+1}\right),\quad v_4^+=\left(-1+\dfrac{Y}{Rp_2}, -\dfrac{1}{Rp_2}\right).\]
For the second collision,
\[v_3^-=\left(\dfrac{-\eps_1}{\eps_0}, \dfrac{-1}{\eps_0}\right),\ v_4^-=\left(1, \dfrac{\sqrt{2}}{\eps_0}\right),\quad v_3^+=\left(1, \dfrac{-1}{\eps_0}\right),\ v_4^+=\left(\dfrac{-\eps_1}{\eps_0}, \dfrac{\sqrt{2}}{\eps_0}\right).\]

\subsection{Numerical information for a particularly chosen $\eps_0=1/2$}\label{subsection: numerics}
For the first collision $e_3: \frac{1}{2}\to \frac{\sqrt{3}}{2}$.\\
We want to figure out the Delaunay coordinates $(L,u,G,g)$ for both $Q_3$ and $Q_4$. (Here we replace $\ell$ by $u$ for convenience.)
The first collision point is \[(X,Y)=(-\eps_0\eps_1,\eps_0+\eps_1)=\left(-\dfrac{\sqrt{3}}{4},\dfrac{1+\sqrt{3}}{2}\right).\]
Before collision
\[(L,u, G,g)_3^-=\left(1,-\dfrac{5\pi}{6}, \dfrac{\sqrt{3}}{2}, \pi/2\right),\quad (L,u, G,g)_4^-=(1,1.40034,-p_1, -\arctan p_1),\]
\begin{equation}\nonumber\begin{aligned}
&v_3^-=\left(\dfrac{-3}{\sqrt{3}+4},\dfrac{-2}{\sqrt{3}+4}\right)\simeq-(0.523,0.349),\\
&v_4^-=\left(1-\dfrac{2(1+\sqrt{3})}{(4+\sqrt{3})p_1},\dfrac{4}{(4+\sqrt{3})p_1}\right)\simeq (-0.805,1.322),\end{aligned}\end{equation}
where \[p_1=\dfrac{-Y+\sqrt{Y^2+4(X+R)}}{2}=\dfrac{-(\eps_0+\eps_1)+\sqrt{5+2\eps_0\eps_1}}{2}=0.52798125.\]
After collision
\[(L,u, G,g)_3^+=\left(1,\dfrac{2\pi}{3}, -\dfrac{1}{2}, \pi/2\right),\quad (L,u, G,g)_4^+=(1,0.515747,p_2, -\arctan p_2),\]
\begin{equation}\nonumber\begin{aligned}
&v_3^+=\left(\dfrac{1}{\sqrt{3}+4},\dfrac{2\sqrt{3}}{\sqrt{3}+4}\right)\simeq (0.174,0.604),\\
&v_4^+=\left(-1+\dfrac{2(1+\sqrt{3})}{(4+\sqrt{3})p_2},-\dfrac{4}{(4+\sqrt{3})p_2}\right)\simeq (-1.503,0.368)\end{aligned}\end{equation}
where\[p_2=\dfrac{-Y-\sqrt{Y^2+4(X+R)}}{2}=\dfrac{-(\eps_0+\eps_1)-\sqrt{5+2\eps_0\eps_1}}{2}=-1.894006654.\]
For the second collision $e_3: \frac{\sqrt{3}}{2}\to \frac{1}{2}$.\\
The collision point is $(X,Y)=(\eps_0^2,0)=\left(\dfrac{1}{4},0\right)$.\\
Before collision
\[(L,u, G,g)_3^-=\left(1,-\frac{\pi}{6}, -\frac{1}{2}, \pi/2\right),\ (L,u, G,g)_4^-=\left(1,0.20273,-\sqrt{2}/2, -\arctan \frac{\sqrt{2}}{2}\right),\]
\[v_3^-=\left(-\sqrt{3},-2\right),\quad v_4^-=\left(1,2\sqrt{2}\right).\]
After collision
\[(L,u, G,g)_3^+=\left(\frac{1}{\sqrt{3}},\frac{\pi}{3}, -\frac{1}{2}, -\frac{\pi}{2}\right),\ (L,u, G,g)_4^+=\left(\frac{1}{\sqrt{3}},-0.45815,-\frac{\sqrt{2}}{2}, \arctan \frac{\sqrt{6}}{2}\right),\]
\[v_3^+=\left(1,-2\right),\quad v_4^+=\left(-\sqrt{3},2\sqrt{2}\right).\]

\subsection{Control the shape of the ellipse}
\label{SSEllipse}
As it was mentioned before Lemma \ref{LmGer} was stated by Gerver in \cite{G2}. There is a detailed proof of part $(a)$ of our Lemma \ref{LmGer} in \cite{G2}. However since no details of the proof of part $(b)$
were given in \cite{G2} we go other main steps here for the reader's convenience even though computations are quite
straightforward.

\begin{proof}[Proof of Lemma~\ref{LmGer}]
Recall that Gerver's map depends on a free parameter $e_4$ (or equivalently $G_4$). In the computations below however it is
more convenient to use the polar angle $\psi$ of the intersection point as the free parameter. It is easy to see that as
$G_4$ changes from large negative to large positive value the point of intersection covers the whole orbit of $Q_3$ so it
can be used as the free parameter. Our goal is to show that by changing the angles $\psi_1$ and $\psi_2$ of the first and
second collision we can prescribe the values of $\brre_3$ and $\brrg_3$ arbitrarily. Due to the Implicit Function Theorem
it suffices to show that
$$\det\left[\begin{array}{cc}\frac{\partial \brre_3}{\partial \psi_1}&\frac{\partial \brrg_3}{\partial \psi_1}\\
\frac{\partial \brre_3}{\partial \psi_2}&\frac{\partial \brrg_3}{\partial \psi_2} \end{array}\right]\neq 0. $$
To this end we use the following set of equations
\begin{equation}\label{eq: polarcollision2B}
G_3^++G_4^+=G_3^-+G_4^-, \end{equation}
\begin{equation}\label{eq: polarcollision3B}
\dfrac{e_3^+}{G_3^+}\cos(\psi+g_3^+)+\dfrac{e_4^+}{G_4^+}\cos(\psi-g_4^{-})
=\dfrac{e_3^-}{G_3^-}\cos(\psi+g_3^-)+\dfrac{e_4^-}{G_4^-}\cos(\psi-g_4^{-}),
\end{equation}
\begin{equation}\label{eq: polarcollision4B}
\dfrac{(G^+_3)^2}{1-e^+_3\sin(\psi+g^+_3)}=\dfrac{(G_3^-)^2}{1-e_3^-\sin(\psi+g^-_3)},
\end{equation}
\begin{equation}\label{eq: polarcollision6B}
\dfrac{(G^+_3)^2}{1-e^+_3\sin(\psi+g^+_3)}=\dfrac{(G^+_4)^2}{1-e_4^+\sin(\psi-g^{+}_4)},
\end{equation}
\begin{equation}
\label{OutB}
g_4^+=\arctan\dfrac{G_4^+}{L_4^+}.
\end{equation}

Here $e_3, e_4$ and $L_4$ are functions
of the other variables according to the formulas of Appendix \ref{section: appendix}.

\eqref{eq: polarcollision2B}--\eqref{OutB} are obtained as follows.
\eqref{eq: polarcollision2B} is the angular momentum conservation,
\eqref{eq: polarcollision4B}
means that  the position of $Q_3$ does not change during the collision, \eqref{eq: polarcollision6B}
means that $Q_3$ and $Q_4$ are at the same point immediately after the collision and \eqref{OutB} says
that after the collision the outgoing asymptote of $Q_4$ is horizontal.

It remains to derive \eqref{eq: polarcollision3B}.
Represent the position vector as $\vec{r}= r \hat{e}_r$. Then the velocity is $\dot{\vec{r}}=\dot{r}\hat{e}_r+ r\dot{\psi} \hat{e}_\psi.$ The momentum conservation gives
$$(\dot{\vec{r}}_3)^-+(\dot{\vec{r}}_4)^-=(\dot{\vec{r}}_3)^++(\dot{\vec{r}}_4)^+.$$
Taking the angular component of the velocity we get
\begin{equation}
\label{MC}
r_3^-\dot{\psi}_3^-+r_4^-\dot{\psi}^-_4=r^+_3\dot{\psi}^+_3+r_4^+\dot{\psi}_4^+.
\end{equation}
In our notation the polar representation of the ellipse takes form
$ r=\frac{G^2}{1-e\sin(\psi+g)} .$ Differentiating this equation we obtain the following relation
for the radial component of the Kepler motion
\[\dot{r}=\frac{G^2 }{(1-e\sin(\psi+g))^2} e\cos (\psi+g) \dot{\psi}=\dfrac{r^2}{G^2} e\cos(\psi+g) \dfrac{G}{r^2}=
\dfrac{e}{G}\cos(\psi+g).\]

Plugging this into \eqref{MC} we obtain
\eqref{eq: polarcollision3B}.

We can write \eqref{eq: polarcollision2B}--\eqref{OutB} in the form
$$ \FF(Z^-, \tZ,  Z^+)=0 $$
where
$Z^-=(E_3^-, G_3^-, g_3^-, \psi)$, $Z^+=(E_3^+, G_3^+, g_3^+, G_4^+, g_4^+),$ and
$\tZ=(G_4^-, g_4^-)$ are considered as functions $Z^-.$

By the Implicit Function Theorem we have
$$ \dfrac{\partial Z^+}{\partial Z^-}
=-\left(\dfrac{\partial \FF}{\partial Z^+}\right)^{-1} \left(\dfrac{\partial \FF}{\partial Z^-}+
\dfrac{\partial \FF}{\partial \tZ} \dfrac{\partial \tZ}{\partial Z^-}
\right). $$
Thus to complete the computation we need to know
$\dfrac{\partial \tZ}{\partial Z^-}.$
In order to compute this expression we use the equations
\begin{equation}
\label{InB}
g_4^-=-\arctan\dfrac{G_4^-}{L_4^-}
\end{equation}
which means that the incoming asymptote of $Q_4$ is horizontal and
\begin{equation}\label{eq: polarcollision7B}
\dfrac{(G_3^-)^2}{1-e_3^-\sin(\psi+g^-_3)}=\dfrac{(G_4^-)^2}{1-e_4^-\sin(\psi-g_4^{-})}, \end{equation}
which means that $Q_3$ and $Q_4$ are at the same place immediately before the collision.
Writing these equations as
$\II(Z^-, \tZ)=0$ we get by the Implicit Function Theorem
$$\dfrac{\partial \tZ}{\partial Z^-}
=-\left(\dfrac{\partial \II}{\partial \tZ}\right)^{-1} \dfrac{\partial \II}{\partial Z^-} $$
so that the required derivative equals to
\begin{equation}
\label{DerGerFormal}
 \dfrac{\partial Z^+}{\partial Z^-}
=-\left(\dfrac{\partial \FF}{\partial Z^+}\right)^{-1} \left(\dfrac{\partial \FF}{\partial Z^-}-
\dfrac{\partial \FF}{\partial \tZ} \left(\dfrac{\partial \II}{\partial \tZ}\right)^{-1} \dfrac{\partial \II}{\partial Z^-}
\right).
\end{equation}

Combining \eqref{DerGerFormal}
with the formula
$$de_3=-\dfrac{2G_3E_3dG_3+G_3^2dE_3}{\sqrt{1-2G_3^2E_3}} $$
which follows from the relation $e_3=\sqrt{1-2G_3^2E_3}$ we obtain
the two entries 
$$\dfrac{\partial \brre_3}{\partial \psi_2}=-0.158494\text{ and }\dfrac{\partial \brrg_3}{\partial \psi_2}= 0.369599.$$
The meanings of these two entries are the changes of the eccentricity and argument of periapsis after the second collision if we vary the phase of the \emph{second} collision.

We need more work to figure out the two entries $\frac{\partial \brre_3}{\partial \psi_1}$ and $\frac{\partial \brrg_3}{\partial \psi_1}$, which are the changes of the eccentricity and argument of periapsis after the second collision if we vary the phase of the \emph{first} collision. We describe the computation of the first entry, the second one is similar.
We use the relation
\[\dfrac{\partial \brre_3}{\partial \psi_1}=\dfrac{\partial \brre_3}{\partial \bar E_3^+}
\dfrac{\partial \bar E_3^+}{\partial \psi_1}+
 \dfrac{\partial \brre_3}{\partial \bar G_3^+}\dfrac{\partial \bar G_3^+}{\partial \psi_1}+\dfrac{\partial \brre_3}{\partial \bar g_3^+}\dfrac{\partial \bar g_3^+}{\partial \psi_1}.\]
Now $\left(\frac{\partial \bar E_3^+}{\partial \psi_1}, \frac{\partial \bar G_3^+}{\partial \psi_1},
\frac{\partial \bar g_3^+}{\partial \psi_1}\right)$
is computed using
\eqref{DerGerFormal} and the data for the first collision.
Noticing that the quantities $E_3,G_3,g_3$ after the first collision are the same as those before the second collision, we replace
$\left(\frac{\partial \brre_3}{\partial \bar E_3^+}, \frac{\partial \brre_3}{\partial \bar G_3^+},
\frac{\partial \brre_3}{\partial \bar g_3^+}\right)$ by $\left(\frac{\partial \brre_3}{\partial \bar{\bar E}_3^-}, \frac{\partial \brre_3}{\partial \bar{\bar G}_3^-},
\frac{\partial \brre_3}{\partial \bar{\bar g}_3^-}\right)$ and compute it using \eqref{DerGerFormal} and the data for the second collision.
It turns out that the resulting matrix is
\[\left[\begin{array}{cc}\dfrac{\partial \brre_3}{\partial \psi_1}&\dfrac{\partial \brrg_3}{\partial \psi_1}\\
\dfrac{\partial \brre_3}{\partial \psi_2}&\dfrac{\partial \brrg_3}{\partial \psi_2} \end{array}\right]=\left[\begin{array}{cc}0.620725&2.9253\\
-0.158494& 0\end{array}\right],\]
which is obviously nondegenerate.
\end{proof}

\section*{Acknowledgement}
The authors would like to thank Prof. John Mather many illuminating discussions, and Vadim Kaloshin for introducing us to the problem.  We would also like to thank J. Fejoz for pointing to us that the hyperbolic Delaunay coordinates are singular near collision, which is handled in Lemma \ref{LmSmallu}.
This research was supported by the NSF grant DMS 1101635.

\end{document}